\documentclass[11pt,reqno]{amsart}
\usepackage[utf8]{inputenc}

\usepackage{amssymb,latexsym, mathtools,tikz}
\usepackage{enumerate}
\usepackage{graphicx}
\usepackage[mathscr]{euscript}
\usepackage{float}
\usepackage{placeins}
\usepackage{mdframed}
\usepackage{amssymb}
\usepackage{esint}
\usepackage{cool}
\usepackage[all,cmtip]{xy}
\usepackage{mathtools}
\usepackage{amstext} 
\usepackage{array}   
\usepackage[shortlabels]{enumitem}
\usepackage{ytableau}
\usepackage{tikz}
\usetikzlibrary{arrows,decorations.markings, cd}
\tikzstyle arrowstyle=[scale=1]
\tikzstyle directed=[postaction={decorate,
decoration={markings,mark=at position .65 with {\arrow[arrowstyle]{stealth}}}}]
\usepackage{mathdots}
\usepackage{hyperref}
\usepackage{quiver}
\usepackage{rotating}
\usepackage{calligra,mathrsfs}

\newcolumntype{L}{>{$}l<{$}} 
\newcolumntype{C}{>{$}c<{$}}
\newtheorem{theorem}{Theorem}[section]
\newtheorem{lemma}[theorem]{Lemma}

\newtheorem{cor}[theorem]{Corollary}
\newtheorem{prop}[theorem]{Proposition}

\theoremstyle{definition}
\newtheorem{definition}[theorem]{Definition}
\newtheorem{example}[theorem]{Example}

\newtheorem{obs}[theorem]{Observation}
\newtheorem{notation}[theorem]{Notation}

\newtheorem{construction}[theorem]{Construction}

\newtheorem{chunk}[theorem]{}
\theoremstyle{remark}
\newtheorem{remark}[theorem]{Remark}

\newtheorem{the context}[theorem]{The Context}
\newtheorem{question}[theorem]{Question}
\numberwithin{equation}{theorem}
\numberwithin{equation}{section}

\newcommand{\defi}[1]{{\bf\upshape\sffamily #1}}

\newcommand{\pd}{\operatorname{pd}}

\newcommand{\id}{\operatorname{id}}

\newcommand{\rank}{\operatorname{rank}}

\newcommand{\ann}{\operatorname{Ann}}

\newcommand{\coker}{\operatorname{Coker}}

\newcommand{\im}{\operatorname{Im}}

\newcommand{\proj}{\operatorname{-proj}}
\newcommand{\End}{\operatorname{End}}
\newcommand{\cone}{\operatorname{Cone}}
\newcommand{\Ker}{\operatorname{Ker}}

\newcommand{\ideal}[1]{\mathfrak{#1}}
\newcommand{\m}{\ideal{m}}

\newcommand{\fm}{\ideal{m}}

\newcommand{\unfold}{\operatorname{Unfold}}
\newcommand{\fold}{\operatorname{Fold}}
\renewcommand{\phi}{\varphi}

\newcommand{\bbz}{\mathbb{Z}}
\newcommand{\bbn}{\mathbb{N}}

\newcommand{\xra}{\xrightarrow}

\renewcommand{\Mod}{\operatorname{Mod}}
\newcommand{\har}{\operatorname{HArrow}}

\newcommand{\dm}{\operatorname{DM}}
\renewcommand{\geq}{\geqslant}
\renewcommand{\leq}{\leqslant}
\renewcommand{\ker}{\Ker}
\renewcommand{\hom}{\Hom}

\newcommand{\pflag}{\operatorname{PFlag}}
\newcommand{\flag}{\operatorname{Flag}}
\newcommand{\obj}{\operatorname{Obj}}

\newcommand{\Hom}{\operatorname{Hom}}

\newcommand{\colim}{\operatorname{colim}}

\newcommand{\ch}{\operatorname{Ch}}

\newcommand{\maps}[5]{\xymatrix{#1 \ar[r]^-{#3} & #2 \\
#4 \ar@{|->}[r] & #5 \\}}

\newcommand*{\sheafhom}{\mathscr{H}\text{\kern -3pt {\calligra\large om}}\,}
\newcommand*{\sheafext}{\mathscr{E}\text{\kern -3pt {\calligra\large xt}}\,}

\def\kk{\mathbf{k}}

\def\w{\wedge}

\def\im{\operatorname{im}}

\newcommand{\codim}{\operatorname{codim}}

\title{Flagged Perturbations and Anchored Resolutions}
\author{Keller VandeBogert}

\begin{document}

\begin{abstract}
    In this paper, we take advantage of a reinterpretation of differential modules admitting a flag structure as a special class of perturbations of complexes. We are thus able to leverage the machinery of homological perturbation theory to prove strong statements on the homological theory of differential modules admitting additional auxiliary gradings and having infinite homological dimension. One of the main takeaways of our results is that the category of differential modules is much more similar than expected to the category of chain complexes, and from the K-theoretic perspective such objects are largely indistinguishable. This intuition is made precise through the construction of so-called anchored resolutions, which are a distinguished class of projective flag resolutions that possess remarkably well-behaved uniqueness properties in the (flag-preserving) homotopy category. We apply this theory to prove an analogue of the Total Rank Conjecture for differential modules admitting a $\bbz / 2 \bbz$-grading in a large number of cases.
\end{abstract}

\maketitle

\tableofcontents

\section{Introduction}

A \defi{differential $R$-module} $D$ is an $R$-module $D$ equipped with a squarezero endomorphism $d^D$. Differential modules were introduced under the name ``modules with differentiation" in the classic treatise of Cartan--Eilenberg \cite{cartan2016homological} as a pedagogical tool to introduce singly-graded spectral sequences associated to filtrations. More recently, differential modules have made their appearance in a variety of places in the literature. The work of Avramov--Buchweitz--Iyengar established a differential module analogue of the ``New Intersection Theorem" \cite{avramov2007class}, and also connected the problem of understanding lower bounds on the ranks of certain types of differential modules to longstanding conjectures of Halperin \cite{halperin1985rational} and Carlsson \cite{carlsson1986free} on the rational cohomology of spaces admitting a free action by a large torus.

Differential modules also make their appearance when considering notions of Koszulness/Koszul duality over rings with nonstandard gradings \cite{hawwa2012koszul}. This was one of the motivations of Brown--Erman for developing a notion of minimal resolutions of differential modules \cite{brown2021minimal}, where the authors realized that the natural replacement of a Tate resolution \cite{brown2021tate} in the setting of toric varieties is instead a differential module obtained by adding ``higher components" to the differentials of some chain complex. From a different perspective, the category of \defi{linear factorizations} of $0$ (see for instance \cite{brown2017adams}) is equivalently the category of $\bbz / 2 \bbz$-graded differential modules, which can be viewed as the category of $2$-periodic complexes:
\[\begin{tikzcd}
	{D_1} && {D_0}
	\arrow["{d_1^D}", curve={height=-18pt}, from=1-1, to=1-3]
	\arrow["{d_0^D}", curve={height=-18pt}, from=1-3, to=1-1]
\end{tikzcd}\]
A long open conjecture of Buchweitz--Greuel-Schreyer (see \cite{buchweitz1987cohen}, also \cite{erman2021matrix}) on maximal Cohen--Macaulay modules over hypersurfaces can in fact be reformulated as a question about the minimal rank of certain types of $\bbz / 2 \bbz$-graded differential modules having finite length homology. Indeed, the connection between matrix factorizations and $\bbz / 2 \bbz$-graded differential modules is further supported by work of Oblomkov--Rozansky, where certain functors on categories of matrix factorization cannot be represented by complexes, but are instead represented by $\bbz / 2 \bbz$-graded \emph{flags} (see \cite[Lemma 5.3.6]{oblomkov2020soergel}).

Differential modules with a flag structure are objects that lie somewhere between complexes---whose differentials respect a strict homological grading---and arbitrary differential modules. A flag structure on $D$ is simply an exhaustive filtration
$$0 = D^{-1} \subset D^0 \subset D^1 \subset \cdots$$
such that $d^D (D^i) \subset D^{i-1}$ for all $i \geq 0$. A nonnegatively-graded complex $C$ may be viewed as a differential module with a flag structure by taking the direct sum of all the modules in each homological degree of $C$, equipped with squarezero endomorphism given by the direct sum of all of the differentials. This association is functorial and yields a \defi{folding} functor
$$\fold : \ch(R) \to \dm (R),$$
where a morphism of differential modules is any map of the underlying $R$-modules that commutes with the differentials. One of the interesting pathologies of the functor $\fold$ is that chain complexes that are not isomorphic may become isomorphic after viewing them as differential modules:

\begin{example}\label{ex:nonFullExample}
    Let $R$ be any commutative ring. Consider the two complexes $R \xra{0} R \xra{1} R$ and $R \xra{1} R \xra{0} R$. As objects of $\ch (R)$ these complexes are obviously not isomorphic, but when viewed in the category of differential modules there \emph{is} an isomorphism:
\[\begin{tikzcd}
	R && R && R \\
	\\
	R && R && R
	\arrow["1", from=1-3, to=1-5]
	\arrow["0", from=1-1, to=1-3]
	\arrow["0", from=3-3, to=3-5]
	\arrow["1", from=3-1, to=3-3]
	\arrow[Rightarrow, no head, from=1-3, to=3-1]
	\arrow[Rightarrow, no head, from=3-3, to=1-5]
	\arrow[Rightarrow, no head, from=3-5, to=1-1]
\end{tikzcd}\]
     Since $R$ is assumed to be any commutative ring, this example shows that the embedding $\ch (R) \to \dm (R)$ is \emph{never} full (even after descending to homotopy categories).
\end{example}

Of course, the issue with Example \ref{ex:nonFullExample} is that maps in the category $\dm (R)$ no longer need to respect homological degree in a coherent way. It is however not clear which types of objects pick up no ``new" maps under the embedding $\ch (R) \to \dm (R)$. Indeed, a fundamental obstacle for many of the results of \cite{banks2022differential,banks2023differential} is the problem of understanding and controlling the non-fullness of this embedding.

More generally, there is a fundamental dichotomy stemming from the fact that there are two possible ways to define a morphism of flags. The first definition is to just view a flag as a differential module and define morphisms to come from the category $\dm (R)$, but the second definition is to assume the morphism preserves the defining filtrations of the flags. This latter definition is much more restrictive in view of Example \ref{ex:nonFullExample}, but such morphisms are also much more well-behaved. In general, we would like to know when there is no loss of generality when restricting to flag-preserving morphisms; one of the main themes of this paper is that this question can be answered using \defi{homological perturbation theory}.

\subsection{Flagged Perturbations of Complexes}

Let $D$ be a projective flag. After some choice of splittings for the defining filtration, the underlying $R$-module of $D$ decomposes as a direct sum of its associated graded pieces $D  = \bigoplus_{i \geq 0} D_i$ and the differential $d^D$ decomposes as a sum of maps $d^D = \delta_0^D + \delta_1^D + \cdots$, so that $D$ may be represented diagramatically as so:
\[\begin{tikzcd}
	{D:} & \cdots & {D_4} & {D_3} & {D_2} & {D_1} & {D_0}
	\arrow["{\delta_0^D}", from=1-2, to=1-3]
	\arrow["{\delta_0^D}", from=1-3, to=1-4]
	\arrow["{\delta_0^D}", from=1-4, to=1-5]
	\arrow["{\delta_0^D}", from=1-5, to=1-6]
	\arrow["{\delta_0^D}", from=1-6, to=1-7]
	\arrow["{\delta_1^D}"{description}, curve={height=18pt}, from=1-3, to=1-5]
	\arrow["{\delta_1^D}"{description}, curve={height=18pt}, from=1-4, to=1-6]
	\arrow["{\delta_1^D}"{description}, curve={height=18pt}, from=1-5, to=1-7]
	\arrow["{\delta_2^D}"{description}, curve={height=-30pt}, from=1-3, to=1-6]
	\arrow["{\delta_2^D}"{description}, curve={height=-30pt}, from=1-4, to=1-7]
\end{tikzcd}\]

In other words, the sum $\delta_1^D + \delta_2^D + \cdots$ is a \defi{perturbation} of the underlying $R$-module of $D$ equipped with the differential $\delta_0^D$. The perturbation $\delta_1^D + \delta_2^D + \cdots$ has the special property that it drops the homological degree by at least $2$, and we call such a perturbation a \defi{flagged perturbation}. The underlying complex $(D , \delta_0^D)$ is called the \defi{anchor} of $D$, and if the anchor of $D$ is a resolution\footnote{That is, the homology of the complex $(D , \delta_0^D)$ is concentrated in degree $0$.} then $D$ is called an \defi{anchored resolution}.

This reinterpretation of projective flags in terms of perturbations is useful since it allows us study projective flag resolutions in the context of homological perturbation theory, which seeks to transfer properties such as homotopy equivalence between perturbed objects. The foundational result in homological perturbation theory is the Perturbation Lemma (see Lemma \ref{lem:thePertLemma1}), which shows that (small) perturbations can be transferred along deformation retracts (see Definition \ref{def:DeformationRetract}). 

One of our first main results uses this interpretation to show that there are no ``new" morphisms introduced when considering the image of a projective resolution under the inclusion $\ch (R) \to \dm (R)$, at least up to homotopy. Moreover, any morphism of anchored resolutions must in fact be a flagged perturbation of some morphism of the underlying anchors.

\begin{theorem}\label{thm:IntroFlagPreserving}
    Let $D$ and $D'$ be $R$-projective flags anchored on resolutions. Assume that $ \phi : D \to D'$ is any morphism of projective flags that does not necessarily preserve the flag structure. Then $\phi$ is flag-preserving up to homotopy; that is, there exists a flag-preserving morphism $\psi : D \to D'$ and a homotopy $h$ such that
    $$\phi - \psi = d^{D'} h + h d^D.$$
\end{theorem}

Notice that aside from asking that the homology is finitely generated, there are no other finiteness assumptions on the projective flags $D$ and $D'$, and thus extends \cite{banks2023differential} to the case of infinite flag structures. The proof of Theorem \ref{thm:IntroFlagPreserving} uses completely different techniques and takes advantage of both the perturbation lemma and a generalized notion of Cartan--Eilenberg resolutions for differential modules. More precisely, given any differential module $D$ there exists a canonical construction of a projective flag resolution $\widetilde{D}$ due originally to Stai \cite{stai2018triangulated}, which we call a \defi{Cartan--Eilenberg resolution}.\footnote{The reason for this terminology is that in Subsection \ref{subsec:CEResolutions} we extend Stai's construction to respect any kind of ambient $\bbz / d \bbz$-grading, and when $d = 0$ we recover the classical notion of a Cartan--Eilenberg resolution of complexes. The similarity of Stai's construction with Cartan--Eilenberg resolutions has already been observed in \cite{stai2018triangulated,brown2021minimal}.}

Cartan--Eilenberg resolutions have a very well-behaved flag-preserving lifting property (Lemma \ref{lem:CELiftingProperty}), but there is a fundamental problem that the canonical augmentation map $\widetilde{D} \xra{\sim} D$ is \emph{not} flag preserving if $D$ is itself a projective flag. This means that we cannot just na\"ively employ the flag-preserving lifting property of Cartan--Eilenberg resolutions to prove Theorem \ref{thm:IntroFlagPreserving}, so we will have to do something a little more subtle.

Let $D$ denote any anchored resolution. Then we instead show that there is a canonical construction of a Cartan--Eilenberg resolution $\widetilde{D}$ of $D$ in such a way that \emph{another} copy of $D$ appears as a flag-preserving \defi{deformation retract} of $\widetilde{D}$; thus we may employ the flag-preserving lifting property of Cartan--Eilenberg resolutions and restrict this flag-preserving morphism to the respective copies of $D$ and $D'$ appearing as deformation retracts to obtain the map $\psi$ (as in Theorem \ref{thm:IntroFlagPreserving}). The perturbation lemma plays a key role in all of this, not only by ensuring that the deformation retracts exist, but also by providing explicit formulas for the resulting deformation retract datum which we use to show compatibility with the canonical augmentation maps. 

Theorem \ref{thm:IntroFlagPreserving} simplifies and streamlines almost all of the arguments used in \cite{banks2022differential,banks2023differential}, since essentially all of the proofs there can now freely assume without any loss of generality that the morphisms being considered are flag-preserving without any kind of finiteness assumptions.



\subsection{Uniqueness of Anchored Resolutions}

Assume for the moment that $(R , \m , \kk)$ is a local ring (or positively-graded $\kk$-algebra). In \cite{brown2021minimal}, Brown--Erman defined the notion of a minimal free resolution of a differential module. Their goal was to initiate the development a homological theory of differential modules that mimicked the classical theory for modules, motivated by the connection they had noticed between differential modules and the more general theory of Tate resolutions over toric varieties. 

One of the first steps in this program would be to define the notion of a \emph{minimal} free resolution of a differential module in such a way that these resolutions exist and are \emph{unique} up to isomorphism. This turns out to be a nontrivial task: although free flag resolutions always exist, it is not always possible to choose a free flag resolution that is minimal (in the sense that the differential has entries in the maximal ideal $\m$). Brown--Erman's solution to this problem was to instead define a minimal free resolution of a differential module $D$ as a quasi-isomorphism of differential modules $F \xra{\sim} D$, where the differential of $F$ is minimal and the ambient $R$-module is free, but also defined in such a way that there is a diagram
\[\begin{tikzcd}
	& {\widetilde{F}} \\
	F && D
	\arrow["\sim", from=2-1, to=2-3]
	\arrow[hook, from=2-1, to=1-2]
	\arrow["\sim", from=1-2, to=2-3]
\end{tikzcd}\]
where $\widetilde{F}$ is a free flag resolution of $D$, and the inclusion $F \hookrightarrow \widetilde{F}$ is a split inclusion with contractible cokernel. This definition of a minimal free resolution satisfies both of the conditions posed above: they exist and are unique up to isomorphism. However, one potential shortcoming of this definition is that a minimal free resolution of a differential module is not necessarily a free flag anymore. This fact is rather inconvenient since direct summands of free flags can exhibit some pathological behavior that is not possible for free flags themselves (see Example \ref{ex:failureOfIneq}). One of our motivating goals for this paper is to instead define a notion of minimal resolution that exists, is unique, but \emph{also} is a free flag resolution:

\begin{definition}
    An anchored resolution $F$ is \defi{quasiminimal} if the anchor is a minimal free resolution.
\end{definition}

Interestingly, Brown--Erman had already proved in \cite[Theorem 3.2]{brown2021minimal} that anchored resolutions exist, a fact which they called ``Degeneration to the homology". In fact, their proof of existence can alternatively be interpreted as an application of the perturbation lemma. The advantage of using the perturbation lemma to prove this result is that we can prove a functorial version of Brown--Erman's ``Degeneration to the homology", which, combined with Theorem \ref{thm:IntroFlagPreserving} yields the following:

\begin{theorem}\label{thm:IntroQuasiminimal}
    Assume $D$ is a differential module with finitely generated homology. Then $D$ admits a quasiminimal anchored resolution, and any such resolution is unique up to isomorphism. 
    
    Moreover, every anchored resolution is unique up to flag-preserving homotopy and contains the quasiminimal anchored resolution as a (flag-preserving) direct summand.
\end{theorem}

Thus, Theorem \ref{thm:IntroQuasiminimal} shows that quasiminimal anchored resolutions in the category of differential modules are the direct analogues of the classical notion of minimal free resolutions in the category of $R$-modules. Theorem \ref{thm:IntroQuasiminimal} also removes all finiteness assumptions on the projective dimension of the homology imposed in \cite{brown2021minimal}, and also makes no assumptions on the ambient grading nor the degree of the differential in the homogeneous case. It also shows that one does not need to leave the category of free flags in order to have a well-behaved notion of minimal free resolutions for differential modules, nor make any finiteness assumptions on the projective dimension of the homology.

\begin{example}\label{ex:BEquasiMinimalRes}
    Let $S = \kk [x,y]$ where $\kk$ is some field. As illustrated in \cite[Example 1.1]{brown2021minimal}, the differential module $D$ with underlying $S$-module $D := S^{\oplus 2}$ and squarezero endomorphism
    $$d^D := \begin{pmatrix}
        -x y & -y^2 \\ x^2 & xy
    \end{pmatrix}$$
    has free flag resolution $F$ given by a perturbation of the Koszul complex
\[\begin{tikzcd}
	{K_2} && {K_1} && {K_0}
	\arrow["{d_2^K}", from=1-1, to=1-3]
	\arrow["{d_1^K}", from=1-3, to=1-5]
	\arrow["1", curve={height=24pt}, from=1-1, to=1-5]
\end{tikzcd}\]
    More precisely, the underlying free module of $F$ is $S^{\oplus 4}$ equipped with the squarezero endomorphism
    $$d^F := \begin{pmatrix}
        0 & x & y & 1 \\
        0 & 0 & 0 & -y \\
        0 & 0 & 0 & x \\
        0 & 0 & 0 & 0
    \end{pmatrix}.$$
    Now, in the context of Brown--Erman, the minimal free resolution of $D$ is actually $D$ itself, whereas the free flag resolution $F$ is a quasiminimal resolution and thus already unique up to isomorphism (even though the defining endomorphism is \emph{not} minimal, the content of Theorem \ref{thm:IntroQuasiminimal} is that $F$ is still unique up to isomorphism).
\end{example}
Another advantage of employing the perturbation lemma to prove Theorem \ref{thm:IntroQuasiminimal} is that quasiminimal resolutions may be constructed in a way that is compatible with any kind of additional gradings, a fact that will be essential for our applications to rank conjectures.

\subsection{A Total Rank Conjecture for Free Flags}

In the work \cite{avramov2007class}, Avramov--Buchweitz--Iyengar recognized the significance of projective flags to understanding and lending new perspectives on the classical rank conjectures arising in the context of algebraic topology \cite{halperin1985rational,carlsson1986free}. One of their main results was the establishment of lower bounds on the \defi{projective class} of a differential module $D$, which can be thought of as the analogue of projective dimension, purely in terms of the height of the annihilator $\ann_R H(D)$. They furthermore conjectured that for any projective flag $D$, there is an inequality
$$\rank_R D \geq 2^{c} , \quad \text{where} \ c := \codim_R \ann_R H(D).$$
This conjecture is now known to be false with counterexamples being provided (even in the category of chain complexes) by Iyengar--Walker \cite{iyengar2018examples}, but it still begs the question: is there an analogue of the Total Rank Conjecture in the category of differential modules?

One of the reasons this question can now be tackled is that the existence/uniqueness of anchored resolutions implies that the category of differential modules is largely indistinguishable from the category of chain complexes, at least from the K-theoretic perspective. This is crucial since the techniques of Walker \cite{walker2017total,vandebogert2023total} used to prove the Total Rank Conjecture for resolutions of modules heavily take advantage of an understanding of certain K-theoretic operations on the full subcategory of perfect complexes with finite length homology.

Before stating the theorem, we need to define two types of objects. Firstly, a $\bbz / 2 \bbz$-graded projective flag is simply a projective flag that also admits a splitting
$$F = F_0 \oplus F_1,$$
where the differential of $F$ has $\bbz / 2 \bbz$-degree $1$ (for the more precise definition see Definition \ref{def:flagDefs}). The homology of any such $\bbz / 2 \bbz$-graded object splits into a degree $0$ and $1$ portion. Secondly, we will say that a differential module is \defi{homologically finite} if $\pd_R H(D) < \infty$. Then we have the following analogue of the Total Rank Conjecture for differential modules:

\begin{theorem}\label{thm:TRCIntro}
    Assume $R$ is either
    \begin{enumerate}
        \item A quasi-Roberts ring (such as any locally complete intersection) with no $2$-torsion, or
        \item any ring of equicharacteristic $\neq 2$. 
    \end{enumerate}
    Let $D$ be any homologically finite $\bbz / 2 \bbz$-graded projective flag with $H(D) = H_0 (D)$.\footnote{One can instead assume $H (D) = H_1 (D)$ and the theorem still holds.} Then there is an inequality
    $$\rank_R D \geq 2^c,$$
    where $c := \codim_R \ann_R H(D)$. 
\end{theorem}

Note that if one does not assume $H(D) = H_0 (D)$ then Theorem \ref{thm:TRCIntro} is false with counterexamples being provided by Iyengar--Walker \cite{iyengar2018examples}. If $D$ arises as the folding of a complex then the condition that $H(D) = H_0 (D)$ is equivalent to asking that the homology of the original complex is concentrated in even homological degrees, a case which was already known by Walker \cite{walker2017total}. Finally, the reason we need a $\bbz / 2 \bbz$-grading on our differential modules is that the category $\dm (R)$ has \emph{no} monoidal structure, but differential modules with a $\bbz /2 \bbz$-grading do admit a monoidal structure.

The proof of Theorem \ref{thm:TRCIntro} in the end is strikingly similar to the original proof of Walker, but establishing the necessary K-theoretic facts in this new category of $\bbz / 2 \bbz$-graded differential modules is not straightforward. Indeed, Theorem \ref{thm:TRCIntro} is a consequence of the following more general statement about the interaction of \emph{cyclic} Adams operations $\psi^k_{cyc}$ introduced in \cite{brown2017cyclic,brown2017adams} with Euler characteristic/Dutta multiplicities of differential modules (see Section \ref{sec:KtheoryAndRank} for definitions):

\begin{theorem}\label{thm:IntroAdamsIdentities}
    Let $k \geq 1$ be any integer. Assume $R$ is a local ring containing all $t$-th roots of unity for $t \leq k$ and such that $k!$ is a unit. Then:
    \begin{enumerate}
        \item If $R$ is a quasi-Roberts ring and $D$ is any homologically finite $\bbz / 2 \bbz$-graded free flag with finite length homology, then there is an equality
        $$ \chi \circ \psi_{cyc}^k (D)  = k^{\dim R} \chi (D),$$
        where $\chi$ denotes the Euler characteristic.
        \item If $R$ is an equicharacteristic ring and $D$ is any homologically finite $\bbz / 2 \bbz$-graded free flag with finite length homology, then there is an equality
        $$ \chi_\infty \circ \psi_{cyc}^k  (D) = k^{\dim R} \chi_\infty (D),$$
        where $\chi_\infty$ denotes the Dutta multiplicity \cite{dutta1983frobenius}.
    \end{enumerate}
    If $R$ is a regular local ring, it suffices to assume only that $k$ is a unit and $R$ contains a primitive $k$th root of unity. 
\end{theorem}

As previously mentioned, anchored resolutions are an essential ingredient in the above proofs. In the work \cite{brown2017adams} the authors were able to prove Theorem \ref{thm:IntroAdamsIdentities} under the assumption that $R$ is regular, but their proof heavily relied on the fact that $R$ has finite global dimension, so it is not obvious how to extend their techniques to more general rings.\footnote{Indeed, taking advantage of the existence/uniqueness of anchored resolutions significantly streamlines the argument originally given in \cite{brown2017adams}.} A key step in the proof of Theorem \ref{thm:IntroAdamsIdentities} is that we can define certain derived eigenspace operators $\mathbf{t}^k_{\zeta}$ (similar to \cite[Lemma 3.9]{brown2017adams}) for any $k$th root of unity $\zeta$. The existence and uniqueness of anchored resolutions ensure that this operator is well-defined, and we reduce the equalities of Theorem \ref{thm:IntroAdamsIdentities} to the analogous identities on the category of perfect chain complexes using a certain spectral sequence that computes the homology of a free flag. 

\subsection*{Acknowledgments} 

The author thanks Maya Banks, Ben Briggs, Michael Brown, Daniel Erman, and Josh Pollitz for helpful conversations related to this material and/or suggestions on earlier drafts of this paper. Thanks to Srikanth Iyengar for suggesting that the perturbation lemma might be useful for studying differential modules. The author also gives a special thanks to Mark E. Walker for many helpful discussions related to K-theoretic properties of linear factorizations/cyclic Adams operations. The author gratefully acknowledges support from NSF grant DMS-2202871.

\subsection*{Conventions} Throughout this paper, $R$ will denote any commutative Noetherian ring. All of the constructions of this paper still go through when taking into account any kind of auxiliary gradings inherited from the ring $R$. All differential modules will be assumed to have finitely generated homology (though the ambient module itself need not be finitely generated). In the graded setting, all morphisms are assumed to be homogeneous with respect to the ambient grading. 

Given a chain complex of $R$-modules $C$, the notation $C[i]$ is defined to be the chain complex with $C[i]_j := C_{i+j}$, and differential $d^{C[i]} := (-1)^i d^{C}$. 

\subsection*{Notation} We use the following notation throughout the paper:
\begin{itemize}
    \item The notation $\codim_R I$ for an ideal $I \subset R$ is defined via $\codim_R I := \dim R - \dim (R/I)$. 
    \item The notation $\pd_R M$ denotes the projective dimension of an $R$-module $M$. 
    \item The notation $R\proj$ denotes the category of projective $R$-modules. 
    \item The notation $K_0 (R)$ denotes the Grothendieck group of bounded complexes of projective $R$-modules.
    \item The notation $C_{\geq i}$ for a chain complex $C$ denotes the brutal truncation of $C$.
\end{itemize}




\section{Flagged Perturbations of Complexes}\label{sec:flaggedPerts}

In this section, we initiate our study of projective flags through the lens of homological perturbation theory. We first begin by recalling the notions of perturbation and two versions of the perturbation lemma: the first version (Lemma \ref{lem:thePertLemma1}) is the more well-known version that applies to deformation retracts, but the second version (Lemma \ref{lem:PertLemma2}) originally due to Huebschmann--Kadeishvili \cite{huebschmann1991small} applies to arbitrary homotopy equivalences. We then extend a construction of Stai \cite{stai2018triangulated} to the $\bbz / d \bbz$-graded setting, and in turn use this to define $\bbz / d \bbz$-graded Cartan--Eilenberg resolutions (Definition \ref{def:CEResolution}). These Cartan--Eilenberg resolutions will be a convenient ``standard resolution" that we will continually use as an intermediary object between anchored resolutions and arbitrary differential modules. 

Subsection \ref{subsec:degenToHomology} proves the aforementioned functorial version of the ``Degeneration to the homology" proved by Brown--Erman \cite{brown2021minimal} in the $\bbz / d \bbz$-graded setting by taking advantage of the fact that Cartan--Eilenberg resolutions are obtained as flagged perturbations of particularly simple complexes. For full background on $\bbz / d \bbz$-graded differential modules/flag structures, see the Appendix.

\subsection{The Perturbation Lemma and its Naturality}\label{subsec:perturbationLemmaAndNaturality}

In this subsection, we recall some of the more standard facts of homological perturbation theory. For more details and many different variants of the perturbation lemma, see, for instance, \cite{brown1967twisted,huebschmann1991small,barnes1991fixed,crainic2004perturbation,Stasheff+2010+203+215}.

\begin{definition}\label{def:perturbations}
    Let $D \in \dm_{\bbz / d \bbz} (R)$. A \defi{perturbation} $\delta$ of $D$ is any morphism of the underlying $R$-module of $D$ satisfying
    $$(d^D + \delta)^2 = 0.$$
    A perturbation $\delta$ is $\bbz / d \bbz$-graded if the resulting differential $d^D + \delta$ is $\bbz / d \bbz$-graded. The notation $D_\delta$ will denote the perturbed differential module, equipped with the new differential $d^{D_\delta} := d^D + \delta$. 
\end{definition}

\begin{remark}
    We do not assume as in \cite{crainic2004perturbation} that the perturbation $\delta$ has the same degree as the differential $d^D$ of the differential module $D$. Since the computations involved in the proof of the perturbation lemma are completely formal in nature and only use the fact that the differential squares to $0$, the statements of the perturbation lemmas (Lemmas \ref{lem:thePertLemma1} and \ref{lem:PertLemma2} below) may be stated for differential modules without any loss of generality. There is also a version of the perturbation lemma for curved objects (see, for instance, \cite{hitchcock2022perturbation}), but we will have no need for this level of generality here.
\end{remark}

\begin{example}\label{ex:perturbedKoszul}
    Let $R = \kk [ x_1,x_2 , x_3]$ be a polynomial ring over some field $\kk$ and consider the Koszul complex $K$ on the variables $x_1,x_2,x_3$:
\[\begin{tikzcd}
	{K:} & {K_3} & {K_2} & {K_1} & {K_0}
	\arrow["{d_3}", from=1-2, to=1-3]
	\arrow["{d_2}", from=1-3, to=1-4]
	\arrow["{d_1}", from=1-4, to=1-5]
\end{tikzcd}\]
    We can perturb the Koszul complex by adding an identity map $\id_{R} : K_3 \to K_0$ to obtain a free flag of the form
\[\begin{tikzcd}
	{K_\delta:} & {K_3} & {K_2} & {K_1} & {K_0}
	\arrow["{d_3^K}", from=1-2, to=1-3]
	\arrow["{d_2^K}", from=1-3, to=1-4]
	\arrow["{d_1^K}", from=1-4, to=1-5]
	\arrow["1"', curve={height=18pt}, from=1-2, to=1-5]
\end{tikzcd}\]
    This perturbation is no longer $\bbz$-graded, but it is $\bbz / 2 \bbz$-graded since the new perturbed differential still changes the parity of the homological degree. More precisely, this perturbed object may be viewed as a $\bbz / 2 \bbz$-graded differential module with
    $$(K_\delta)_0 = K_0 \oplus K_2, \quad (K_\delta)_1 = K_1 \oplus K_3,$$
    $$d_1^{K_\delta} = \begin{pmatrix}
        d_1^K & 1 \\
        0   & d_3^K \\
    \end{pmatrix}, \quad d_0^{K_\delta} = \begin{pmatrix}
        0 & d_2^K \\
        0 & 0 \\
    \end{pmatrix}.$$
\end{example}

\begin{definition}[Flagged Perturbations]\label{def:flaggedPerturbation}
    Let $C$ be a nonnegatively-graded chain complex. A \defi{flagged perturbation} of $C$ is a perturbation $\delta$ of $C$ that admits a decomposition $$\delta = \delta_1 + \delta_2 + \cdots , \quad \text{with} \ \delta_i \in \hom^{i+1} (C,C) \quad \text{for each} \ i \geq 1.$$
\end{definition}

\begin{remark}
    Note that if $\delta$ is any flagged perturbation of a bounded below complex $C$ and $h$ is any homotopy of $C$, then $\delta h$ is locally nilpotent since this map \emph{strictly} drops homological degree. This means that flagged perturbations are always \defi{small} perturbations, which means that $1 - \delta h$ is invertible. This is particularly important for employing the perturbation lemma (see Lemma \ref{lem:thePertLemma1}).
\end{remark}

As mentioned previously, the main point of the perturbation lemma is to be able to transfer perturbations along homotopy equivalences. The most well-known version of the perturbation lemma is actually stated for a particularly well-behaved type of homotopy equivalence, known as a (strong) deformation retract.

\begin{definition}[(Strong) Deformation Retracts]\label{def:DeformationRetract}
     A \defi{deformation retract}
\[\begin{tikzcd}
	h, \quad C && D
	\arrow["p", curve={height=-18pt}, from=1-1, to=1-3]
	\arrow["i", curve={height=-18pt}, from=1-3, to=1-1]
\end{tikzcd}\]
    between two differential modules $C$ and $D$ is the data of morphisms $p : C \to D$ and $i : D \to C$ with $p \circ i = \id_D$ and $i \circ p$ homotopic to $\id_{C}$ via some homotopy $h$. In other words, it is a homotopy equivalence with $h' = 0$ as in Notation \ref{not:HtpyEquivalence} (so that $D$ is a direct summand of $C$). 

    A deformation retract is \defi{strong (or special)} if the following additional conditions are satisfied:
    $$h^2 = 0, \quad h i = 0, \quad \text{and} \quad p h = 0.$$
\end{definition}

\begin{remark}\label{rk:modificationForSDR}
    Any deformation retract can be modified in such a way that it becomes a strong deformation retract. This is achieved by the following sequence of modifications:
    \begin{enumerate}
        \item Replace $h$ with $h(d^C d + h d^C)$ to obtain $h \iota = 0$,
        \item then replace $h$ with $(d^C h + h d^C) h$ to obtain $p h = 0$, and
        \item then replace $h$ with $h d^C h$ to obtain $h^2 = 0$.
    \end{enumerate}
\end{remark}

\begin{remark}
    Note that throughout this paper we use the definition of homotopy equivalence as any morphism that is invertible up to homotopy (Notation \ref{not:HtpyEquivalence}). This definition is different from the definition used by Crainic \cite{crainic2004perturbation}, and in fact one must use a stronger formulation of the perturbation lemma (Lemma \ref{lem:PertLemma2}) than that given by Crainic for our definition of homotopy equivalence.
\end{remark}

We can now state the first version of the perturbation lemma:

\begin{lemma}[The Perturbation Lemma: First Version]\label{lem:thePertLemma1}
    Assume that $\delta$ is a small perturbation of some differential module $C$ and there is a deformation retract
\[\begin{tikzcd}
	C && D
	\arrow["p", curve={height=-18pt}, from=1-1, to=1-3]
	\arrow["i", curve={height=-18pt}, from=1-3, to=1-1]
\end{tikzcd}\]
    Given a perturbation $\delta$ of $C$ with $1-\delta h$ invertible, define the following data:
    $$p_\infty := p + p (1-\delta h)^{-1} \delta h, \quad i_\infty := i + h (1 - \delta h )^{-1} \delta i, \quad h_\infty := h + h (1-\delta h)^{-1} \delta h, \quad \delta_{\infty} := p(1 - \delta h)^{-1} \delta i.$$
    Then $\delta_\infty$ is a perturbation of $D$ and there is an induced homotopy equivalence:
\[\begin{tikzcd}
	{h_\infty , \quad (C, d^C + \delta)} && {(D, d^D + \delta_\infty)}
	\arrow["{p_\infty}", curve={height=-24pt}, from=1-1, to=1-3]
	\arrow["{\iota_\infty}", curve={height=-24pt}, from=1-3, to=1-1]
\end{tikzcd}\]
    If the original deformation retract is strong, then the resulting perturbed data is also a strong deformation retract. 
\end{lemma}

\begin{remark}
    In more modern language, the perturbation lemma is often viewed as a special case of the more general \defi{homotopy transfer theorem} (see, for instance, \cite[Chapter 10.3]{loday2012algebraic}), which instead transfers homotopy operadic structures on algebras along homotopy equivalences of dg algebras.
\end{remark}

The perturbation lemma is typically stated for deformation retracts, but there is a more general version that applies to arbitrary homotopy equivalences \cite{huebschmann1991small}. The idea of proving this second version is to use the more classical version of the perturbation lemma on an appropriate intermediary object. More precisely, the data of a homotopy equivalence $\phi: C \xra{\sim} D$ induces deformation retracts from the mapping cylinder $\operatorname{Cyl} (\phi)$ onto \emph{both} $C$ and $D$:
\[\begin{tikzcd}
	& {\operatorname{Cyl} (\phi)} \\
	C && D
	\arrow["\sim"', from=1-2, to=2-1]
	\arrow["\sim", from=1-2, to=2-3]
\end{tikzcd}\]
There is a natural choice for extending a perturbation of $D$ to the cylinder $\operatorname{Cyl} (\phi)$, and the assumption that $\phi$ is a homotopy equivalence implies that there is a deformation retract of the mapping cylinder onto $C$. The first version of the perturbation lemma applied to this deformation retract yields the desired perturbation on $C$:

\begin{lemma}[The Perturbation Lemma: Second Version]\label{lem:PertLemma2}
     Assume that there is a homotopy equivalence of differential modules:
\[\begin{tikzcd}
	{h , \quad (C, d^C)} && {(D, d^C ), \quad h'}
	\arrow["p", curve={height=-24pt}, from=1-1, to=1-3]
	\arrow["\iota", curve={height=-24pt}, from=1-3, to=1-1]
\end{tikzcd}\]
     Given a small pertubation $\delta$ of $C$, there exists a perturbation $\delta_\infty$ of $D$ and an induced homotopy equivalence:
\[\begin{tikzcd}
	{h_\infty , \quad (C, d^C + \delta)} && {(D, d^C + \delta_\infty), \quad h'_\infty}
	\arrow["{p_\infty}", curve={height=-24pt}, from=1-1, to=1-3]
	\arrow["{\iota_\infty}", curve={height=-24pt}, from=1-3, to=1-1]
\end{tikzcd}\]
\end{lemma}

\begin{remark}
    Since we will not use the explicit form of the perturbation or homotopy equivalence data in the above version of the perturbation lemma, we will not record these formulas and instead direct the interested reader to \cite{huebschmann1991small}.
\end{remark}

The following fact is not typically stated explicitly with the perturbation lemma, but seems to be tacitly used in some spots in the literature. In our context, we will often have to refer to this naturality statement for later results, so we state it here for convenience. Recall the notation $C_\delta$ established in Definition \ref{def:perturbations} in the statement below:

\begin{lemma}[Naturality of Perturbation]\label{lem:PerturbationNaturality}
    Assume that there are homotopy equivalences of differential modules
\[\begin{tikzcd}
	{h , \quad C} && {D, \quad s} & {\text{and}} & {h', \quad C'} && {D', \quad s'}
	\arrow["p", curve={height=-18pt}, from=1-1, to=1-3]
	\arrow["\iota", curve={height=-18pt}, from=1-3, to=1-1]
	\arrow["{p'}", curve={height=-18pt}, from=1-5, to=1-7]
	\arrow["{\iota'}", curve={height=-18pt}, from=1-7, to=1-5]
\end{tikzcd}\]
     Suppose furthermore that there exist perturbations $\delta^C, \delta^{C'}$ of $C$ and $C'$, respectively and a morphism $\phi : C_{\delta^C} \to C'_{\delta^{C'}}$ such that the following diagram commutes up to homotopy:
\[\begin{tikzcd}
	{C_{\delta^C}} && {C'_{\delta^{C'}}} \\
	\\
	E && {E'}
	\arrow["\phi", from=1-1, to=1-3]
	\arrow[from=1-1, to=3-1]
	\arrow[from=3-1, to=3-3]
	\arrow[from=1-3, to=3-3]
\end{tikzcd}\]
    Then there exist perturbations $\delta^D$ and $\delta^{D'}$ of $D$ and $D'$, respectively, such that the following diagram commutes up to homotopy:
    \[\begin{tikzcd}
	{D_{\delta^D}} && {D'_{\delta^{D'}}} \\
	\\
	E && {E'}
	\arrow["\phi", from=1-1, to=1-3]
	\arrow[from=1-1, to=3-1]
	\arrow[from=3-1, to=3-3]
	\arrow[from=1-3, to=3-3]
\end{tikzcd}\]
\end{lemma}

\begin{proof}
    The existence of the corresponding perturbations $\delta^D$ and $\delta^{D'}$ is a consequence of the perturbation lemma (Lemma \ref{lem:PertLemma2}), and the fact that the evident induced diagram commutes up to homotopy is a consequence of Lemma \ref{lem:HtpyEquivAndCommutingDiagrams}.
\end{proof}

\subsection{Cartan--Eilenberg Resolutions for Differential Modules}\label{subsec:CEResolutions}

In this section, we recall Stai's construction of Cartan--Eilenberg type resolutions with special attention to the case of $\bbz / d \bbz$-gradings. This allows us to define $\bbz / d \bbz$-graded Cartan--Eilenberg resolutions, and the explicitness/simplicity of this construction will allow us to prove directly that morphisms of differential modules can always be canonically extended to Cartan--Eilenberg resolutions in a flag-preserving way (Lemma \ref{lem:CELiftingProperty}). This fact will be one of the first fundamental stepping stones that we will need to prove many of our main results on the existence/uniqueness of anchored resolutions.

\begin{definition}[$\bbz/d\bbz$-graded Projective Flag Resolutions]
    Let $D \in \dm_{\bbz / d \bbz} (R)$ be any $\bbz / d \bbz$-graded differential module. Then a \defi{$\bbz / d \bbz$-graded projective flag resolution} is an object $F \in \flag_{\bbz / d \bbz} (R\proj)$ equipped with a quasi-isomorphism
    $$F \xra{\sim} D.$$
\end{definition}

\begin{remark}
    In the $d = 1$ case, it is well-known that projective flag resolutions always exist. Two such constructions of these resolutions are outlined in \cite{brown2021minimal}, although the authors did not need to consider the $\bbz / d \bbz$-graded setting for their purposes. The iterated mapping cone construction of \cite[Construction 2.8]{brown2021minimal}, which is itself precisely the construction of an Adams resolution \cite{christensen1998ideals}, extends to the $\bbz / d \bbz$-graded setting in a straightforward way. For our purposes it will be more desirable to have a $\bbz / d \bbz$-graded analogue of a construction due to Stai. Keeping track of the $\bbz / d \bbz$-graded pieces requires just a little more bookkeeping.
\end{remark}

\begin{remark}
    For the reader familiar with dg algebras, the definition of a free flag resolution should look very familiar to the notion of a \defi{semifree resolution} \cite{avramov1998infinite}. It is, however, important to note that the theory of free flag resolutions is \emph{not} a subset of the theory of semifree resolutions, since there is no way to view a differential module as a dg module over any dg algebra. 
\end{remark}

\begin{construction}[$\bbz / d \bbz$-graded Stai Construction]\label{cons:StaiRes}
    Let $D$ be a $\bbz / d \bbz$-graded differential module. Use the shorthand notation $Z_i := Z_i (D)$, $B_i := B_i (D)$, and $H_i := H_i (D)$. For each $i = 0 , \dots , d-1$ there are short exact sequences $$ 0 \to Z_i \to D_i \to B_{i-1} \to 0, \quad \text{and} \quad 0 \to B_i \to Z_i \to H_i \to 0.$$
    Choose projective resolutions $F^{B_i}$ and $F^{H_i}$ of $B_i$ and $H_i$, respectively, for each $i = 0 , \dots , d-1$. Use the notation $G^{B_i}$ to denote the same complex as $F^{B_i}$ but with the differential negated; this distinction will be useful since the Stai construction will have multiple copies of $F^{B_i}$ that will be hard to distinguish otherwise.

    Employing the Horsehoe lemma iteratively yields a resolution of $D_j$ for each $j = 0 ,\dots , d-1$. The differential of the resolution of each term $D_j$ has a block form decomposition of the form
    $$\begin{pmatrix}
        d^{F^{B_i}} & \alpha^i & \gamma^i \\
        0 & d^{F^{H_i}} & \beta^i \\ 
        0 & 0 & d^{G^{B_{i-1}}} \\
    \end{pmatrix},$$
    for maps $\alpha^i : F^{H_i} \to F^{B_i}[-1]$, $\beta^i : G^{B_{i-1}} \to F^{H_i}[-1]$, and $\gamma^i : $

    Define the $\bbz / d \bbz$-graded projective flag $G$ with associated graded pieces
    $$G_{i,j} := F_i^{B_{j-i}} \oplus F_i^{H_{j-i}} \oplus G_{i-d}^{B_{j-i-1}},$$
    and differential induced as follows:
    \begin{itemize}
        \item The differential $d^G$ restricted to $F_i^{B_{j-i}}$ has one component and it is just the differential of $F^{B_{j-i}}$.
        \item The differential $d^G$ restricted to $F_{i}^{H_{j-i}}$ has two components: it maps to $F_{i-1}^{H_{j-i}}$ via the differential of $F^{H_{j-i}}$ and it maps to $F_{i-1}^{B_{j-i}}$ via the map $\alpha^{j-i}$. 
        \item The differential $d^G$ restricted to $G_{i-d}^{B_{j-i}}$ has four components: it maps to $G_{i-d-1}^{B_{j-i-1}}$ via the differential of $G^{B_{i-j-1}}$, it maps to $F^{B_{j-i-1}}_{i-d}$ as the identity map, it maps to $F^{H_{j-i}}_{i-d-1}$ via the map $\beta^{j-i}$, and it maps to $F^{B_{j-i}}_{i-d-1}$ via the map $\gamma^{j-i}$. 
    \end{itemize}
    Diagramatically, the differential restricted to $G_{i,j}$ is represented as follows:
\[\begin{tikzcd}
	&&&& {F^{B_{j-i-1}}_{i-d}} \\
	{G^{B_{j-i-1}}_{i-d}} && {G^{B_{j-i-1}}_{i-d-1}} && \cdots && {G^{B_{j-i-1}}_{j-2d-1}} \\
	{F_i^{H_{j-i}}} && {F^{H_{j-i}}_{i-1}} && \cdots && {F^{H_{j-i}}_{i-d-1}} \\
	{F^{B_{j-i}}_i} && {F^{B_{j-i}}_{i-1}} && \cdots && {F^{B_{j-i}}_{i-d-1}} \\
	{G_{i,j}} && {G_{i-1,j-1}} & \cdots & {G_{i-d,j-1}} && {G_{i-d-1,j-1}}
	\arrow["{d^{B_{j-i}}}"{description}, from=4-1, to=4-3]
	\arrow["{d^{H_{j-i}}}", from=3-1, to=3-3]
	\arrow["{d^{B_{j-i-1}}}"{description}, from=2-1, to=2-3]
	\arrow[curve={height=-6pt}, Rightarrow, no head, from=2-1, to=1-5]
	\arrow["{\alpha^{j-i}}"{description}, from=3-1, to=4-3]
	\arrow["{\gamma^{j-i}}"{description}, from=2-1, to=4-7]
	\arrow["{\beta^{j-i}}"{description}, from=2-1, to=3-7]
	\arrow["\bigoplus"{description}, draw=none, from=2-3, to=3-3]
	\arrow["\bigoplus"{description}, draw=none, from=3-3, to=4-3]
	\arrow["\bigoplus"{description}, draw=none, from=3-1, to=4-1]
	\arrow["\bigoplus"{description}, draw=none, from=2-1, to=3-1]
	\arrow["\bigoplus"{description}, draw=none, from=3-7, to=4-7]
	\arrow["\bigoplus"{description}, draw=none, from=2-7, to=3-7]
\end{tikzcd}\]
    In particular, note that $d^G (G_{i,j}) \subset G_{i-1,j-1} \oplus G_{i-d,j-1} \oplus G_{i-d-1,j-1}$.\footnote{When $d = 1$ this just becomes $d^G (G_i) \subset G_{i-1} \oplus G_{i-2}$, as was already noted in \cite[Construction 2.7]{brown2021minimal}.} 

    There is a canonical morphism of differential modules induced by the augmentation map $F^{D_j} \to D_j$; more precisely, the only nonzero components of the augmentation map mapping to $D_j$ come from $G_{0,j} = F_0^{B_j} \oplus F_0^{H_j}$ and the direct summand $G_0^{B_{j-1}}$ of $G_{d,j} = F_d^{B_{j}} \oplus F_d^{H_{j}} \oplus G_0^{B_{j-1}}$, and are induced by the natural augmentation maps coming from the chosen Horseshoe resolutions. 
\end{construction}

\begin{theorem}\label{thm:CEresExists}
    Given any differential module $D$, the projective flag $G$ of Construction \ref{cons:StaiRes} is a $\bbz / d \bbz$-graded projective flag resolution of $D$ (with augmentation map given as in Construction \ref{cons:StaiRes}).
\end{theorem}

\begin{proof}
    The only thing to check is that the augmentation map $\eta: G \to D$ is a well-defined morphism of differential modules. However, the augmentation $\eta$ is chosen in such a way that $\eta (F^{B_{j-1}}) = d^D \circ \eta (G^{B_j})$; this implies that the augmentation restricted to $G_0^{B_j}$ commutes with the differential for each $j=0 , \dots , d-1$. Since $\eta$ restricted to the modules $F_0^{B_j}$ and $F^{H_j}$ maps to the cycles of $D$ by construction, it is trivial that $\eta$ commutes with the differential when restricted to these terms. Thus $\eta : G \xra{\sim} D$ is a well-defined augmentation map, and by construction $\eta$ induces an isomorphism on homology.
\end{proof}

\begin{remark}\label{rk:sameAsCartanEilenberg}
    In the case that $d = 0$ (that is, for the category of complexes), Stai's construction is precisely the construction of a Cartan--Eilenberg resolution, after being suitably interpreted. More precisely, the double indexing on $G$ is a $\bbz \times \bbz$-grading that realizes the corresponding Stai construction as the total complex of the same bicomplex constructed for the Cartan--Eilenberg construction.  
\end{remark}

In view of Remark \ref{rk:sameAsCartanEilenberg}, we define:

\begin{definition}\label{def:CEResolution}
    A projective flag resolution constructed as in Construction \ref{cons:StaiRes} is called a ($\bbz / d \bbz$-graded) \defi{Cartan--Eilenberg resolution} of the differential module $D$. 
\end{definition}

\begin{lemma}[Lifting Property of Cartan--Eilenberg Resolutions]\label{lem:CELiftingProperty}
    Let $ \phi : D \to D'$ be a morphism of differential modules. Then for any two Cartan--Eilenberg resolutions $G$ and $G'$ of $D$ and $D'$, respectively, there exists a flag-preserving morphism $\widetilde{\phi} : G \to G'$ such that the following diagram commutes:
\[\begin{tikzcd}
	G && {G'} \\
	\\
	D && {D'}
	\arrow[from=1-1, to=3-1]
	\arrow[from=1-3, to=3-3]
	\arrow["\phi"', from=3-1, to=3-3]
	\arrow["{\widetilde{\phi}}", from=1-1, to=1-3]
\end{tikzcd}\]
where both vertical arrows are the canonical augmentation maps. 
\end{lemma}

\begin{proof}
    The map $\phi : D \to D'$ induces maps $H_j(\phi) : H_j(D) \to H_j(D')$ and $B_j (\phi ) : B_j (D) \to B_j (D')$ for each $j = 0 ,\dots , d-1$. The comparison theorem ensures that there are induced morphisms of complexes
    $$\nu^j :F^{H_j (D)} \to F^{H_j (D')}, \quad  \psi^j : F^{B_j (D)} \to F^{B_j (D')}$$
    extending the maps $H_j (\phi)$ and $B_j (\phi)$, respectively. The fact that $F^{B_j (D')}$ is a resolution implies that the diagram
\[\begin{tikzcd}
	{F^{H_j (D)}} && {F^{B_j (D)}[-1]} \\
	\\
	{F^{H_j (D')}} && {F^{B_j (D')}[-1]}
	\arrow["{\alpha^j_D}", from=1-1, to=1-3]
	\arrow["{\psi^j}"', from=1-1, to=3-1]
	\arrow["{-\nu^j}", from=1-3, to=3-3]
	\arrow["{\alpha^j_{D'}}"', from=3-1, to=3-3]
\end{tikzcd}\]
   commutes up to some homotopy $h^j$, whence by Lemma \ref{lem:coneOnHarrow} there is an induced morphism of complexes $\cone (\alpha^j_D) \to \cone (\alpha^j_{D'})$ for each $j = 0 ,\dots , d-1$. Another application of the comparison theorem ensures that there exist maps $\theta^j = \begin{pmatrix} \gamma^j_D \\ \beta^j_D \end{pmatrix} : G^{B_j (D)} \to \cone (\alpha^j_D)$ and $\eta^j = \begin{pmatrix} \gamma^j_{D'} \\ \beta^j_{D'} \end{pmatrix} : G^{B_j (D')} \to \cone (\alpha^j_{D'})$ making the following diagram commute up to some homotopy ${h'}^j = \begin{pmatrix} s^j & {s'}^j \end{pmatrix} : \cone (\alpha^j_{D'}) \to G^{B_{j-1} (D)} [1]$:
\[\begin{tikzcd}
	{G^{B_{j-1} (D)}} && {\cone (\alpha^j_D)[-1]} \\
	\\
	{G^{B_{j-1} (D')}} && {\cone (\alpha^j_{D'})[-1]}
	\arrow["{\theta^j}", from=1-1, to=1-3]
	\arrow["{\nu_j}"', from=1-1, to=3-1]
	\arrow[from=1-3, to=3-3]
	\arrow["{\eta^j}"', from=3-1, to=3-3]
\end{tikzcd}\]
   Let $G$ and $G'$ denote the Cartan--Eilenberg resolutions constructed from the resolutions of $H(D)$ and $B(D)$ (resp. $H(D')$ and $B(D')$) above and consider the map whose block matrix decomposition has the following form:
   $$\Phi:=\begin{pmatrix}
       \nu & -s & \nu - s' \\
       0 & \psi & -h \\
       0 & 0 & \nu
   \end{pmatrix} : G \underbrace{=}_{\text{as} \ R\text{-modules}} F^{B(D)} \oplus F^{H(D)}  \oplus G^{B(D)} \to G' \underbrace{=}_{\text{as} \ R\text{-modules}} F^{B(D')} \oplus F^{H(D')}  \oplus G^{B(D')},$$
   where 
   $$h := \bigoplus_{j=0}^{d-1} h^j, \quad s := \bigoplus_{j=0}^{d-1} s^j, \quad  s' := \bigoplus_{j=0}^{d-1} {s'}^j, \quad \nu := \bigoplus_{j=0}^{d-1} \nu^j, \quad \text{and} \quad \psi := \bigoplus_{j=0}^{d-1} \psi^j.$$
   Then note that the differentials of $G$ and $G'$ have block forms given by
   $$d^G = \begin{pmatrix}
       d^{F^{B(D)}} & \alpha_D & \id_{F^{B(D)}} + \gamma_D \\
       0 & d^{F^{H(D)}} & \beta_D \\ 
       0 & 0 & d^{G^{B(D)}} \\
   \end{pmatrix}, \quad d^{G'} = \begin{pmatrix}
       d^{F^{B(D')}} & \alpha_{D'} & \id_{F^{B(D')}} + \gamma_{D'} \\
       0 & d^{F^{H(D)}} & \beta_{D'} \\ 
       0 & 0 & d^{G^{B(D)}} \\
   \end{pmatrix},$$
   where as above $d^{F^{B(D)}} := \bigoplus_{j=0}^{d-1} d^{F^{B_j(D)}}$ and similarly for all other terms. First, let us show that $\Phi : G \to G'$ induces a well-defined morphism of differential modules:
   \begingroup\allowdisplaybreaks
    \begin{align*}
        &\begin{pmatrix}
       d^{F^{B(D')}} & \alpha_{D'} & \id_{F^{B(D')}} + \gamma_{D'} \\
       0 & d^{F^{H(D)}} & \beta_{D'} \\ 
       0 & 0 & d^{G^{B(D)}} \\
   \end{pmatrix} \begin{pmatrix}
       \nu & -s & \nu - s' \\
       0 & \psi & -h \\
       0 & 0 & \nu
   \end{pmatrix}\\
   =& \begin{pmatrix}
       d^{F^{B(D')}} \nu & -d^{F^{B(D')}} s + \alpha_{D'} \psi & d^{F^{B(D')}} \nu -d^{F^{B(D')}} s' - \alpha_{D'} h + \nu + \gamma_{D'} \nu \\ 
       0 & d^{F^{H(D')}} \psi & -d^{F^{H(D')}} h + \beta_{D'} \nu \\
       0 & 0 & d^{G^{B(D)}} \nu \\
   \end{pmatrix} \\
   =& \begin{pmatrix}
       \nu d^{F^{B(D)}} & s d^{G^{B(D)}} + \nu \alpha_D & \nu d^{F^{B(D)}} + \nu  -s \beta_D + \nu \gamma_D - s' d^{G^{B(D)}}  \\
       0 & \psi d^{F^{H(D)}} & -h d^{G^{B(D)}} + \psi \beta_D \\
       0 & 0 & \nu d^{G^{B(D)}} \\
   \end{pmatrix} \\
   =& \begin{pmatrix}
       \nu & -s & \nu - s' \\
       0 & \psi & -h \\
       0 & 0 & \nu
   \end{pmatrix} \begin{pmatrix}
       d^{F^{B(D)}} & \alpha_D & \id_{F^{B(D)}} + \gamma_D \\
       0 & d^{F^{H(D)}} & \beta_D \\ 
       0 & 0 & d^{G^{B(D)}} \\
   \end{pmatrix}.
    \end{align*}
   \endgroup
   Finally, it remains to verify that the morphism $\Phi$ is compatible with the augmentation maps $\eta : G \to D$, $\eta' : G' \to D'$. Recall that the maps $\nu$ and $\psi$ were chosen from the start to be induced by the restriction of the initial map $\phi : D \to D'$ to the cycles and boundaries of $D$ and $D'$. However, the differentials of $G$ and $G'$ are built from perturbing horseshoe resolutions of the underlying modules of $D$ and $D'$ by an additional identity map (as in Construction \ref{cons:StaiRes}). Moreover, the morphism $\Phi$ is just a perturbation of the morphism of horseshoe resolutions:
   $$\begin{pmatrix}
       \nu & -s & \nu - s' \\
       0 & \psi & -h \\
       0 & 0 & \nu
   \end{pmatrix},$$
   and this morphism has been constructed precisely so that it extends the morphism of underlying $R$-modules $\phi : D \to D'$. Thus by construction the diagram of the lemma commutes, completing the proof.
\end{proof}

\begin{cor}[Uniqueness of Cartan--Eilenberg Resolutions Up To Homotopy]\label{cor:uniqueCEHtpy}
    Assume that $G$ and $G'$ are two $\bbz / d \bbz$-graded Cartan--Eilenberg resolutions of a fixed differential module $D$. Then $G$ and $G'$ are homotopy equivalent in a flag-preserving way.
\end{cor}

\begin{proof}
    If $D$ is fixed, any of the choices involved in the data of Construction \ref{cons:StaiRes} will be unique up to homotopy. Thus, the induced homotopy equivalence of Horseshoe resolutions of the underlying $R$-module $D$ can be modified precisely as in the proof of Lemma \ref{lem:CELiftingProperty} to obtain a flag preserving homotopy equivalence of the resulting Cartan--Eilenberg resolutions. 
\end{proof}

\subsection{Degeneration to the Homology via the Homological Perturbation Lemma}\label{subsec:degenToHomology}

In this section we use the perturbation lemma to deduce that anchored resolutions can be obtained explicitly as deformation retracts of Cartan--Eilenberg resolutions. Combining the naturality of the perturbation lemma with Lemma \ref{lem:CELiftingProperty} actually yields the flag-preserving analogue of Lemma \ref{lem:naiveLiftingLemma}. One important subtlety here is that this does \emph{not} imply Theorem \ref{thm:IntroFlagPreserving}, since the statement of Lemma \ref{lem:naiveLiftingLemma} is only an existence statement -- we still need to prove that we can obtain \emph{any} anchored resolution as a deformation retract of a Cartan--Eilenberg resolution in order to deduce Theorem \ref{thm:IntroFlagPreserving}. A dg-analogue of this degeneration to the homology can also be found in work of Keller \cite[Section 3.1, first theorem]{keller1994deriving}.

\begin{theorem}[Brown--Erman Degeneration to the Homology]\label{thm:degenToHomology}
    Let $D$ be a $\bbz / d \bbz$-graded differential module and let $F^{H_j (D)} \to H_j (D)$ denote projective resolutions of $H_j (D)$ for each $j = 0 ,\dots , d-1$. Then there exists a $\bbz / d \bbz$-graded projective flag resolution of $D$ anchored on $\bigoplus_{j=0}^{d-1} F^{H_j (D)}$. Moreover, this projective flag resolution is obtained as a flag-preserving deformation retract of a Cartan--Eilenberg resolution of $D$. 
\end{theorem}

\begin{remark}
    In the case $d=1$, a direct proof of Theorem \ref{thm:degenToHomology} was originally given by Brown--Erman \cite[Theorem 3.2]{brown2021minimal}. Proving Theorem \ref{thm:degenToHomology} in a functorial way for arbitrary $d$ using similar techniques seems to lead to a much more technical argument; this technicality is what led us to approach the more general statement through the lens of the perturbation lemma.
\end{remark}

\begin{proof}
    As in Construction \ref{cons:StaiRes}, choose projective resolutions $F^{B_i}$ and $F^{H_i}$ of $B_i$ and $H_i$, respectively, for each $i = 0 , \dots , d-1$. Use the notation $G^{B_i}$ to denote the same complex as $F^{B_i}$ but with the differential negated.

    Consider the direct sum of complexes $C := \bigoplus_{j=0}^{d-1} \left( F^{B_j}[-j] \oplus F^{H_j}[-j] \oplus G^{B_{j-1}}[-d-j] \right)$ and equip $C$ with the perturbation $\delta$ that simply maps the complex $G^{B_{j-1}}[-d-j]$ to $F^{B_{j-1}}[-j-d]$ via the identity map. There is a strong deformation retract $C_\delta \to \bigoplus_{j=0}^{d-1} F^{H_j}$ where the homotopy $h$ on $C$ reverses the perturbation $\delta$: it maps $F^{B_{j-1}}[-j-d]$ to $G^{B_{j-1}}[-d-j]$ via the identity map. 

    The Cartan--Eilenberg resolution associated to the data $F^{B_i}$ and $F^{H_i}$ is built as a perturbation of $C_\delta$ by the maps $\alpha^j$, $\beta^j$, and $\gamma^j$ for $j=0 , \dots , d-1$ arising from the Horseshoe lemma as in Construction \ref{cons:StaiRes}; let $\delta'$ denote this perturbation. In order to apply the perturbation lemma, we must check that $\delta' h$ is locally nilpotent. Since $h$ only acts nontrivially on the terms $F^{B_j}[-j]$, we assume $f \in F_i^{B_{j-i}}$. Then $h (f) = f \in G^{B_{j-i}}_i$, and applying $\delta'$ yields the element
    $$\delta' h (f) = \underbrace{\beta^{j-i} (f)}_{\in F_{i-1}^{H_{j-i+1}}} + \underbrace{\gamma^{j-i} (f)}_{\in F_{i-1}^{B_{j-i+1}}}.$$
    Thus $\delta'h$ strictly decreases the flag degree and also preserves the $\bbz / d \bbz$-grading, and since $C$ is bounded below, it follows that $\delta' h$ is locally nilpotent (hence small). By the perturbation lemma (Lemma \ref{lem:thePertLemma1}) there exists a $\bbz / d \bbz$-graded flagged perturbation of $\bigoplus_{j=0}^{d-1} F^{H_j} [-j]$ and a strong deformation retract of a Cartan--Eilenberg resolution onto this perturbation. By definition, any such perturbation is a flagged perturbation with anchor $\bigoplus_{j=0}^{d-1} F^{H_j}$, so we are done.
\end{proof}

\begin{remark}\label{rk:pertLemmaFormulas}
    One of the many advantages of the perturbation lemma is that it gives effective and explicit formulas for the resulting perturbation and perturbed deformation retracts. In the context of Theorem \ref{thm:degenToHomology}, the original deformation retract is simple enough that the perturbed data is easy to write concretely. In order to make the formulas easier to read, we will use the following shorthand: let $\alpha^j$, $\beta^j$, and $\gamma^j$ for $j = 0 , \dots , d-1$ denote the data arising from the Horseshoe resolution as in Construction \ref{cons:StaiRes}. We will employ a similar notation to that used in the proof of Lemma \ref{lem:CELiftingProperty} by defining
    $$F^B := \bigoplus_{j=0}^{d-1} F^{B_j}, \quad F^H := \bigoplus_{j=0}^{d-1} F^{H_j}, \quad G^B := \bigoplus_{j=0}^{d-1} G^{B_j},$$
    where recall that the notation $G^B$ is used to distinguish between the two copies of $F^B$ appearing in the Cartan--Eilenberg resolution.

    Given an integer $\ell \geq 0$, the notation $(\gamma^\bullet)^\ell$ denotes the endomorphism of $G^{B} = \bigoplus_{j=0}^{d-1} G^{B_j}$ defined by
    $$(\gamma^\bullet)^\ell (f) = \gamma^{j-\ell} \circ \gamma^{j-\ell+1} \circ \cdots \circ \gamma^j (f), \quad \text{where} \ f \in F^{B_j}.$$
    Implicit in the above formula is the identification $F^{B_j} = G^{B_j }$ for each $j = 0 ,\dots ,d-1$, as well as the convention that the superscripts of $\gamma$ are taken modulo $d$. It will be understood that the map $(\gamma^\bullet)^\ell$ has source and target given by $F^B$ or $G^B$, and this will be specified in the formulas below. We use a similar notational convention for expressions of the form $(\gamma^\bullet)^\ell \alpha$ and $\beta (\gamma^\bullet)^\ell$. 

    With this notation, the projection map $p_\infty : D \to C$ is given by:
    $$p_{\infty}|_{G^{B_j}} = 0, \quad p_\infty |_{F^{H_j}} = \id_{F^{H_j}}, \quad p_\infty |_{F^{B_j}} = \beta( 1 + (\gamma^\bullet) + (\gamma^\bullet)^2 + \cdots),$$
    where the maps $(\gamma^\bullet)^\ell$ in the above have source $F^B$ and target $G^B$. The inclusion $\iota_\infty : C \to D$ is given by the formula
    $$\iota_\infty  = \id_{F^{H_j}} + \underbrace{(1 + (\gamma^\bullet) + (\gamma^\bullet)^2 + \cdots ) \alpha}_{\text{image in} \ G^B}.$$
    The homotopy $h_\infty$ is given by the formula
    $$h_\infty|_{G^{B_j}} = 0, \quad h_\infty|_{F^{H_j}} = 0, \quad h_\infty|_{F^{B_j}} = \underbrace{1 + (\gamma^\bullet) + (\gamma^\bullet)^2 + \cdots}_{\text{image in} \ G^B},$$
    where in this case we must view $(\gamma^\bullet)^\ell$ as having source $F^{B}$ and target $G^{B}$. The induced perturbed differential is given by the formula
    $$d^{C_\delta} = d^C + \beta ( 1 + (\gamma^\bullet) + (\gamma^\bullet)^2 + \cdots ) \alpha.$$
\end{remark}

Finally, we arrive at the aforementioned result that one can choose anchored resolutions in a functorial way:

\begin{theorem}[Functorial Degeneration to the Homology]\label{thm:functorialDegen}
    Let $\phi : D \to D'$ be any morphism of $\bbz / d \bbz$-graded differential modules. Then there exist anchored resolutions $F \xra{\sim} D$ and $F' \xra{\sim} D'$ and a flag-preserving morphism $\widetilde{\phi} : F \to F'$ such that the following diagram commutes up to homotopy:
\[\begin{tikzcd}
	F && {F'} \\
	\\
	D && {D'}
	\arrow["{\widetilde{\phi}}", from=1-1, to=1-3]
	\arrow["\sim"', from=1-1, to=3-1]
	\arrow["\phi"', from=3-1, to=3-3]
	\arrow["\sim", from=1-3, to=3-3]
\end{tikzcd}\]
\end{theorem}

\begin{proof}
    Theorem \ref{thm:degenToHomology} equips us with strong deformation retracts from Cartan--Eilenberg resolutions $\widetilde{F}$ and $\widetilde{F'}$ of $D$ and $D'$, respectively, to their respective anchored resolutions. By Lemma \ref{lem:CELiftingProperty} there is a morphism of Cartan--Eilenberg resolutions making the diagram 
\[\begin{tikzcd}
	{\widetilde{F}} && {\widetilde{F'}} \\
	\\
	D && {D'}
	\arrow[from=1-1, to=1-3]
	\arrow["\sim"', from=1-1, to=3-1]
	\arrow["\sim", from=1-3, to=3-3]
	\arrow["\phi"', from=3-1, to=3-3]
\end{tikzcd}\]
    commute, whence the result follows from Lemma \ref{lem:PerturbationNaturality}.
\end{proof}

When $d = 0$, Theorem \ref{thm:degenToHomology} shows that the classically defined Cartan--Eilenberg resolutions admit a deformation retract onto a complex built from resolutions of the homology.

\begin{example}[The $\bbz$-Graded Case]\label{ex:ZZgradedCE}
    As mentioned in Remark \ref{rk:sameAsCartanEilenberg}, Construction \ref{cons:StaiRes} recovers the classical notion of a Cartan--Eilenberg resolution of complexes. However, the degeneration to the homology of Theorem \ref{thm:degenToHomology} result gives a deformation retract of the Cartan--Eilenberg resolution onto a complex that is obtained as the totalization\footnote{that is, take the direct sum of the modules appearing along a fixed dotted arrow.} of a diagram of the form
\[\begin{tikzcd}
	& \vdots & \vdots & \vdots & \vdots \\
	\cdots & {F_3^{H_3}} & {F_3^{H_2}} & {F_3^{H_1}} & {F_3^{H_0}} \\
	\cdots & {F_2^{H_3}} & {F_2^{H_2}} & {F_2^{H_1}} & {F_2^{H_0}} \\
	\cdots & {F_1^{H_3}} & {F_1^{H_2}} & {F_1^{H_1}} & {F_1^{H_0}} \\
	\cdots & {F_0^{H_3}} & {F_0^{H_2}} & {F_0^{H_1}} & {F_0^{H_0}}
	\arrow[from=2-5, to=3-5]
	\arrow[from=3-5, to=4-5]
	\arrow[from=4-5, to=5-5]
	\arrow[from=4-4, to=5-4]
	\arrow[from=3-4, to=4-4]
	\arrow[from=2-4, to=3-4]
	\arrow[from=2-3, to=3-3]
	\arrow[from=3-3, to=4-3]
	\arrow[from=4-3, to=5-3]
	\arrow[from=4-2, to=5-2]
	\arrow[from=3-2, to=4-2]
	\arrow[from=2-2, to=3-2]
	\arrow[color={rgb,255:red,221;green,14;blue,28}, from=3-5, to=5-4]
	\arrow[color={rgb,255:red,221;green,14;blue,28}, from=2-4, to=4-3]
	\arrow[color={rgb,255:red,221;green,14;blue,28}, from=3-4, to=5-3]
	\arrow[color={rgb,255:red,221;green,14;blue,28}, from=2-3, to=4-2]
	\arrow[color={rgb,255:red,221;green,14;blue,28}, from=3-3, to=5-2]
	\arrow[color={rgb,255:red,9;green,32;blue,200}, from=2-5, to=5-3]
	\arrow[color={rgb,255:red,221;green,14;blue,28}, from=2-5, to=4-4]
	\arrow[color={rgb,255:red,9;green,32;blue,200}, from=2-4, to=5-2]
	\arrow[color={rgb,255:red,9;green,32;blue,200}, from=2-3, to=5-1]
	\arrow[dashed, no head, from=5-4, to=4-5]
	\arrow[dashed, no head, from=5-3, to=4-4]
	\arrow[dashed, no head, from=4-4, to=3-5]
	\arrow[dashed, no head, from=5-2, to=4-3]
	\arrow[dashed, no head, from=4-3, to=3-4]
	\arrow[dashed, no head, from=3-4, to=2-5]
	\arrow[dashed, no head, from=3-3, to=2-4]
	\arrow[dashed, no head, from=4-2, to=3-3]
	\arrow[dashed, no head, from=3-2, to=2-3]
	\arrow[dashed, no head, from=3-2, to=4-1]
	\arrow[dashed, no head, from=4-2, to=5-1]
	\arrow[dashed, no head, from=3-1, to=2-2]
	\arrow[dashed, no head, from=2-4, to=1-5]
	\arrow[dashed, no head, from=2-3, to=1-4]
	\arrow[dashed, no head, from=2-2, to=1-3]
\end{tikzcd}\]
The non-dotted arrows indicate the components of the differentials induced after employing the perturbation lemma. More precisely, the only nonzero components of the differential restricted to the term $F_i^{B_j}$ map to $\bigoplus_{k \geq 1} F_{i-k}^{B_{j+k-1}}$. 

This complex gives a concrete description of the page maps arising in the hyper(co)homology spectral sequence, since the corresponding higher page maps are precisely induced by the higher components of the differentials induced by the perturbation lemma. As will be proved in Section \ref{sec:anchoredRes's}, if the complexes $F^{H_j}$ are chosen to be minimal free resolutions for each $j \geq 0$, then the above complex is actually \emph{unique} up to isomorphism. This result is surprising since the higher components of the differentials are almost certainly nonminimal in general, so despite the complex itself being generally nonminimal, it exhibits minimal behavior. This construction is likely well-known to experts, though we could not find an explicit reference for this type of resolution of complexes.
\end{example}

\begin{remark}
    The pattern of the differentials in the resolution discussed in Example \ref{ex:ZZgradedCE} is also very similar to the pattern of the differentials constructed in the so-called Eisenbud--Shamash construction \cite{shamash1969poincareseries,eisenbud1980homological} over a complete intersection ring; in the setting of the Eisenbud--Shamash construction, these ``higher" portions of the differentials are induced by sequences of higher homotopies, which are also known to be induced by $A_\infty$ structures \cite{burke2015higher,briggs2023koszul}. 
\end{remark}

\begin{remark}
    Some machinery has already been implemented for the computation of free flag resolutions of ungraded differential modules \cite{banks2024multigraded}. Having such explicit formulas for the resulting deformation retract data as outlined in Remark \ref{rk:pertLemmaFormulas} may allow for more efficient and functorial implementations of these types of resolutions, with the benefit of being able to easily handle additional $\bbz / d \bbz$-gradings.
\end{remark}

\subsection{Inheritance Properties of Flagged Perturbations and their Anchors}\label{subsec:inheritanceProperties}

In this section, we explore some additional consequences of the perturbation that seem to be extremely nontrivial to try to prove directly. The following theorem can intuitively be stated as saying that projective flags inherit homotopy equivalences and deformation retracts coming from the anchor. This further supports the terminology ``anchor" as we see that any kind of property (up to homotopy) of the anchor will also be inherited by the projective flag because of the perturbation lemma.

\begin{theorem}\label{thm:htpyEquivalentAnchors}
    Let $D$ be a projective flag anchored on some complex $F$. 
    \begin{enumerate}
        \item If $G$ is any other complex homotopy equivalent to $F$, then there exists a projective flag $D'$ anchored on $G$ such that $D'$ is homotopy equivalent to $D$.
        \item If, furthermore, $G$ is a deformation retract of $F$, then there exists a projective flag $D'$ anchored on $G$ such that $D'$ is a deformation retract of $D$.
    \end{enumerate}
\end{theorem}

\begin{proof}
    \textbf{Proof of (1):} This is an immediate consequence of Lemma \ref{lem:PertLemma2}.

    \textbf{Proof of (2):} By Remark \ref{rk:modificationForSDR} is of no loss of generality to assume that $G$ is a strong deformation retract of $F$, whence the result follows from the standard perturbation lemma (Lemma \ref{lem:thePertLemma1}).
\end{proof}

\begin{remark}
    Theorem \ref{thm:htpyEquivalentAnchors} is a particularly nice illustration of the power of the perturbation lemma. The process of modifying the deformation retract data in such a way that it becomes a \emph{strong} deformation retract makes the resulting data significantly more complicated, so a method of proof similar to that employed in \cite{brown2021minimal} would be extremely technical and would also need to assume some kind of finiteness of the ambient $R$-module. In our case, no such finiteness is needed.
\end{remark}

As a consequence, we obtain the following statement. This statement will be useful for the part of Theorem \ref{thm:IntroQuasiminimal} that claims that quasiminimal resolutions appear as direct summands of all other anchored resolutions. 

\begin{cor}\label{cor:anchoredIsSummand}
    Assume $(R , \m , \kk)$ is a local ring. Let $D$ be a free flag anchored on some complex $F$. If $G$ is the minimization of $F$, then there exists a free flag $D'$ anchored on $G$ such that $D'$ is a strong deformation retract of $D$ (in particular $D'$ is a direct summand of $D$). 
\end{cor}

The following lemma is well-known as it is actually a special case of ``triangular inversion" in the $A_\infty$-setting \cite[Subsection 3.1]{cirici2018derived}; we state it here for convenience. The statement itself can be proved directly, where the only nontrivial aspect comes from showing that the evident choice of inverse is indeed a well-defined morphism of projective flags. 

\begin{lemma}\label{lem:triangularInversion}
    Let $\phi : D \to D'$ be a morphism in $\flag_{\bbz / d \bbz} (R\proj)$ inducing an isomorphism of the underlying anchors. Then $\phi$ is an isomorphism of flags.
\end{lemma}

\begin{remark}
    One can imagine this result as being analogous to the statement that a power series is invertible if and only if its constant term is a unit. Indeed, these perspectives can be unified using the language of operads \cite[Chapter 10.4.1]{loday2012algebraic}.
\end{remark}


\section{Anchored Resolutions and Flag-Preserving Morphisms}\label{sec:anchoredRes's}

In this section, we introduce anchored resolutions and prove that these objects are indeed unique up to flag-preserving homotopy. As mentioned in the Introduction, this gives a notion of quasiminimal resolutions which are the ``best of both worlds" in the sense that they always exist, are unique, and \emph{also} have a flag structure. This uniqueness up to isomorphism comes from a strong statement about how any morphism of anchored resolutions is flag-preserving up to homotopy (Corollary \ref{cor:flagPresToHtpy}), which in turn is the consequence of a more general fact that realizes every anchored resolution as a flag-preserving deformation retract of a Cartan--Eilenberg resolution. 

\subsection{Anchored Resolutions and Uniqueness up to Homotopy}\label{subsec:anchoredResAndHtpy}

In this subsection, we give the definition of an anchored resolution, state the main theorem of this section (Theorem \ref{thm:mainAnchoredResThm}), and give some further discussion of these notions and how they compare with other constructions in the literature. We postpone the proof of Theorem \ref{thm:mainAnchoredResThm} until Subsection \ref{subsec:proofOfAnchoredThm}, which as previously mentioned relies on a more general result about flag-preserving morphisms.

\begin{definition}
    Let $D$ be a differential $R$-module. Then an \defi{anchored projective flag resolution} (or just an \defi{anchored resolution}) of $D$ is any choice of projective flag resolution that is anchored on a projective resolution of the homology $H(D)$. If $(R , \m , \kk)$ denotes a local ring, then an anchored resolution is \defi{quasiminimal} if the anchor is a minimal free resolution. 
\end{definition}

Since the homology of a $\bbz / d \bbz$-graded differential module splits as a direct sum
$$H(D) = H_0 (D) \oplus H_1 (D) \oplus \cdots \oplus H_{d-1} (D),$$
an anchored resolution is equivalently a $\bbz / d \bbz$-graded flagged perturbation of the direct sum $\bigoplus_{j=0}^{d-1} F^{H_j (D)}$, where $F^{H_j (D)}$ denotes a projective resolution of $H_j (D)$ for each $j = 0 ,\dots , d-1$. Recall that with our conventions, the module $F_i^{H_j (D)}$ has $\bbz / d \bbz$-degree $i+j$.

\begin{theorem}\label{thm:mainAnchoredResThm}
    Let $\phi : D \to D'$ be a morphism of $\bbz / d \bbz$-graded differential modules. Then there exist anchored resolutions $F \xra{\sim} D$ and $F' \xra{\sim} D'$ and a flag-preserving morphism $\widetilde{\phi} : F \to F'$ such that the following diagram commutes up to homotopy:
\[\begin{tikzcd}
	F && {F'} \\
	\\
	D && {D'}
	\arrow["{\widetilde{\phi}}", from=1-1, to=1-3]
	\arrow["\sim"', from=1-1, to=3-1]
	\arrow["\phi"', from=3-1, to=3-3]
	\arrow["\sim", from=1-3, to=3-3]
\end{tikzcd}\]
Moreover, any two anchored resolutions of a fixed differential module $D$ are homotopy equivalent in a flag-preserving way. 
\end{theorem}

\begin{cor}\label{cor:qMinimalUniqueness}
    Assume $(R , \m , \kk)$ is a local ring. Any two quasiminimal resolutions of a differential module are isomorphic by a flag-preserving isomorphism, and every anchored resolution contains the quasiminimal resolution as a flag-preserving direct summand.
\end{cor}

\begin{proof}
    Given two quasiminimal resolutions $F$ and $F'$, Theorem \ref{thm:mainAnchoredResThm} ensures that there exists a flag-preserving quasi-isomorphism $\phi : F \to F'$. Since $\phi$ is flag-preserving, it induces a quasi-isomorphism of the underlying anchors. Since the anchors are assumed to be minimal, this means $\phi$ induces an \emph{isomorphism} of the underlying anchors; by Lemma \ref{lem:triangularInversion} the original morphism $\phi$ is an isomorphism (with flag-preserving inverse). The fact that every anchored resolution contains the quasiminimal anchored resolution as a direct summand is an immediate consequence of Corollary \ref{cor:anchoredIsSummand}. 
\end{proof}

One interesting consequence of Theorem \ref{thm:mainAnchoredResThm} is that the functor
$$D_{\dm_{\bbz / d \bbz}} (R) \to K(\flag_{\bbz / d \bbz})$$
induced\footnote{The category $D_{\dm_{\bbz / d \bbz}} (R)$ denotes the derived category of $\bbz / d \bbz$-graded differential $R$-modules (see \cite{stai2018triangulated}), and $K(\flag_{\bbz / d \bbz} (R))$ denotes the homotopy category of flags (with flag preserving homotopies).} by sending a differential module to any choice of anchored resolution is well-defined and fully-faithful (note that it is not essentially surjective in general).

\begin{remark}
In the work \cite{sagave2010dg}, Sagave develops the theory of \emph{derived} $A_\infty$-algebras, which can be thought of as the analogue of a projective resolution for $A_\infty$-algebras. Such resolutions are interesting since the augmentation maps are instead defined to be $E^2$-equivalances, in the sense that the induced augmentation is a quasi-isomorphism on the second page of the associated spectral sequence. In view of Lemma \ref{lem:flagTauEquivAndAinfty}, we suspect that anchored resolutions should be the shadow of the notion of a \defi{derived $A_\infty$-module}, which should mimic the notion of a derived $A_\infty$-algebra as in \cite{sagave2010dg,cirici2018derived}.
\end{remark}

\subsection{Anchored Resolutions as Deformation Retracts}\label{subsec:anchoredResAsDR}

In this section, we prove that all anchored resolutions may be obtained as deformation retracts of Cartan--Eilenberg resolutions. This follows by constructing a particularly explicit Cartan--Eilenberg resolution of an arbitrary anchored resolution, which in turn follows from understanding the $R$-projective resolutions of the cycles and boundaries of an anchored resolution in terms of the defining data of the original anchor (and its flagged perturbation). Throughout this section, we will write proofs as if there is no auxiliary $\bbz / d \bbz$-grading for purposes of easing notation and making the proofs readable. It is however clear by a similar folding process as used in the proof of Lemma \ref{lem:CELiftingProperty} that these proofs carry over immediately to arbitrary $\bbz / d \bbz$-gradings.

We first construct a complex that will end up being an explicit $R$-projective resolution of the boundaries of an achored resolution:

\begin{construction}\label{cons:BoundariesResolution}
    Let $D$ be a projective flag anchored on a resolution $F^{H(D)}$. Given any integer $i$, let $F_{\geq i}^{H(D)}$ denote the na\"ive truncation of $F^{H(D)}$. Consider the complex $F^{B(D)}$ whose underlying projective $R$-module is given by the direct sum
    $$F^{B(D)} \underbrace{=}_{\text{as} \ R\text{-modules}} \bigoplus_{i \geq 1} F^{H(D)}_{\geq i} [i].$$
    The differential of $F^{B(D)}$ restricted to each summand $F^{H(D)}_{\geq i}[i]$ has nonzero components mapping to the summands $F^{H(D)}_{\geq i-j} [i-j-1]$ with $j > 0$ given by the map $\delta_j^D \in \hom_R^{j+1}  (F^{H(D)} , F^{H(D)})$.

    The complex $F^{B(D)}$ comes equipped with a canonical augmentation map $F^{B(D)} \to B(D)$ induced by mapping the component $D_i = F_i^{H(D)} = \left( F_{\geq i}^{H(D)} [i] \right)_0$ to $B(D)$ via the differential $d^D$. 
\end{construction}

Diagramatically, the structure of the differentials of the complex of Construction \ref{cons:BoundariesResolution} may be visualized as follows:
\[\begin{tikzcd}
	{F_{\geq1}^{H(D)} [1]} && {F_{\geq 2}^{H(D)} [2]} && {F_{\geq3}^{H(D)} [3]} & \cdots \\
	\vdots && \vdots && \vdots \\
	{D_3} && {D_4} && {D_5} & \cdots \\
	\\
	{D_2} && {D_3} && {D_4} & \cdots \\
	\\
	{D_1} && {D_2} && {D_3} & \cdots
	\arrow["{\delta_0^D}"{description}, from=5-1, to=7-1]
	\arrow["{\delta_0^D}"{description}, from=5-3, to=7-3]
	\arrow["{\delta_0^D}"{description}, from=5-5, to=7-5]
	\arrow["{\delta_1^D}"{description}, from=5-3, to=7-1]
	\arrow["{\delta_1^D}"{description}, from=5-5, to=7-3]
	\arrow["{\delta_2^D}"{description, pos=0.3}, curve={height=12pt}, from=5-5, to=7-1]
	\arrow["{\delta_0^D}"{description}, from=3-1, to=5-1]
	\arrow["{\delta_0^D}"{description}, from=3-3, to=5-3]
	\arrow["{\delta_0^D}"{description}, from=3-5, to=5-5]
	\arrow["{\delta_1^D}"{description}, from=3-5, to=5-3]
	\arrow["{\delta_1^D}"{description}, from=3-3, to=5-1]
	\arrow["{\delta_2^D}"{description, pos=0.3}, curve={height=12pt}, from=3-5, to=5-1]
\end{tikzcd}\]

\begin{lemma}\label{lem:BoundariesResolution}
    Let $D$ be a projective flag anchored on a resolution $F^{H(D)}$. Then the complex $F^{B(D)}$ of Construction \ref{cons:BoundariesResolution} is an $R$-projective resolution of the boundaries $B(D)$.
\end{lemma}

\begin{proof}
    Let $B(D^i)$ denote the boundaries of $D^i \subset D$; that is, $B(D^i)$ is the image of $\delta^D|_{D^i}$. As in Construction \ref{cons:BoundariesResolution}, the complex $F^{H(D)}$ will denote the $R$-projective resolution that $D$ is anchored on. For every $i \geq 0$ define the modules $B_i (D)$ via the short exact sequence
    $$0 \to B(D^{i}) \to B(D^{i+1}) \to B_{i} (D) \to 0.$$
    \textbf{Step 1:} Concretely, the module $B_i (D)$ is the image of $\delta^D|_{D_{i+1}}$ modulo the image of $d^D|_{D^i}$, and we claim that an $R$-projective resolution of $B_i (D)$ is given by $F^{H(D)}_{\geq i+1} [i+1]$, equipped with the augmentation map
    $$D_{i+1} = \left( F^{H(D)}_{\geq i +1} [i+1] \right)_0 \xra{d^D|_{D_{i+1}}}  B_i (D).$$
    Notice that by selection the complex $F^{H(D)}$ is a resolution, whence $F^{H(D)}_{\geq i+1} [i+1]$ is also a resolution, so it remains only to show that the $3$-term sequence of maps
    $$(*) \quad D_{i+2} \xra{\delta_0^D} D_{i+1} \xra{d^D} B_i (D) \to 0$$
    is a well-defined sequence, exact at the rightmost and middle spots. To see well-definedness it suffices to show that $d^D \delta_0^D \in B(D^i)$; collecting graded pieces of the equality $(d^D)^2 = 0$ implies that there is an equality
    $$(**) \quad d^D \delta_0^D|_{D_{i+1}} = - \sum_{j =1}^i d^D \delta_j^D|_{D_{i+1}} \in B(D^i).$$
    Thus the sequence $(*)$ is a well-defined complex. The augmentation map $D_{i+1} \to B_i (D)$ is surjective by definition, so it remains to show that $(*)$ is exact at the module $D_{i+1}$. Let $f \in D_{i+1}$; if $d^D (f) \in B(D^i)$, then, in particular $\delta_0^D (f) = 0$. Thus $f \in Z_{i+1} (F^{H(D)}) = B_{i+1} (F^{H(D)})$, and we see that $(*)$ is indeed exact in the middle. 

    \textbf{Step 2:} Next, we claim that the $R$-projective resolution of $B(D^{i+1})$ is given by the complex $F^{B(D^{i+1})}$ with
    $$F^{B(D^{i+1})} \underbrace{=}_{\text{as} \ R\text{-modules}} \bigoplus_{j=1}^{i+1} F^{H(D)}_{\geq j} [j], $$
    with differential induced in an identical manner to Construction \ref{cons:BoundariesResolution}. To see this, proceed by induction on $i \geq 1$ with the base case $i = 1$ being clear since $B(D^1) = B_0 (D)$.

    Assume now that $i > 1$ and consider applying the Horseshoe lemma to the defining short exact sequence
    $$0 \to B (D^i) \to B(D^{i+1}) \to B_i (D) \to 0.$$
    The fact that $F^{B(D^{i+1})}$ is a complex implies that the map $\delta_1^D + \cdots + \delta_i^D : F^{H(D)}_{\geq i+1} [i+1] \to F^{B(D^i)}[-1]$ is a well-defined morphism of complexes

    \textbf{Step 3:} Notice that by definition there is a directed system
    $$B (D^1) \subset B(D^2) \subset B (D^3) \subset \cdots,$$
    and the colimit of this directed system is just the union $\bigcup_{i \geq 1} B(D^i) = B(D)$. There is likewise by construction a directed system of complexes
    $$F^{B(D^1)} \subset F^{B(D^2)} \subset F^{B (D^3)} \subset \cdots,$$
    and since taking directed colimits is an exact functor on the category of $R$-modules, the colimit $\colim_{i} F^{B(D^i)}$ is a resolution $\colim_i B(D^i) = B(D)$. Since the colimit $\colim_{i} F^{B(D^i)}$ is precisely the complex $F^{B(D)}$ of Construction \ref{cons:BoundariesResolution}, the result now follows. 
\end{proof}

Iterating the Horseshoe lemma easily allows us to construct a resolution of the cycles, which in turn will be used as in Construction \ref{cons:StaiRes} to build a particularly convenient Cartan--Eilenberg resolution.

\begin{lemma}\label{lem:CycleResolution}
    Let $D$ be a projective flag anchored on a resolution $F^{H(D)}$. 
    \begin{enumerate}
        \item An $R$-projective resolution $F^{Z(D)}$ of $Z(D)$ may be obtained as the mapping cone of the morphism of complexes
            $$F^{H(D)} \to F^{B(D)} [-1]$$
        whose only nonzero component is given by the natural projection $F^{H(D)} \twoheadrightarrow F^{H(D)}_{\geq 1} \subset F^{B(D)}$.
        \item An $R$-projective resolution of the $R$-module $D$ may be obtained as the mapping cone of the morphism of complexes
        $$F^{B(D)} \to F^{Z(D)} [-1]$$
        whose only nonzero components come from the natural projection maps $F^{H(D)}_{\geq i} [i] \twoheadrightarrow F^{H(D)}_{\geq i+1} [i] \subset F^{Z(D)} [-1]$. 
    \end{enumerate}
\end{lemma}

\begin{proof}
    Both statements $(1)$ and $(2)$ are an immediate consequence of Lemma \ref{lem:BoundariesResolution} combined with the Horseshoe lemma on the short exact sequences
    $$0 \to B(D) \to Z(D) \to H(D) \to 0, \quad \text{and} \quad 0 \to Z(D) \to D \to B(D) \to 0, \quad \text{respectively.}$$
\end{proof}

\begin{cor}\label{cor:filteredCERes}
    Let $D$ be a projective flag anchored on a resolution $F^{H(D)}$. Construction \ref{cons:StaiRes} applied to the resolution of part $(2)$ of Lemma \ref{lem:CycleResolution} yields a Cartan--Eilenberg resolution of $D$.
\end{cor}

The structure of the differentials of the Cartan--Eilenberg resolution may be encoded in the diagram below. In the following, we adopt the convention of Construction \ref{cons:StaiRes} where the notation $G^{B(D)}$ denotes the second copy of the resolution $F^{B(D)}$ (for the sake of distinguishing the two copies). In this diagram, note that the complex $G^{B(D)}$ maps to $F^{H(D)}$ via the components of the original differentials of the flag $D$, and it maps to $F^{B(D)}$ via the canonical projection maps (these are the identity maps jumping two homological degrees in the below diagram). Similarly, the complex $F^{H(D)}$ maps to $F^{B(D)}$ via the natural projection onto the component $F^{H(D)}_{\geq 1}$, seen by the leftmost two complexes in the diagram.

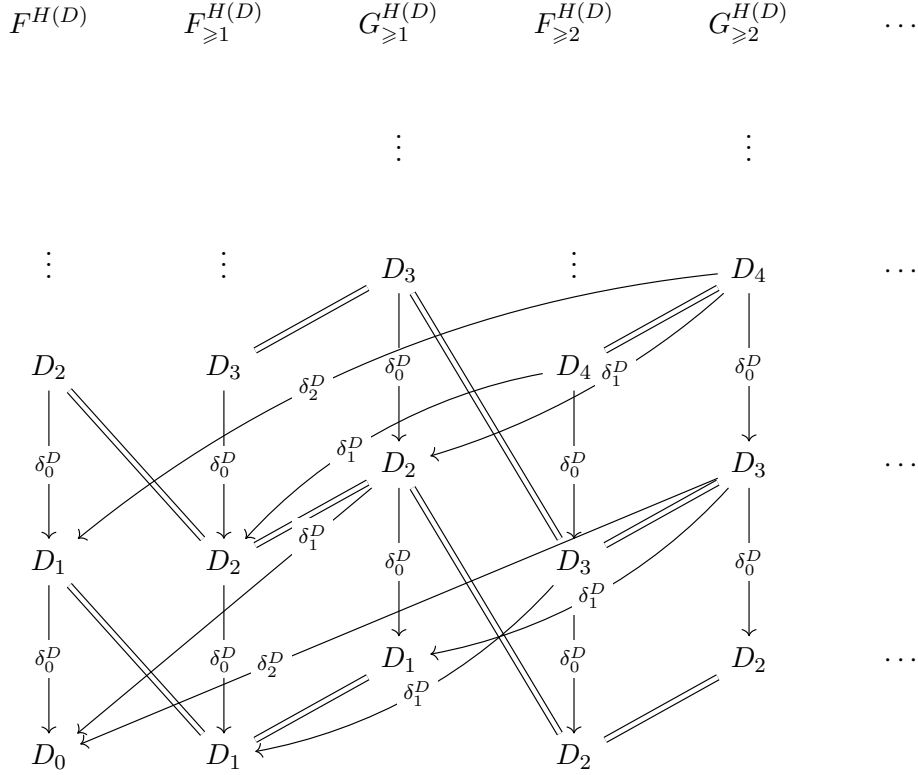
\[\begin{tikzcd}
	{F^{H(D)}} & {F_{\geq 1}^{H(D)}} & {G_{\geq1}^{H(D)}} & {F_{\geq2}^{H(D)}} & {G_{\geq2}^{H(D)}} & \cdots \\
	&& \vdots && \vdots \\
	\vdots & \vdots & {D_3} & \vdots & {D_4} & \cdots \\
	{D_2} & {D_3} && {D_4} \\
	&& {D_2} && {D_3} & \cdots \\
	{D_1} & {D_2} && {D_3} \\
	&& {D_1} && {D_2} & \cdots \\
	{D_0} & {D_1} && {D_2}
	\arrow["{\delta_0^D}"{description}, from=6-1, to=8-1]
	\arrow["{\delta_0^D}"{description}, from=4-1, to=6-1]
	\arrow[Rightarrow, no head, from=6-1, to=8-2]
	\arrow[Rightarrow, no head, from=4-1, to=6-2]
	\arrow["{\delta_0^D}"{description}, from=4-2, to=6-2]
	\arrow["{\delta_0^D}"{description}, from=6-2, to=8-2]
	\arrow[Rightarrow, no head, from=7-3, to=8-2]
	\arrow[Rightarrow, no head, from=5-3, to=6-2]
	\arrow[Rightarrow, no head, from=3-3, to=4-2]
	\arrow[Rightarrow, no head, from=3-5, to=4-4]
	\arrow[Rightarrow, no head, from=5-5, to=6-4]
	\arrow["{\delta_1^D}"{description, pos=0.2}, from=5-3, to=8-1]
	\arrow["{\delta_0^D}"{description}, from=6-4, to=8-4]
	\arrow[Rightarrow, no head, from=7-5, to=8-4]
	\arrow[Rightarrow, no head, from=5-3, to=8-4]
	\arrow[Rightarrow, no head, from=3-3, to=6-4]
	\arrow["{\delta_0^D}"{description}, from=5-3, to=7-3]
	\arrow["{\delta_0^D}"{description}, from=3-3, to=5-3]
	\arrow["{\delta_0^D}"{description}, from=4-4, to=6-4]
	\arrow["{\delta_0^D}"{description}, from=3-5, to=5-5]
	\arrow["{\delta_0^D}"{description}, from=5-5, to=7-5]
	\arrow["{\delta_1^D}"{description, pos=0.6}, curve={height=18pt}, from=4-4, to=6-2]
	\arrow["{\delta_1^D}"{description}, curve={height=-18pt}, from=6-4, to=8-2]
	\arrow["{\delta_1^D}"{description}, curve={height=-18pt}, from=5-5, to=7-3]
	\arrow["{\delta_2^D}"{description, pos=0.7}, from=5-5, to=8-1]
	\arrow["{\delta_2^D}"{description, pos=0.6}, curve={height=30pt}, from=3-5, to=6-1]
	\arrow["{\delta_1^D}"{description, pos=0.4}, curve={height=-12pt}, from=3-5, to=5-3]
\end{tikzcd}\]

The following theorem is the main result of this subsection, and as previously mentioned shows that any anchored resolution appears as a direct summand of a Cartan--Eilenberg resolution. Actually, this additional deformation retract can be identified as coming from the leftmost complex (the resolution $F^{H(D)}$) in the above diagram.

\begin{theorem}\label{thm:anchoredAsDRs}
    Every $\bbz / d \bbz$-graded anchored projective resolution is a flag-preserving deformation retract of a Cartan--Eilenberg resolution.
\end{theorem}

\begin{proof}
    Let $D$ be any projective flag anchored on a resolution $F^{H(D)}$. Then we may construct a Cartan--Eilenberg resolution of $D$ as in Corollary \ref{cor:filteredCERes}. The proof then follows by showing that the perturbation lemma applied as in the proof of Theorem \ref{thm:degenToHomology} yields \emph{precisely} the anchored resolution $D$.

    To see this, we use the explicit formulas outlined in Remark \ref{rk:pertLemmaFormulas}. Firstly, note that in our notation the map $\gamma : F^{B(D)} \to F^{B(D)}[-1]$ restricted to a component $F_i^{H(D)}$ of $F^{B(D)}$ is simply the identity map, sending $F_i^{H(D)}$ to the same module but with the homological degree dropped by $1$. The map $\alpha : F^{H(D)} \to F^{B(D)}$ is the projection $F^{H(D)} \to F_{\geq 1}^{H(D)} \subset F^{B(D)} [-1]$, which restricted to a component $F_i^{H(D)}$ is again just the identity map but viewed as dropping the homological degree by $1$. Finally, the map $\beta : F^{B(D)} \to F^{H(D)} [-1]$ restricted to a component $F^{H(D)}_i$ of $F^{H(D)}_{\geq j}$ is the map $\delta_{j-1}^D : F^{H(D)}_i \to F^{H(D)}_{i-j}$.

    With these formulas now in hand, the Cartan--Eilenberg resolution of $D$ admits a deformation retract onto a perturbation of $F^{H(D)}$ with differential given by the formula
    $$\delta_0^D + \beta (1 + (\gamma^\bullet) + (\gamma^\bullet)^2 + \cdots)\alpha.$$
    Let $f_i \in D_i$; then $\alpha (f_i) = f_i$ but viewed in homological degree $i$ of $G^{B(D)}$. Applying $(\gamma^\bullet)^\ell$ to $f_i$ again yields the element $f_i$, but now viewed in homological degree $i- \ell$ of $F_{\geq \ell}^{H(D)}$. The map $\beta$ restricted to this component is precisely the map $\delta_{\ell-1}^D$; it thus follows that the differential of the perturbation of $F^{H(D)}$ is the sum
    $$\delta_0^D + \delta_1^D + \cdots.$$
    Since this is precisely the differential of the original flag $D$, the result follows. 
\end{proof}

\begin{cor}\label{cor:augmentedIdentity}
     Let $D$ be any projective flag and $\eta : \widetilde{D} \xra{\sim} D$ a Cartan--Eilenberg resolution constructed as in Corollary \ref{cor:filteredCERes}. Let $\iota_\infty: D \hookrightarrow \widetilde{D}$ denote the split inclusion realizing $D$ as a deformation retract of $\widetilde{D}$, as in Theorem \ref{thm:anchoredAsDRs}. Then the composition
    $$D \xra{\iota_\infty} \widetilde{D} \xra{\eta} D$$
    is equal to the identity on $D$. 
\end{cor}

\begin{proof}
    This is an immediate consequence of the formula for $\iota_\infty$ outlined in Remark \ref{rk:pertLemmaFormulas}; in that notation, notice that $\gamma$ is simply the identity map, but viewed as dropping the homological degree by $1$ at each step. The augmentation map $\eta : \widetilde{D} \to D$ is simply the identity map when restricted to each term of the form $G_i^{B(D)}  = D_i$, thus composing these two maps yields the identity $D \xra{\id_D} D$. 
\end{proof}

To end this section, we prove a result that addresses the question of choosing a flag-preserving augmentation map from a Cartan--Eilenberg type resolution to any other differential module equipped with a flag structure. It turns out that with a rather mild degeneration assumption on the associated spectral sequence, this can always be done:

\begin{theorem}\label{thm:flagPresAugmentation}
    Let $D \in \flag_{\bbz / d \bbz} (R)$ be an arbitrary flag and assume that the spectral sequence of Lemma \ref{lem:ZZdSS} degenerates on the $E^r$ page for some $r \geq 2$. Then there exists a projective flag resolution $\widetilde{F}$ of $D$ equipped with a flag-preserving augmentation map $\widetilde{F} \xra{\sim} D$.
\end{theorem}

\begin{proof}
    The proof of this result is similar to the construction of a Cartan--Eilenberg resolution, but with slightly more bookkeeping involved. Using the degeneration of the $E^r$-page of the associated spectral sequence, we can construct an alternative filtration of the homology $H(D)$ which will be better suited for defining a Cartan--Eilenberg type resolution equipped with a flag-preserving augmentation map. Firstly, define
    $$H (D)_{\leq i}^r := Z(D^i) /  \left( B (D^{i+r-1}) \cap Z(D^i) \right) , \quad \text{for all} \ i \geq 0.$$
    We then claim that there are induced inclusions $H(D)_{\leq i}^r \hookrightarrow H(D)_{\leq i+1}^r$ for all $i \geq 0$. Indeed, there is a commutative diagram:
\[\begin{tikzcd}
	& 0 & 0 & 0 \\
	0 & {B(D^{i+r-1}) \cap Z(D^i)} & {Z(D^i)} & {H(D)_{\leq i}^r} & 0 \\
	0 & {B(D^{i+r})\cap Z(D^{i+1})} & {Z(D^{i+1})} & {H(D)_{\leq i+1}^r} & 0 \\
	0 & {B_{i+1}^r (D)} & {Z_{i+1} (D)} & {H(D)_{i+1}^r} & 0 \\
	& 0 & 0 & 0
	\arrow[from=1-2, to=2-2]
	\arrow[from=1-3, to=2-3]
	\arrow[from=1-4, to=2-4]
	\arrow[from=2-1, to=2-2]
	\arrow[from=2-2, to=2-3]
	\arrow[from=2-2, to=3-2]
	\arrow[from=2-3, to=2-4]
	\arrow[from=2-3, to=3-3]
	\arrow[from=2-4, to=2-5]
	\arrow[from=2-4, to=3-4]
	\arrow[from=3-1, to=3-2]
	\arrow[from=3-2, to=3-3]
	\arrow[from=3-2, to=4-2]
	\arrow[from=3-3, to=3-4]
	\arrow[from=3-3, to=4-3]
	\arrow[from=3-4, to=3-5]
	\arrow[from=3-4, to=4-4]
	\arrow[from=4-1, to=4-2]
	\arrow[from=4-2, to=4-3]
	\arrow[from=4-2, to=5-2]
	\arrow[from=4-3, to=4-4]
	\arrow[from=4-3, to=5-3]
	\arrow[from=4-4, to=4-5]
	\arrow[from=4-4, to=5-4]
\end{tikzcd}\]
    The top two rows and leftmost two columns are exact in the above diagram by definition, so a standard diagram chase shows that it is sufficient to prove that the induced map $B_i^r (D) \to Z_i (D)$ is injective for all $i \geq 0$. This however follows from the degeneration at the $r$th page of the associated spectral sequence: firstly, define $Z_i^r (D) := \{ x \in D^i \mid d^D (x) \in D^{i-r} \}$; then by definition of the spectral sequence associated to a filtration there is an equality
    $$E^r_i = \frac{ Z_i^r (D) }{\left( D^{i-1} + B(D^{i+r-1}) \right) \cap Z_i^r (D)  }.$$
    The statement that the spectral sequence degenerates on the $r$th page translates into the statement that for all $i \geq 0$, if $x \in D^{i+r-1}$ is such that $d^D (x) \in D^{i-1}$, then $d^D (x) \in B(D^{i+r-2})$. This is of course precisely the statement needed for the induced map $B^r_i (D) \to Z_i (D)$ to be injective, whence we obtain a directed system of short exact sequences:
\[\begin{tikzcd}
	& \vdots & \vdots & \vdots \\
	0 & {B(D^{i+r-1}) \cap Z(D^i)} & {Z(D^i)} & {H(D)_{\leq i}^r} & 0 \\
	0 & {B(D^{i+r}) \cap Z(D^{i+1})} & {Z(D^{i+1})} & {H(D)_{\leq i+1}^r} & 0 \\
	& \vdots & \vdots & \vdots
	\arrow[from=1-2, to=2-2]
	\arrow[from=1-3, to=2-3]
	\arrow[from=1-4, to=2-4]
	\arrow[from=2-1, to=2-2]
	\arrow[from=2-2, to=2-3]
	\arrow[from=2-2, to=3-2]
	\arrow[from=2-3, to=2-4]
	\arrow[from=2-3, to=3-3]
	\arrow[from=2-4, to=2-5]
	\arrow[from=2-4, to=3-4]
	\arrow[from=3-1, to=3-2]
	\arrow[from=3-2, to=3-3]
	\arrow[from=3-2, to=4-2]
	\arrow[from=3-3, to=3-4]
	\arrow[from=3-3, to=4-3]
	\arrow[from=3-4, to=3-5]
	\arrow[from=3-4, to=4-4]
\end{tikzcd}\]
    where all vertical arrows are inclusions. Taking the colimit thus amounts to taking a nested union, and taking the union of the leftmost terms yields $B(D)$ and $Z(D)$, respectively, whence we deduce
    $$\colim_i H(D)^r_{\leq i} = H(D).$$
    Choose projective resolutions $F^{B_i^r (D)}$ and $F^{H(D)^r_i}$ of $B_i^r (D)$ and $H(D)^r_i$, respectively, for all $i \geq 0$. Applying the Horseshoe lemma it follows that there exist directed systems of resolutions $F^{B(D^{i+r-1}) \cap Z (D^i)}$ and $F^{H(D)_{\leq i}^r}$ with
    $$F^{B(D^{i+r-1}) \cap Z (D^i)} \underbrace{=}_{\text{as} \ R\text{-modules}} \bigoplus_{j =0}^i F^{B_j^r (D)} \quad \text{and} \quad F^{H(D)_{\leq i}^r} \underbrace{=}_{\text{as} \ R\text{-modules}} \bigoplus_{j=0}^i  F^{H(D)^r_j}.$$
    Consider the short exact sequences:
    \begin{equation}\label{eqn:ses1}
         0 \to B(D^{i+r-1}) \cap Z(D^i) \to Z(D^i) \to H(D)_{\leq i}^r \to 0,
    \end{equation}
    \begin{equation}\label{eqn:ses2}
        0 \to B(D^{i+r-1}) \cap Z(D^i) \to B(D^{i+r-1}) \to Q_i^r (D) \to 0
    \end{equation}
    \begin{equation}\label{eqn:ses3}
        0 \to Z(D^i) \to D^i \to B(D^{i-1}) \to 0.
    \end{equation}
    Choosing a directed system of resolutions $F^{Q_i^r (D)}$ (where $Q_i^r (D) := \coker (B(D^{i+r-1}) \cap Z(D^i) \to B(D^{i+r-1}))$) for all $i \geq 0$ and employing the Horseshoe lemma on all of the above short exact sequences, it follows that there exists a directed system of resolutions $F^{D^i}$ of $D^i$ for all $i \geq 0$ with
    $$F^{D^i} \underbrace{=}_{\text{as} \ R\text{-modules}} F^{B(D^{i+r-1}) \cap Z (D^i)} \oplus F^{H(D)^r_{\leq i}} \oplus G^{B(D^{i+r-1}) \cap Z (D^i)} \oplus F^{Q_i^r (D)}.$$
    In the above, we employ the usual convention that $F^{B(D^{i+r-1}) \cap Z (D^i)}$ and $G^{B(D^{i+r-1}) \cap Z (D^i)}$ are the same complex, but we use different notation to distinguish the two copies. Next, consider taking the colimit of system of short exact sequence of \ref{eqn:ses2}. Since the colimits of the leftmost two modules are equal, it follows that $\colim_i Q_i^r (D) = 0$ and thus $\colim_i F^{Q_i^r (D)}$ is a contractible complex. Thus, as $R$-modules the resolution $\colim_i F^{D^i} := F^D$ is homotopy equivalent to a complex with underlying $R$-module $F^{B(D)} \oplus F^{H(D)} \oplus G^{B(D)}$. 

    Define a projective flag $\widetilde{F}$ with flag structure defined via
    $$\widetilde{F}_i := \bigoplus_{j + \ell = i} \left( F_j^{B_\ell^r (D)} \oplus F_j^{H(D)_\ell^r} \oplus G_{j-r}^{B_\ell^r (D)} \right),$$
    and a differential $d^{\widetilde{F}}$ on $\widetilde{F}$ defined in a manner identical to that of Construction \ref{cons:StaiRes}:
    $$d^{\widetilde{F}} := \begin{pmatrix}
        d^{F^{B(D)}} & \alpha & \id_{F^{B(D)}} + \gamma \\
        0 & d^{F^{H(D)}} & \beta \\
        0 & 0 & d^{G^{B(D)}} 
    \end{pmatrix}.$$
    Note that the projective flag $\widetilde{F}$ is precisely a Cartan--Eilenberg resolution of $D$ with the only difference now being that the flag structure has been chosen so that the augmentation map $\widetilde{F} \xra{\sim} D$ of Construction \ref{cons:StaiRes} is flag-preserving. This yields the desired result.

\end{proof}

\begin{remark}
     Notice that the flag resolution obtained in Theorem \ref{thm:flagPresAugmentation} is not necessarily an anchored resolution, and indeed, it is not always possible to construct an anchored resolution that comes equipped with a flag-preserving augmentation map. Moreover, one cannot just apply the perturbation lemma to the projective flag resolution $\widetilde{F}$ na\"ively to obtain a smaller projective flag resolution equipped with a flag-preserving augmentation map. This is because the component $\id_{F_j^{B_\ell^r (D)}} :  G_j^{B_\ell^r (D)} \to F_j^{B_\ell^r (D)}$ drops the filtration degree by $r$, and thus after applying the perturbation lemma we will obtain a differential module $\widetilde{D}$ equipped with a filtration that instead satisfies
     $$d^{\widetilde{D}} (\widetilde{D}^i ) \subset \widetilde{D}^{i+r-2},$$
     and likewise the augmentation map $\eta : \widetilde{D} \to D$ will satisfy $\eta (\widetilde{D}^{i}) \subset D^{i+r-2}$. Thus $\widetilde{D}$ does not necessarily have a flag structure after employing the perturbation lemma. 
\end{remark}

    

\subsection{Proof of Theorem \ref{thm:mainAnchoredResThm}}\label{subsec:proofOfAnchoredThm}

In this section, we prove Theorem \ref{thm:mainAnchoredResThm}. Indeed, many of the ingredients have already been established above, and the proof is just a matter of combining all of these ingredients.

\begin{proof}[Proof of Theorem \ref{thm:mainAnchoredResThm}]
    \textbf{Existence:} The existence of anchored resolutions is ensured by Theorem \ref{thm:degenToHomology}.

    \textbf{Lifting property:} Let $\phi : D \to D'$ be any morphism of differential modules with projective flag resolutions $F \xra{\sim} D$ and $F' \xra{\sim} D'$. By Lemma \ref{lem:naiveLiftingLemma} there exists a (not necessarily flag-preserving) morphism $\psi : F \to F'$ making the following diagram commute up to homotopy: 
\[\begin{tikzcd}
	F && {F'} \\
	\\
	D && {D'}
	\arrow["\psi", from=1-1, to=1-3]
	\arrow["\sim"', from=1-1, to=3-1]
	\arrow["\sim", from=1-3, to=3-3]
	\arrow["\phi", from=3-1, to=3-3]
\end{tikzcd}\]
    Choose Cartan--Eilenberg resolutions $\widetilde{F}$ and $\widetilde{F'}$ of $F$ and $F'$, respectively. Lemma \ref{lem:CELiftingProperty} ensures that there exists a flag-preserving morphism $\widetilde{\psi} : \widetilde{F} \to \widetilde{F'}$ so that the following diagram also commutes up to homotopy:
\[\begin{tikzcd}
	{\widetilde{F}} && {\widetilde{F'}} \\
	\\
	F && {F'} \\
	\\
	D && {D'}
	\arrow["\psi", from=3-1, to=3-3]
	\arrow["\sim"', from=3-1, to=5-1]
	\arrow["\sim", from=3-3, to=5-3]
	\arrow["\phi", from=5-1, to=5-3]
	\arrow[from=1-1, to=3-1]
	\arrow[from=1-3, to=3-3]
	\arrow["{\widetilde{\psi}}"', from=1-1, to=1-3]
\end{tikzcd}\]
    Let $\iota_\infty^F : F \to \widetilde{F}$ and $p_\infty^{F'} : \widetilde{F'} \to F'$ denote the flag-preserving deformation retract data ensured by Theorem \ref{thm:degenToHomology} and define the map 
    $$\widetilde{\phi} := p_\infty^{F'} \circ \widetilde{\psi} \circ \iota_\infty^F : F \to F'.$$
    By construction, this is a flag preserving morphism and a combination of Lemma \ref{lem:HtpyEquivAndCommutingDiagrams} and Corollary \ref{cor:augmentedIdentity} ensures that the following diagram commutes up to homotopy:
\[\begin{tikzcd}
	F && {F'} \\
	\\
	D && {D'}
	\arrow["{\widetilde{\phi}}", from=1-1, to=1-3]
	\arrow["\sim"', from=1-1, to=3-1]
	\arrow["\sim", from=1-3, to=3-3]
	\arrow["\phi", from=3-1, to=3-3]
\end{tikzcd}\]
   Note: a key point in the above is that the augmentation maps are \emph{the same} as the originally given augmentation maps. This is due to the fact that $\iota_\infty^F$ composed with the augmentation map $\widetilde{F} \to F$ is precisely the identity $\id_F$ (likewise for $F'$) by Corollary \ref{cor:augmentedIdentity}. 
    
    \textbf{Uniqueness up to homotopy:} Let $F$ and $F'$ denote two anchored resolutions of a fixed differential module $D$. Choose Cartan--Eilenberg resolutions $\widetilde{F}$ and $\widetilde{F'}$ of $F$ and $F'$, respectively. Corollary \ref{cor:uniqueCEHtpy} ensures that there exists a flag-preserving homotopy equivalence
\[\begin{tikzcd}
	{h, \quad\widetilde{F}} && {\widetilde{F'}, \quad h'}
	\arrow["\phi", curve={height=-18pt}, from=1-1, to=1-3]
	\arrow["\psi", curve={height=-18pt}, from=1-3, to=1-1]
\end{tikzcd}\]
    Since $F$ and $F'$ arise as flag-preserving deformation retracts of $\widetilde{F}$ and $\widetilde{F'}$, respectively, Lemma \ref{lem:HtpyEquivAndCommutingDiagrams} implies that there is an induced flag-preserving homotopy equivalence between $F$ and $F'$. 
\end{proof}

\begin{remark}
    Notice that there is a subtle shift in logic that does not allow one to deduce Theorem \ref{thm:mainAnchoredResThm} from Theorem \ref{thm:functorialDegen}. In Theorem \ref{thm:functorialDegen}, the input data is a fixed morphism of differential modules, and the output is a \emph{choice} of augmentations and a flag-preserving morphism of anchored resolutions making the relevant diagram commute. Theorem \ref{thm:mainAnchoredResThm} requires a much stricter version of Theorem \ref{thm:functorialDegen} to deduce uniqueness up to homotopy: we need to be able to alter \emph{any} morphism of anchored resolutions without changing the associated augmentation maps, which is why realizing any anchored resolution as a flag-preserving deformation retract of a Cartan--Eilenberg resolution was necessary. 
    
    One of the other interesting aspects of Theorem \ref{thm:mainAnchoredResThm} is that this uniqueness up to homotopy is also guaranteed to hold in the category $\flag_{\bbz / d \bbz} (R)$ and $\pflag_{\bbz / d \bbz} (R)$, so there is no risk of picking up ``new" maps as in Example \ref{ex:nonFullExample} when the flag-preserving assumption is dropped.
\end{remark}

We are able to deduce the following corollary from the proof of Theorem \ref{thm:mainAnchoredResThm}.

\begin{cor}\label{cor:flagPresToHtpy}
    Let $\phi : D \to D'$ be any morphism of anchored resolutions. Then $\phi$ is flag-preserving up to homotopy.
\end{cor}

\begin{proof}
    The proof of the lifting property of Theorem \ref{thm:mainAnchoredResThm} combined with Lemma \ref{lem:HtpyEquivAndCommutingDiagrams} imply that the diagram
\[\begin{tikzcd}
	F && {F'} \\
	\\
	F && {F'}
	\arrow["{\widetilde{\phi}}"', from=1-1, to=1-3]
	\arrow[Rightarrow, no head, from=1-1, to=3-1]
	\arrow[Rightarrow, no head, from=1-3, to=3-3]
	\arrow["\psi", from=3-1, to=3-3]
\end{tikzcd}\]
    commutes up to homotopy. In other words, $\psi$ is flag preserving up to homotopy.
\end{proof}

In view of the utility of realizing a projective flag as a deformation retract of a Cartan--Eilenberg resolution, we pose the following question to end this subsection:

\begin{question}
    When does a projective flag arise as a deformation retract of a Cartan--Eilenberg resolution? 
\end{question}

\section{Applications to K-Theory of Linear Factorizations and Rank Conjectures}\label{sec:KtheoryAndRank}

In this section, we finally apply the results on anchored resolutions established throughout this paper to the analogue of the total rank conjecture for free flags. We begin by recalling the definition of a cyclic Adams operation as defined in \cite{brown2017adams,brown2017cyclic}, as well as some other basic facts related to free flags and their homological properties. 

\subsection{Cyclic Adams Operations on $\bbz / 2 \bbz$-graded Differential Modules}\label{subsec:cyclicAdamsOps}

We recall the definition of cyclic Adams operations, as well as establish some more basic properties of free flags. We then introduce the operator $\mathbf{t}^k_\zeta$, which can be viewed as a sort of ``derived" eigenspace operator (similar to one introduced in \cite{brown2017adams}). We then prove that, when restricted to homologically finite differential modules, these operators see no difference between a differential module and its homology after taking Euler characteristic. 

We first begin by establishing some standard notation that we will use throughout this section.

\begin{definition}
    Let $D \in \dm_{\bbz / d \bbz} (R)$ be any $\bbz / d \bbz$-graded differential module with finite length homology. We use the notation $h_i (D) := \ell_R (H_i (D))$. The \defi{Euler characteristic} is defined via
    $$\chi (D) := \sum_{i \in \bbz / d \bbz} (-1)^i h_i (D).$$
    The total homology is defined via
    $$h (D) := \sum_{i \in \bbz / d \bbz} h_i (D).$$
\end{definition}

\begin{lemma}\label{lem:tensorInequality}
    Let $D, D' \in \dm_{\bbz / d \bbz} (R\proj)$ be two projective flags of finite class, where $d \neq 1$. If $D$ has finite length homology, then there is an inequality
    $$h (D \otimes_R D') \leq h(D) \cdot \rank_R (D').$$
    In particular, the tensor product $D \otimes_R D'$ also has finite length homology.
\end{lemma}

\begin{proof}
    Let $r$ denote the class of $D$ and define a flag structure on $D \otimes_R D'$ as follows: given an integer $t \geq 0$, we can uniquely write $t = i + jr$ for $0 \leq i \leq r-1$. Then for any $\ell \in \bbz / d \bbz$ define
    $$(D \otimes_R D')_{t,\ell} := \bigoplus_{s = 0}^{d-1} D_{i,s} \otimes_R D'_{j,\ell-s}.$$
    The grading here is chosen to ``prioritize" the first tensor factor $D$ over $D'$. Moreover, by the class assumption on $D$ combined with the spectral sequence of Lemma \ref{lem:ZZdSS} associated to this filtration, there is a spectral sequence
    $$E^{r+1}_{i + jd , \ell} = \bigoplus_{s=0}^{d-1} H_{i,s} (D) \otimes_R D'_{j,\ell-s} \implies H_\ell (D \otimes_R D').$$
    Thus $H (D \otimes_R D')$ is a subquotient of a finite length $R$-module, and taking $R$-module lengths yields the desired inequality.
\end{proof}

\begin{remark}\label{rk:tensorDiffs}
    Note that even though the $r$th page of the above spectral sequence consists of direct sums of terms of the form $H_{i,s} (D) \otimes_R D'_{j,\ell-s}$, the differentials are not literally of the form $\id_{H_{i,s} (D)} \otimes d^{D'}$.

    This is prominent in the $\bbz / 2 \bbz$-perturbation of the Koszul complex given in Example \ref{ex:perturbedKoszul}. In a little more detail, view $K_\bullet$ as the exterior algebra on a free $R$-module with basis $e_1 , e_2 , e_3$, where the differential is induced by sending $e_i \mapsto x_i$ for $i=1 , 2,3$. At some point in the associated spectral sequence, we must compute the homology of a complex of the form
\[\begin{tikzcd}
	{\kk \otimes_R K_3} & {\kk \otimes_R K_2} & {\kk \otimes_R K_1} & {\kk \otimes_R K_0}
	\arrow["\sigma", curve={height=-24pt}, from=1-1, to=1-4]
\end{tikzcd}\]
    If the map $\sigma$ were merely induced by $\kk \otimes \delta$ (where $\delta$ is the original perturbation of $K_\bullet$) then we would find that there is an equality $h(K_\delta \otimes_R K_\delta) = 6$, which would be a contradiction to the proof of Theorem \ref{thm:ZZ2TRC} proved later. Instead, this map is induced by the differential of $K_\delta \otimes_R K_\delta$ applied to the element
    $$e_1 \w e_2 \w e_3 \otimes 1 - \Delta_{2,1} (e_1 \w e_2 \w e_3) + \Delta_{1,2} (e_1 \w e_2 \w e_3) - 1 \otimes e_1 \w e_2 \w e_3 \in (K_\delta \otimes_R K_\delta)_3,$$
    where $\Delta_{i,j} : \bigwedge^{i+j} \to \bigwedge^i \otimes \bigwedge^j$ is a homogeneous component of the exterior algebra comultiplication. The induced page differential would be the image of the component of the above element mapping to $K_0 \otimes_R K_0$, which is precisely 
    $$\delta (e_1 \w e_2 \w e_3) \otimes 1 - 1 \otimes \delta( e_1 \w e_2 \w e_3) = 1 \otimes 1 - 1 \otimes 1 = 0.$$
    Thus $\sigma = 0$ and there is actually an equality $h (K_\delta \otimes K_\delta) = h(K_\delta) \rank_R (K_\delta) = 8$.

    The reason we mention this is that, classically, the equality $h (D \otimes_R D') = h(D) \rank_R D'$ for \emph{complexes} of free modules implies that the homology $H(D)$ must annihilate all entries of the differentials of $D'$. This detail is a key fact used to prove the ``extended" formulation of the Total Rank Conjecture \cite[Conjecture 1.11]{vandebogert2023total}, which ensures that the smallest possible resolution with finite length homology is in fact the Koszul complex on a system of parameters. In the $\bbz / d \bbz$-graded setting (with $d > 1$) the most we can say is that the equality $h(D \otimes_R D') = h(D) \rank_R D'$ implies that the homology $H(D)$ annihilates the differentials of the \emph{anchor} of $D'$.
\end{remark}

\begin{example}\label{ex:failureOfIneq}
    Continuing with the perturbation of the Koszul complex as in Example \ref{ex:perturbedKoszul} (and Remark \ref{rk:tensorDiffs} above), we can also see that Lemma \ref{lem:tensorInequality} fails even when $D$ and $D'$ are assumed to be deformation retracts of free flags. It is straightforward to verify that the perturbation $K_\delta$ admits a deformation retract onto the $\bbz / 2 \bbz$-graded differential module $D$ with
    $$D_1 = R^{\oplus 3} \xra{\begin{pmatrix}
        -x_{2}&-x_{3}&0\\
      x_{1}&0&-x_{3}\\
      0&x_{1}&x_{2}
    \end{pmatrix}} D_0 = R^{\oplus 3} \xra{\begin{pmatrix}
        x_{1}x_{3}&x_{2}x_{3}&x_{3}^{2}\\
      -x_{1}x_{2}&-x_{2}^{2}&-x_{2}x_{3}\\
      x_{1}^{2}&x_{1}x_{2}&x_{1}x_{3}
    \end{pmatrix}} D_1 $$
    Then it is a consequence of the discussion in Remark \ref{rk:tensorDiffs} that $h (D \otimes_R D) = 8$, whereas $h (D) \rank_R (D) = 6$, so the inequality of Lemma \ref{lem:tensorInequality} fails even for well-behaved summands of free flags (these homology computations can also be verified directly using the package \cite{favero2023factorizations} in Macaulay2 \cite{grayson1997macaulay}).
\end{example}




Now we recall the definition of cyclic Adams operations for $\bbz / 2 \bbz$-graded differential modules (that is, linear factorizations of $0$). For full details, see \cite{brown2017adams}:

\begin{definition}
    Let $p > 0$ be any prime. Let $P$ be a $\bbz / 2 \bbz$-graded $R$-complex, where $R$ is any ring such that $p$ is invertible and $R$ contains a primitive $p$th root of unity, denoted $\zeta$. The symmetric group $\Sigma_p$ acts on the $p$th tensor power $T^p (P) := P^{\otimes p}$ by graded commutators:
    $$(i \ i+1 ) \cdot (x_1 \otimes \cdots \otimes x_i \otimes x_{i+1} \otimes \cdots \otimes x_p) = (-1)^{|x_i| \cdot |x_{i+1}|} x_1 \otimes \cdots \otimes x_{i+1} \otimes x_i \otimes \cdots \otimes x_p.$$
    There is an embedding of the cyclic group $C_p \cong \bbz / p \bbz \subset \Sigma_p$ sending $1 \mapsto \sigma :=  (1 \ 2 \ \dots \ p)$; this action induces a well-defined endomorphism of complexes
    $$\sigma : T^p (P) \to T^p (P),$$
    in which case we may consider the $\zeta^i$ eigenspaces of $T^p (P)$ for each $0 \leq i \leq p-1$, denoted $T^p (P)^{(\zeta^i)}$. The $p$th \defi{cyclic Adams operator} $\psi^p_{cyc}$ is the map
    $$\psi^p_{cyc} : K_0 (R) \to K_0 (R), \quad \text{defined by}$$
    $$\psi^p_{cyc} ([P]) := [T^p (P)^{(1)}] - [T^p (P)^{(\zeta)}].$$
\end{definition}

To define the cyclic Adams operations $\psi^k_{cyc}$ for composite integers $k$, one writes $k = p_1^{e_1} \cdot p_2^{e_2} \cdots p_\ell^{e_\ell}$ as a product of prime numbers and sets
$$\psi^k_{cyc} := \left( \psi^{p_1}_{cyc} \right)^{e_1} \circ \left( \psi^{p_2}_{cyc} \right)^{e_2} \circ \cdots \circ \left( \psi^{p_\ell}_{cyc} \right)^{e_\ell}.$$
This is well-defined by \cite[Section 4]{brown2017cyclic}.

In the work \cite{brown2017cyclic}, the authors show that the cyclic Adams operations satisfy the ``Gillet--Soul\'e axioms" for Adams operations, even on categories of linear factorizations. Although their definition of a cyclic Adams operation requires the existence of roots of unity/invertibility of certain elements, it has the significant advantage that one can easily extend these Adams operations to act on more general objects than just chain complexes. Moreover, the requirement that roots of unity exist will be of no real significance for our purposes, since this may always be assumed up to taking some faithfully flat extension.

In the following theorem, the notation $\psi^k$ denotes the Adams operation of Gillet--Soul\'e \cite{gillet1987intersection} and the notation $K_0 (R)$ denotes the Grothendieck group of chain complexes of free $R$-modules.

\begin{theorem}[{\cite[Corollary 6.14]{brown2017adams}}]\label{thm:cyclicVsGS}
    Let $k \geq 1$ be any integer and assume $R$ is a local Noetherian ring containing all primitive $t$-th roots of unity for $t \leq k$, with $k!$ invertible. Then there is an equality of operators on $K_0 (R)$:
    $$\psi^k_{cyc} = \psi^k.$$
\end{theorem}

The above theorem will be essential for the proofs in Subsection \ref{subsec:DuttaMultsandEuler}. This is because, even though we know how to take cyclic Adams operations of $\bbz / 2 \bbz$-graded differential modules, we then need to be able to understand how these operators interact after taking Euler characteristic/Dutta multiplicities. The cases in which these interactions are best understood are for the Gillet--Soul\'e version of the Adams operation, whence Theorem \ref{thm:cyclicVsGS} will be our tool for freely interchanging between the two notions of Adams operation.

Next, we must introduce a ``derived" version of the eigenspace operator $P \mapsto T^p (P)^{(\zeta)}$ for some root of unity $\zeta$. A very similar operator was introduced in \cite{brown2017adams} when the ambient ring was assumed to be regular --- this assumption bypassed any of the subtleties related to the difference between quasi-isomorphism versus homotopy equivalence for linear factorizations over arbitrary local rings. However, now that the notion of an anchored resolution gives us a particularly well-behaved class of free flag resolutions, we can extend the definition to arbitrary rings using this machinery.

\begin{definition}
    Let $k \geq 1$ be any integer and assume $R$ is a local Noetherian ring containing a primitive $k$th root of unity, with $k$ invertible. Given any (not necessarily primitive) $k$th root of unity $\zeta$, let $\mathbf{t}^k_\zeta : D_{\dm}^{\bbz / 2 \bbz} (R) \to K(\flag_{\bbz / d \bbz} (R\proj))$ denote the functor defined as follows:
    $$\mathbf{t}^k_\zeta (D) := T^k (P)^{(\zeta)},$$
    where $P \xra{\sim} D$ is any anchored free resolution of $D$. 
\end{definition}

Note that the uniqueness of anchored resolutions up to homotopy combined with the lifting properties of Theorem \ref{thm:mainAnchoredResThm} guarantee that the above definition of $\mathbf{t}^k_\zeta$ is indeed well-defined and does not depend on the quasi-isomorphism class of the underlying differential module. 

\begin{lemma}\label{lem:derivedEigenAndChi}
     Let $D$ be any $\bbz / 2 \bbz$-graded differential module with finite length homology. Then there is an equality
    $$\chi \left(  \mathbf{t}^k_\zeta (D) \right) = \chi \left( \mathbf{t}^k_\zeta (H(D)) \right). $$
\end{lemma}

\begin{proof}
     Let $P \xra{\sim} D$ denote a free flag resolution of $P$ anchored on a finite free resolution $F^{H(D)}$ of $H(D)$. Then the complex $T^k (P)^{(\zeta)}$ is anchored on the complex $T^p (F^{H(D)})^{(\zeta)}$, and the spectral sequence of Lemma \ref{lem:ZZdSS} ensures that 
    $$\chi (T^p (P)^{(\zeta)}) = \chi (T^p (F^{H(D)})^{(\zeta)}).$$
    By definition, $\chi (T^p (F^{H(D)})^{(\zeta)}) = \chi \left( \mathbf{t}^p_\zeta (H(D)) \right)$. 
\end{proof}

\subsection{Interactions with Dutta Multiplicities and Euler Characteristic}\label{subsec:DuttaMultsandEuler}

We first recall/observe some properties of extension of scalars along the Frobenius map when applied
to free flags with finite length homology. In particular, we recall the definition of ``Dutta multiplicity", a notion typically reserved for chain complexes but, as we will see, can be extended to $\bbz / 2 \bbz$-graded differential modules without much difficulty.

Given a free flag of $R$-modules $P$, we write
$F(P)$ for the flag of $R$-modules obtained from $P$ by extension of scalars along the Frobenius endomorphism $r \mapsto r^{p}$ of $R$. 
For an integer $e \geq 0$, we write $F^e(P)$ for the $e$-th iterate of this functor.
Concretely, if bases are chosen for each component of $P$, so that the defining endomorphism $d^P$ is represented by a matrix, then $F^e(P)$ is isomorphic to the free flag
whose component free modules are the same as those of $P$, but with each entry of $d^P$ replaced by its $p^e$ power.

We first need to show that a certain class of limits exists for homologically finite free flags with finite length homology:

\begin{prop}\label{prop:limitsexist} Assume $(R, \m, k)$ is a complete local ring of characteristic $p > 0$ that has a perfect residue field $\kk = R/\fm$. Given a homologically finite $P \in \flag_{\bbz / d \bbz}(R\proj)$ with finite length homology and an integer $i \in \bbz / d \bbz$, the limit
  $$\lim_{e \to \infty} \frac{h_i(F^eP)}{p^{e \cdot \dim R}}$$
  exists.
\end{prop}

\begin{proof}
    Choose any bounded anchored resolution $\widetilde{P}$ of $P$. The flag $P$ is in fact homotopy equivalent to $\widetilde{P}$ and thus it is of no loss of generality to assume that $P$ is itself an anchored resolution. By Lemma \ref{lem:ZZdSS} there is a spectral sequence
    $$E^2_{r,s} = H_r (F^e P_{\bullet , s}) \implies H_{r+s} (F^e P).$$
    Since $D_{\bullet , s}$ is a bounded complex of free modules with finite length homology for any $s = 0 , \dots , d-1$, each of the limits $h_i^\infty (D_{\bullet,j})$ exists (by \cite[Theorem 7.3.3]{RobertsBook}). Since the homology of $F^e P$ has a filtration whose associated graded pieces are subquotients of $F^e$ applied to the anchor, again the result of Roberts implies that the limit $\lim_{e \to \infty} \frac{h_i(F^eP)}{p^{de}}$ exists.
\end{proof}

\begin{definition}

With $R$ as in Proposition \ref{prop:limitsexist} and any $P \in \flag_{\bbz / d \bbz}(R\proj)$ with finite length homology, define 
  $$
  h_i^\infty(P) = \lim_{e \to \infty} \frac{h_i(F^eP)}{p^{de}} \quad \text{and} \quad
  h^\infty(P) = \sum_i h_i^\infty(P). 
  $$
The 
 \defi{Dutta multiplicity} of $P$ is
$$
\chi_\infty(P) = \sum_i (-1)^i h^\infty_i(P) = \lim_{e \to \infty} \frac{\chi(F^e P)}{p^{de}}.
$$
\end{definition}

\begin{cor}\label{cor:derivedEigenAndDutta}
    Let $R$ be any complete local ring of positive characteristic with perfect residue field. Assume $k \geq 1$ is a unit and $R$ contains a primitive $k$th root of unity. Then for any homologically finite differential module $D$, there is an equality
    $$\chi_\infty \left(  \mathbf{t}^k_\zeta (D) \right) = \chi_\infty \left( \mathbf{t}^k_\zeta (H(D)) \right). $$ 
\end{cor}

\begin{proof}
    The proof is identical to that of Lemma \ref{lem:derivedEigenAndChi} where one instead applies the $e$th iterate of Frobenius to the spectral sequence of Lemma \ref{lem:ZZdSS} and takes limits to deduce the desired equality.
\end{proof}

Recall in the following corollary that a ring $R$ is \defi{quasi-Roberts} if there is an equality of K-theoretic operators $\chi \circ \psi^k = k^{\dim R} \chi$, where $\psi^k$ denotes the $k$th Adams operator of Gillet--Soul\'e \cite{gillet1987intersection}. Any ring that is locally a complete intersection is quasi-Roberts; work of Miller--Singh \cite{miller2000intersection} constructs an example of a non quasi-Roberts Gorenstein ring.

\begin{cor}\label{cor:AdamsAndPowersEqualities}
     Let $R$ be a local ring and $k >1$ any integer so that $R$ contains all $t$-th roots of unity, for $t \leq k$ and $k!$ is a unit. Assume $D$ is any bounded, homologically finite free flag with finite length homology.

    \begin{enumerate}
        \item Assume that $R$ is a quasi-Roberts ring. Then there is an equality
    $$\chi (\psi^k_{cyc} (D)) = k^{\dim R} \chi (D).$$
        \item Assume that $R$ is a complete local ring of positive characteristic with perfect residue field. Then there is an equality
    $$\chi_\infty (\psi^k_{cyc} (D)) = k^{\dim R} \chi_\infty (D).$$
    \end{enumerate}
\end{cor}

\begin{remark}
     For many K-theoretic purposes, the existence of anchored resolutions implies that the category of differential modules is largely indistinguishable from the category of complexes. A shadow of this fact was used in the work \cite{brown2017adams}, where the authors proved the equality of Corollary \ref{cor:AdamsAndPowersEqualities}$(1)$ holds for \emph{regular} local rings. The regularity assumption ensured that the boundaries of a linear factorization will also have finite projective dimension. The advantage of having Theorem \ref{thm:degenToHomology} combined with the uniqueness of anchored resolutions is that one now only needs the \emph{homology} to have finite projective dimension to prove similar results. 
\end{remark}

\begin{remark}
    If we instead assume that $R$ is a regular local ring, then it suffices to assume that $k$ itself is a unit in $R$. This follows from the fact that $\psi_{cyc}^k ([K]) = k^d [K]$ as proved in \cite{brown2017cyclic} only under the assumption that $k$ is a unit, where $K$ denotes the K-class of the Koszul complex on a system of parameters for $R$.
\end{remark}

\begin{proof}
    \textbf{Proof of (1):} By definition of the cyclic Adams operation, there is an equality
    $$\chi \circ \psi^k_{cyc} (D) =  \sum_\zeta \zeta \cdot \chi  \left(\mathbf{t}_\zeta^k (D) \right),$$
    and by Lemma \ref{lem:derivedEigenAndChi} we have $\chi  \left(\mathbf{t}_\zeta^k (D) \right) = \chi  \left(\mathbf{t}_\zeta^k (H(D)) \right)$ for any $k$th root of unity $\zeta$. Choosing a finite minimal free resolution $P$ of $H(D)$, we thus find that
    $$\chi \circ \psi^k_{cyc} (D) = \chi \circ \psi^k_{cyc} (P),$$
    which, combined with the assumption on $R$ and the fact that $\psi^k_{cyc} = \psi^k$ by Theorem \ref{thm:cyclicVsGS}, implies the equality
    $$ \chi \circ \psi^k_{cyc} (D) = \chi \circ \psi^k (P) = k^{\dim R} \chi (P).$$
    The result thus follows after noting that $\chi (P) = \chi(D)$.

    \textbf{Proof of (2):} The proof in this case is almost identical, except that the Euler characteristic is replaced with the Dutta multiplicity. Recall that by definition there is an equality
    $$\chi_\infty \circ \psi^k_{cyc} (D) =  \sum_\zeta \zeta \cdot \chi_\infty \left(\mathbf{t}_\zeta^k (D) \right),$$
    and by an identical argument to the proof of $(1)$, Corollary \ref{cor:derivedEigenAndDutta} yields the equality
    $$\chi_\infty \circ \psi^k_{cyc} = \chi_\infty \circ \psi^k (P), $$
    where $P \xra{\sim} H(D)$ is any finite free resolution of $H(D)$. Thus the statement reduces to the fact that $\chi_\infty \circ \psi^k = k^{\dim R} \chi_\infty$ on the category of perfect complexes with finite length homology, which is a well-known result of Kurano--Roberts \cite[3.1]{KRAdamsOps}.
\end{proof}



\subsection{The Avramov--Buchweitz--Iyengar Total Rank Conjecture}\label{subsec:TRCforFlags}

In this section, we combine all of the results established thus far to prove a version of the Total Rank Conjecture over the category of $\bbz / 2 \bbz$-graded differential modules. We conclude with some further discussion and questions about methods of relaxing hypotheses, as well as a question about the validity of the ``extended" total rank conjecture for differential modules (see \cite[Conjecture 1.11]{vandebogert2023total}).

\begin{theorem}\label{thm:ZZ2TRC}
    \begin{enumerate}
        \item Assume $R$ is a quasi-Roberts ring in which $2$ is a unit. For any homologically finite $\bbz / 2 \bbz$-graded free flag $D$ having finite length homology, there is an inequality
    $$ \rank_R D \geq 2^{\dim R} \frac{|\chi (D)|}{h(D)}.$$
        \item Assume $R$ is any local ring of positive characteristic $p > 2$. For any homologically finite $\bbz / 2 \bbz$-graded free flag having finite length homology, there is an inequality
    $$ \rank_R D \geq 2^{\dim R} \frac{|\chi_\infty (D)|}{h^\infty (D)}.$$
    \end{enumerate}
\end{theorem}

\begin{proof}
    The proofs of $(1)$ and $(2)$ are essentially verbatim versions of the original arguments given by Walker \cite{walker2017total}, and most of the subtlety comes from the relevant K-theoretic identities proved in the previous subsections.
    
    \textbf{Proof of (1):} We first note that it is of no loss of generality to assume that $R$ contains all $t$-th roots of unity for $t \leq k$. This is because if $R$ does not contain a $t$-th root of unity, there is a faihfully flat extension
    $$R \hookrightarrow R' :=  \left( \frac{R[x]}{(x^t-1)} \right)_{\m'}$$
    where $\m'$ is any maximal ideal lying over $\m$, and for every finite length $R$-module $M$ an equality
    $$\ell_{R'} (M \otimes_R R') = [R' / \m' : R / \m ] \cdot \ell_R (M).$$ 
    Thus in the following string of inequalities, using the ring $R'$ simply results in a factor of the degree of the field extension $[R' / \m' : R / \m ]$ which can be divided in the end. Furthermore, by shifting if necessary we may assume that $\chi (F) \geq 0$:
    \begingroup\allowdisplaybreaks
    \begin{align*}
        h(D) \rank_R (D) \underbrace{\geq}_{\text{Lemma \ref{lem:tensorInequality}}} & h(D \otimes_R D) \\
        \underbrace{=}_{2 \ \text{is a unit}}& h(t^2_{1} (D)) + h (t^2_{-1} (D)) \\
        \geq \qquad &\chi(t^2_{1} (D)) + \chi (t^2_{-1} (D)) \\
        = \qquad &\chi \circ \psi_{cyc}^2 (D) \\
        \underbrace{=}_{\text{Corollary \ref{cor:AdamsAndPowersEqualities}}}& 2^{\dim R} \chi (D).
    \end{align*}
    \endgroup

    \textbf{Proof of (2):} As in the proof of $(1)$ it is no loss of generality to assume that $R$ contains all $t$-th roots of unity for $t \leq k$, and it is furthermore of no loss of generality to assume that $R$ is complete and the residue field $R / \m$ is perfect (as in \cite[Proof of Theorem 2]{walker2017total}). Letting $F^e (D)$ denote the $e$th Frobenius iterate of $D$, an identical sequence of inequalities to part $(2)$ yields the inequalities
    $$h(F^e D) \underbrace{\rank_R (F^e (D))}_{= \rank_R (D)} \geq  |\chi \circ \psi^2_{cyc} (F^e (D))| \quad \text{for all} \ e \geq 0.$$
    Dividing the above inequality by $p^{e \cdot \dim R}$, taking the limit as $e \to \infty$, and noting that the Frobenius iterate commutes with $\psi^2_{cyc}$ yields the (in)equalities
    $$h^\infty (D) \rank_R (D) \geq |\chi_\infty \circ \psi^2_{cyc} (D)| \underbrace{=}_{\text{Corollary \ref{cor:AdamsAndPowersEqualities}}} 2^{\dim R} |\chi_\infty (D)|.$$
\end{proof}

Note that the proof of Theorem \ref{thm:TRCIntro} is an immediate consequence of Theorem \ref{thm:ZZ2TRC}. This follows because if $D$ does not have finite length homology, we can localize at any minimal prime over the annihilator to obtain a differential module with finite length homology over a ring of dimension $= \codim_R \ann_R H(D)$. The property of having homology concentrated in degree $0$ or $1$ is evidently preserved by localization (as well as the rank), so the result follows.

\begin{remark}
      In the work \cite{walker2017total,vandebogert2023total}, an ``extended" version of the Total Rank Conjecture is proved over quasi-Roberts rings that shows that in the case of equality, one must have that the original complex is isomorphic to the Koszul complex on a system of parameters. As indicated in Remark \ref{rk:tensorDiffs}, any equality of the form $h (D) \rank_R (D) = h (D \otimes_R D)$ only indicates that $h (D)$ annihilates the differentials of the anchor of $D$. Similarly, knowing that $h (t^2_{-1} (D)) = 0$ does not necessarily imply that $\bigwedge^2 H(D) = 0$, so we also cannot immediately say that $H(D)$ is a cyclic module. So we still do not know if the extended version of the Total Rank Conjecture holds for free flags, even in the case that $R$ is quasi-Roberts.
\end{remark}

It is also likely that the assumption that $D$ is homologically finite is unnecessary. As noted in \cite[Example 5.12]{brown2021minimal}, there exist differential modules with finite class that are \emph{not} homologically finite. Thus, we pose the following question:

\begin{question}
    Is the assumption that $\pd_R H(D) < \infty$ necessary in Theorem \ref{thm:ZZ2TRC}?
\end{question}

\begin{remark}
     One interpretation of the classical total rank conjecture is that finite length $R$-modules are ``homologically big". Theorem \ref{thm:ZZ2TRC} says that not only do complexes with finite length homology tend to be homologically big, but any complex admitting a flagged perturbation with finite length homology will also tend to be homologically big.
\end{remark}

\begin{remark}
    In characteristic $2$, the classical total rank conjecture is known to be true (for equicharacteristic rings)\footnote{In fact, a strictly stronger version of the total rank conjecture holds in characteristic $2$ \cite{vandebogert2023total} (that is known to be false in odd characteristic)}, but the techniques used to prove this heavily depend on the classical Dold-Kan correspondence \cite{dold1958homology,kan1955abstract,lurie2016higher}. One potential avenue for extending these techniques to the setting of linear factorizations is to formulate the notion of a flagged perturbation of a simplicial $R$-module; this should be done in a way that the appropriate analogue of the normalization functor induces a flagged perturbation of a nonnegative graded complex. Such a correspondence would allow for a more canonical K-theoretic study of categories of linear factorizations that does not depend on invertibility of elements/existence of roots of unity. 
\end{remark}

\appendix
\section{Differential Modules, Flags, and Diagram Categories}\label{sec:DMsFlagsDcats}

In this appendix, we recall and formalize a few different categories, including the categories of $\bbz / d \bbz$-graded differential modules and flags. We also introduce a ``twisted" version of the category of flags, where many of the definitions are essentially identical to the definitions for flags, but twisted by additional signs. This twisted category $\flag^\tau (R)$ will end up being equivalent to the category of $A_\infty$ $R[t]/(t^2)$-modules, where $t$ is an indeterminate of homological degree $2$. Finally, in Subsection \ref{subsec:HarrowCategory} we recall the category of homotopy commutative diagrams (of differential modules). We prove a variety of compatibility theorems between homotopy commutative diagrams and different homotopy equivalent representatives of objects, which we will want to reference in earlier sections.

Though some of the results in this appendix may be considered ``new" (for instance, the equivalence with $A_\infty$-modules of Lemma \ref{lem:flagTauEquivAndAinfty}), some of these results are likely well-known to experts or are tacitly used/observed without proof in other places.

\subsection{The Category $\dm_{\bbz / d \bbz} (R)$}\label{subsec:theCategoryDM}

The purpose of this subsection is to recall the category of $\bbz / d \bbz$-graded differential modules. The category of differential modules (without any additional grading) has been introduced in several other places -- for instance \cite{cartan2016homological,avramov2007class,stai2018triangulated,brown2021minimal}. There are some small differences when taking into account additional auxiliary gradings that are worth specifying, including the monoidal/self-enriched structure when $d \neq  1$ (see Observation \ref{obs:MonoidalStructure}).   

\begin{definition}[$\bbz / d \bbz$-Graded Differential Modules]
    Let $d \geq 1$ be any integer. A \defi{$\bbz / d \bbz$-graded differential $R$-module} $D$ is an $R$-module equipped with a direct sum decomposition
    $$D = \bigoplus_{j=0}^{d-1} D_j$$
    and a square-zero endomorphism (the \defi{differential}) $d^D : D \to D$ of $\bbz / d \bbz$-degree $-1$. In other words, the differential restricts to a map
    $$d^D : D_{j} \to D_{j-1}, \quad \text{with indices taken modulo} \ d.$$
    A \defi{morphism} of $\bbz / d \bbz$-graded differential modules is a map $\phi : D \to D'$ of the underlying $R$-modules with $\bbz / d \bbz$-degree $0$, satisfying
    $$d^{D'} \phi = \phi d^D.$$
    The \defi{mapping cone} of a morphism $\phi : D \to D'$ of differential modules is the differential module $\cone (\phi)$ with
    $$\cone (\phi) := D' \oplus D, \quad \text{and differential} \quad d^{\cone (\phi)} := \begin{pmatrix}
        d^{D'} & -\phi \\
        0 & -d^D 
    \end{pmatrix}.$$
    A morphism $\phi  : D \to D'$ of $\bbz / d \bbz$-graded differential modules is \defi{nullhomotopic} if there exists a map $h : D \to D'$ of the underlying $R$-modules of $\bbz / d \bbz$-degree $1$ satisfying $\phi = d^{D'} h + h d^D = 0$. More precisely, there is a sequence of maps $h_j : D_j \to D'_{j+1}$ satisfying
    $$\phi_j = d^{D'} h_{j} + h_{j-1} d^D \quad \text{for all} \ j=0 , \dots , d-1,$$
    where indices are again taken modulo $d$ in the above. The cycles, boundaries, and homology are all defined in a manner identical to the case of complexes:
    $$Z_j (D) := \ker (d^D : D_j \to D_{j-1}), \quad B_j := \im (d^D : D_{j+1} \to D_j), \quad H_j (D) := \frac{Z_j(D)}{B_j (D)}.$$
    Diagramatically, a $\bbz / d \bbz$-graded differential is represented as
\[\begin{tikzcd}
	{D_{d-1}} & {D_{d-2}} & \cdots & {D_1} & {D_0}
	\arrow["{d^D}", from=1-1, to=1-2]
	\arrow["{d^D}", from=1-2, to=1-3]
	\arrow["{d^D}", from=1-3, to=1-4]
	\arrow["{d^D}", from=1-4, to=1-5]
	\arrow["{d^D}", curve={height=-24pt}, from=1-5, to=1-1]
\end{tikzcd}\]
\end{definition}

\begin{remark}\label{rk:homogeneousDM}
    If the ring $R$ is graded by some abelian group, then there is a notion of a \defi{homogeneous} differential module of degree $a$. The only difference here is that the differential should be a degree $a$ map
    $$d^D : D \to D(a),$$
    and the definition of homology should take account of this twist:
    $$H(D) := \frac{\ker d^D}{\im d^D (-a)}.$$
    Throughout this paper, all of our results are canonical with respect to any kind of auxiliary grading on the ring $R$. For this reason, we will often not mention that our results work in the presence of auxiliary gradings but it is understood that this will be true. Indeed, an equivalent definition of a $\bbz / d \bbz$-graded differential module is the following: equip $R$ with the structure of a $\bbz / d \bbz$-graded ring by assigning all elements of $R$ to have degree $0$. A homogeneous $R$-module is thus a module $M$ equipped with a direct sum decomposition $M = \bigoplus_{i=0}^{d-1} M_i$, and a $\bbz/ d \bbz$-graded differential module $D$ is a homogeneous $\bbz / d \bbz$-graded differential module of degree $-1$.  
    
    Homogeneous differential modules play a prominent role in the work of Brown--Erman \cite{brown2021minimal}, and multigraded differential modules over a standard graded poynomial ring are the main object of study in work of Boocher--DeVries \cite{boocher2010rank}. 
\end{remark}

\begin{remark}  
    Notice that we allow the cases $d= 1$ and $d=0$. When $d=1$, there is no interesting grading and a differential module is simply a module equipped with a squarezero endomorphism; these are the objects originally studied by Cartan--Eilenberg. When $d = 0$ a $\bbz$-graded differential module is identically a complex; although this case may seem trivial, we will see later that some of our results on the structure of projective flag resolutions are interesting even in this setting. When $d = 2$, the category $\dm_{\bbz / 2 \bbz}$ is precisely the category of \defi{linear factorizations} of $0$, typically denoted $\operatorname{LF} (R,0)$ as in \cite{brown2017adams}.

    In practice, one typically only encounters the cases $d = 0, 1$, or $2$ for differential modules, but our methods easily work for arbitrary nonnegative integers $d$. 
\end{remark}

\begin{remark}
    Regardless of the value of $d$, we will always assume in this paper that the differential of an object $D \in \dm_{\bbz / d \bbz} (D)$ squares to $0$. A $\bbz / d \bbz$-graded differential module should not be confused with the notion of an $N$-complex \cite{iyama2017derived}, which is a complex where the $N$th power of the differential is $0$, where $N \geq 2$.
\end{remark}

\begin{remark}
   For $d > 1$, short exact sequences of $\bbz / d \bbz$-graded differential modules induce long exact sequences of homology but instead with periodicity $d$:
\[\begin{tikzcd}
	{H_{d-1} (D)} & {H_{d-1} (E)} & {H_{d-1} (F)} \\
	\vdots & \vdots & \vdots \\
	{H_j(D)} & {H_j(E)} & {H_j(F)} \\
	{H_{j-1}(D)} & {H_{j-1} (E)} & {H_{j-1} (F)} \\
	\vdots & \vdots & \vdots \\
	{H_0 (D)} & {H_0 (E)} & {H_0(F)}
	\arrow[from=1-1, to=1-2]
	\arrow[from=1-2, to=1-3]
	\arrow[from=1-3, to=2-1]
	\arrow[from=2-3, to=3-1]
	\arrow[from=3-1, to=3-2]
	\arrow[from=3-2, to=3-3]
	\arrow[from=3-3, to=4-1]
	\arrow[from=4-1, to=4-2]
	\arrow[from=4-2, to=4-3]
	\arrow[from=4-3, to=5-1]
	\arrow[from=5-3, to=6-1]
	\arrow[from=6-1, to=6-2]
	\arrow[from=6-2, to=6-3]
	\arrow[curve={height=30pt}, from=6-3, to=1-1]
\end{tikzcd}\]
For $d = 1$ a short exact sequence of differential modules induces an exact triangle:
\[\begin{tikzcd}
	{H(D)} && {H(E)} \\
	& {H(F)}
	\arrow[from=1-1, to=1-3]
	\arrow[from=1-3, to=2-2]
	\arrow[from=2-2, to=1-1]
\end{tikzcd}\]
\end{remark}

One of the nice features (and part of our motivation for allowing $\bbz / d \bbz$-gradings) is that the category $\dm_{\bbz / d \bbz} (R)$ has additional structure when $d \neq 1$. This fact is straightforward and well-known to experts, but we state is explicitly for convenience.

\begin{obs}\label{obs:MonoidalStructure}
    For $d \neq 1$, the category $\dm_{\bbz / d \bbz} (R)$ is a symmetric monoidal, self-enriched category\footnote{By self-enriched, we mean that there is corresponding object $\hom_R (D,D')$ between two $\bbz / d \bbz$-graded differential modules that is itself a $\bbz / d \bbz$-graded differential module.}. If $R$ has characteristic $2$, then $\dm_{\bbz / d \bbz} (R)$ is a symmetric monoidal, self-enriched category for all $d \geq 0$. 
\end{obs}

\begin{proof}
    The only subtlety here comes from defining graded pieces of these objects. The tensor product has
    $$(D \otimes_R D)_n := \bigoplus_{i + j \equiv n \mod d} D_i \otimes_R D'_j,$$
    $$\text{with differential} \quad d^{D \otimes_R D'} (a \otimes b) := d^D (a) \otimes b + (-1)^{|a|} a \otimes d^{D'} (b).$$
    In a similar way,
    $$\hom_R^n (D , D') := \prod_{i-j \equiv n \mod d} \hom_R (D_i , D'_j ),$$
    $$\text{with differential} \quad d^{\hom_R (D , D')} (\phi) := d^{D'} \circ \phi - (-1)^{|\phi|} \phi \circ d^{D}.$$
\end{proof}

We will sometimes have a need for the notion of folding and unfolding a differential module. Intuitively, the folding operation should be thought of as an operation that collapses\footnote{In the work \cite{avramov2007class}, the authors even call the Fold functor the ``collapse" functor.} the grading of a $\bbz / d \bbz$-graded differential module, whereas unfolding is something like a ``copy-paste" operation.

\begin{definition}[Folding and Unfolding]
    Let $D \in \dm_{\bbz / d \bbz} (R)$. Then the $\bbz / 2 \bbz$-graded \defi{unfolding} is the $\bbz / 2 \bbz \times \bbz / d \bbz$-graded differential module
\[\begin{tikzcd}
	{\deg 1} && {\deg 0} \\
	{D_{d-1}} && {D_{d-1}} \\
	{D_{d-2}} && {D_{d-2}} \\
	\vdots && \vdots \\
	{D_1} && {D_1} \\
	{D_0} && {D_0}
	\arrow[from=2-1, to=3-3]
	\arrow[from=2-3, to=3-1]
	\arrow[from=3-1, to=4-3]
	\arrow[from=3-3, to=4-1]
	\arrow[from=4-1, to=5-3]
	\arrow[from=4-3, to=5-1]
	\arrow[from=5-1, to=6-3]
	\arrow[from=5-3, to=6-1]
	\arrow[curve={height=-6pt}, from=6-3, to=2-1]
	\arrow[curve={height=6pt}, from=6-1, to=2-3]
\end{tikzcd}\]
Assume $e > d$ are nonnegative integers. The \defi{folding functor} $\fold_{\bbz / e \bbz}^{\bbz / d \bbz} : \dm_{\bbz / e \bbz} (R) \to \dm_{\bbz / d \bbz} (R)$ sends a $\bbz / e \bbz$-graded differential module to the differential module with
$$\fold_{\bbz / e \bbz}^{\bbz / d \bbz} (D)_i := \bigoplus_{j \equiv i \mod d} D_j,$$
where the differential is induced by the original differential $d^D$ of $D$. 
\end{definition}

Notice that with respect to the $\bbz / 2 \bbz \times \bbz / d \bbz$-grading of the $\bbz / 2 \bbz$-grading unfolding of $D$, the induced differential has bidegree $(-1,-1)$. 

\begin{remark}
    In the context of \cite{avramov2007class}, the functor $\unfold$ sends a differential module $D$ to the infinite chain complex
    $$D_\bullet : \quad \cdots \to  D \xra{d^D} D \xra{d^D} D \xra{d^D} \cdots.$$
    For our purposes, we will more frequently want to unfold a differential module to a $2$-periodic object.
\end{remark}



\subsection{The Categories $\flag_{\bbz / d \bbz} (R)$ and $\pflag_{\bbz / d \bbz} (R)$}\label{subsec:theFlagCats}

In this section, we introduce two categories: $\flag_{\bbz / d \bbz} (R)$ and $\pflag_{\bbz / d \bbz}$. The objects of these categories are those differential modules admitting a flag structure, and the only difference is whether we assume that the flag structure is preserved when defining morphisms. 

\begin{definition}[The Categories $\flag_{\bbz / d \bbz} (R)$ and $\pflag_{\bbz / d \bbz} (R)$]\label{def:flagDefs}
    The category $\flag_{\bbz / d \bbz} (R)$ of \defi{$\bbz / d \bbz$-graded flags of $R$-modules} has objects given by all $R$-modules equipped with a $\bbz / d \bbz$-graded filtration
    $$0 = D^{-1}_j \subset D^0_j \subset D^1_j \subset \cdots , \quad j = 0 ,\dots , d-1,$$
    and a square-zero endomorphism $d^D : D \to D$ satisfying
    $$d^D (D^i_j) \subset D^{i-1}_{j-1}, \quad \text{for all} \ i \geq 0, \ j \in \bbz / d \bbz.$$
    In the above containment, we use the convention that the index $j$ is taken modulo $d$; for instance $d^D(D^i_0) \subset D^{i-1}_{d-1}$. A \defi{morphism of flags} is a morphism of $R$-modules $\phi : D \to D'$ satisfying 
    $$d^{D'} \phi = \phi d^D, \quad \text{and} \quad \phi(D^i_j) \subset {D'_j}^i.$$
    Two morphisms $\phi , \psi: D \to D'$ of flags are \defi{homotopic} if there is a map of the underlying $R$-modules $h : D \to D'$ such that
    $$\phi - \psi = d^{D'} h + h d^D, \quad \text{and} \quad h(D^i_j) \subset {D'_{j}}^{i+1}.$$
    Every flag may be viewed as a differential module, in which case notions such as cycles, boundaries, and homology are defined in an identical manner. 
    
    The category $\pflag_{\bbz / d \bbz} (R)$ has the same objects as $\flag_{\bbz / d \bbz} (R)$, but with
    $$\hom_{\pflag_{\bbz / d \bbz} (R)} (D,D') := \hom_{\dm_{\bbz / d \bbz} (R)} (D,D').$$
    In other words, morphisms and homotopies need not be filtration preserving.
\end{definition}

\begin{remark}
    The data of the filtration of an object $D \in \flag_{\bbz / d \bbz} (R)$ is called the \defi{flag structure}. When $d = 1$, there is no ambient $\bbz / d \bbz$-grading. In this case a flag is simply an $R$-module equipped with a filtration and a squarezero endomorphism that lowers the filtration degree; this is the definition of a flag as given in Avramov-Buchweitz--Iyengar \cite{avramov2007class}. Flags were introduced as a natural intermediary between the category of chain complexes and the full category of differential modules, since any chain complex tautologically admits such a flag structure. 
\end{remark}

\begin{remark}
    There is a canonical forgetful functor
    $$\flag_{\bbz / d \bbz} (R) \to \dm_{\bbz / d \bbz} (R)$$
    which simply forgets the flag structure on a given flag. This embedding is faithful, but \emph{not} full. On the other hand, since morphisms in the category $\pflag_{\bbz / d \bbz} (R)$ need not be flag-preserving, the canonical forgetful functor
    $$\pflag_{\bbz / d \bbz} (R) \to \dm_{\bbz / d \bbz} (R)$$
    is by construction fully faithful. 
\end{remark}

\begin{remark}
    We will often use the terminology that a morphism or homotopy is \defi{flag-preserving} if it is a morphism or homotopy in the category $\flag (R)$. Indeed, the only difference between the categories $\flag (R)$ and $\pflag(R)$ is whether or not morphisms preserve the flag structure. 
\end{remark}

\begin{chunk}\label{chunk:flagReformulation}
    Let $D \in \flag_{\bbz / d \bbz} (R\proj)$. Since the underlying $R$-module of $D$ is projective, the defining filtration of $D$ is actually split and there is a direct sum decomposition as $R$-modules
    $$D = \bigoplus_{i \geq 0, \ j \in \bbz / d \bbz} D_{i,j},$$
    where $D_{i,j} := D^i_{i+j} / D^{i-1}_{i+j-1}$ and the index $j$ is taken modulo $d$. 
    
    Thus, the differential $d^D$ of $D$ splits into components
    $$d^D = \sum_{i \geq 0} \delta_i^D, \quad \text{where} \ \delta_i^D|_{D_{j,\ell}} : D_{j,\ell} \to D_{j-i+1,\ell+i},$$
    and by collecting graded pieces of the equality $(d^D)^2 = 0$ we find that there are equalities
    $$(\delta_0^D)^2 = 0, \quad \text{and} \quad  \delta_0^D \delta_i^D + \delta_i^D \delta_0^D = - \sum_{j=1}^{i-1} \delta_j^D \delta_{i-j}^D \quad \text{for} \ i \geq 1.$$
    It follows that for each $j = 0 ,\dots , d-1$ there are induced complexes:
    $$D_{\bullet , j} : \quad \cdots \to D_{i,j} \xra{\delta_0^D} D_{i-1,j} \xra{\delta_0^D} \cdots \xra{\delta_0^D} D_{1,j} \xra{\delta_0^D} D_{0,j}.$$
    The direct sum $\bigoplus_{j \in \bbz / d \bbz} D_{\bullet , j}$ is called the \defi{anchor of the projective flag} $D$. 
    
    Given a morphism $\phi : D \to D'$ of projective flags, the map $\phi$ decomposes as a sum $\phi = \sum_{i \geq 0} \phi_i$ with $\phi_i|_{D^j_\ell} : D^j_\ell \to D^{j-i}_\ell$. Again, collecting graded pieces of the equality $d^{D'} \phi = \phi d^D$ yields the equalities
    $$\sum_{j=0}^i \phi_j \delta_{j-i}^D = \sum_{j=0}^i \delta_{j-i}^{D'} \phi_j.$$
    Finally, decomposing the relation $\phi - \psi = d^{D'} \sigma + \sigma d^{D}$ yields the equalities
    $$\phi_i - \psi_i = \sum_{j=0}^i (\delta^{D'}_j \sigma_{i-j} + \sigma_j \delta^D_{i-j}).$$
\end{chunk}

\begin{remark}
    The terminology \emph{anchor} is chosen intentionally: one of the main themes that we will see later in this paper is that the homological properties of a projective flag cannot stray too far from those of the underlying anchor. This question of the relationship between a flag and its anchor was informally posed in work of Brown--Erman \cite{brown2021minimal}, and later stated explicitly in the work \cite{banks2023differential}. 
\end{remark}

\begin{example}
    Assume $D \in \dm_{\bbz / 2 \bbz} (R\proj)$. Then as in the discussion of \ref{chunk:flagReformulation}, upon setting $D_{i,j} := D^i_{i+j} / D^{i-1}_{i+j-1}$ the projective flag $D$ may be viewed diagramatically as follows:
\[\begin{tikzcd}
	\cdots & {D_{3,0}} & {D_{2,0}} & {D_{1,0}} & {D_{0,0}} \\
	\cdots & {D_{3,1}} & {D_{2,1}} & {D_{1,1}} & {D_{0,1}}
	\arrow[from=1-1, to=1-2]
	\arrow[from=1-2, to=1-3]
	\arrow[from=1-3, to=1-4]
	\arrow[from=1-4, to=1-5]
	\arrow[from=2-4, to=2-5]
	\arrow[from=2-3, to=2-4]
	\arrow[from=2-2, to=2-3]
	\arrow[from=2-1, to=2-2]
	\arrow[from=1-3, to=2-5]
	\arrow[from=2-3, to=1-5]
	\arrow[from=1-2, to=2-4]
	\arrow[curve={height=-18pt}, from=1-2, to=1-5]
	\arrow[curve={height=18pt}, from=2-2, to=2-5]
	\arrow[curve={height=-18pt}, from=1-1, to=1-4]
	\arrow[curve={height=18pt}, from=2-1, to=2-4]
	\arrow[from=2-2, to=1-4]
	\arrow[from=2-1, to=1-3]
	\arrow[from=1-1, to=2-3]
\end{tikzcd}\]

The anchor of $D$ is then given by the direct sum of complexes
\[\begin{tikzcd}
	\cdots & {D_{3,0}} & {D_{2,0}} & {D_{1,0}} & {D_{0,0}} \\
	\cdots & {D_{3,1}} & {D_{2,1}} & {D_{1,1}} & {D_{0,1}}
	\arrow[from=1-1, to=1-2]
	\arrow[from=1-2, to=1-3]
	\arrow[from=1-3, to=1-4]
	\arrow[from=1-4, to=1-5]
	\arrow[from=2-4, to=2-5]
	\arrow[from=2-3, to=2-4]
	\arrow[from=2-2, to=2-3]
	\arrow[from=2-1, to=2-2]
\end{tikzcd}\]
\end{example}

\begin{remark}
    An important aspect of anchors to keep in mind is that the way we have index anchors ensures that the module $D_{i,j}$ has $\bbz / d \bbz$-degree $i+j$ and \emph{not} just degree $j$. 
\end{remark}

The following lemma is a good illustration of how $R$-projective flags have homological behavior that closely resembles that of complexes of projectives in the category of chain complexes. For a proof, see \cite[Proposition 6.4]{felix2012rational}.

\begin{lemma}\label{lem:naiveLiftingLemma}
    Let $\phi : D \to D'$ be a morphism of differential modules. Then there exists a (not necessarily flag-preserving) map $\widetilde{\phi} : F \to F'$ of $R$-projective flags equipped with quasi-isomorphisms $F \xra{\sim} D$ and $F' \xra{\sim} D'$ such that the following diagram commutes up to homotopy:
\[\begin{tikzcd}
	F && {F'} \\
	\\
	D && {D'}
	\arrow["{\widetilde{\phi}}", from=1-1, to=1-3]
	\arrow["\sim"', from=1-1, to=3-1]
	\arrow["\sim", from=1-3, to=3-3]
	\arrow["\phi"', from=3-1, to=3-3]
\end{tikzcd}\]
\end{lemma}

\begin{remark}
     One of the core motivations for much of the work in subsequent sections is to understand when the morphism $\widetilde{\phi} : F \to F'$ of Lemma \ref{lem:naiveLiftingLemma} may be chosen to be flag-preserving.
\end{remark}

The following lemma shows that quasi-isomorphisms of projective flags must in fact be homotopy equivalences. As explained in Remark \ref{rk:regularAndContractible} below, having a flag structure is essential for any such property to hold if the ambient ring is not regular:

\begin{lemma}\label{lem:contractibilityLemma}
    Assume $\phi : D \to D'$ is a quasi-isomorphism of objects $D, D' \in \flag_{\bbz / d \bbz} (R\proj)$. Then $\phi$ is a homotopy equivalence, but not necessarily in a flag-preserving way.
\end{lemma}

\begin{proof}
    This is an immediate consequence of \cite[Proposition 6.4]{felix2012rational}.
\end{proof}

\begin{remark}\label{rk:regularAndContractible}
    Over a regular local ring, any quasi-isomorphism between objects whose ambient $R$-modules are free must be a homotopy equivalence. This follows from the fact that any exact differential module whose ambient $R$-module is free is contractible, a consequence of the fact that there is an exact sequence
    $$0 \to B(D) \to D \to \cdots \to D \to B(D) \to 0.$$
    Choosing this sequence to be of length $> \dim R$ forces $B(D)$ to be a free $R$-module, whence $D$ is contractible. Even for non-regular hypersurfaces, this immediately becomes false, since the $\bbz / 2 \bbz$-folding of a $2$-periodic minimal free resolution is exact, but not contractible (see \cite[Example 1.7]{avramov2007class}).
\end{remark}

It turns out that the homology of a differential module equipped with a flag structure also admits its own kind of flag structure. The following definition makes this structure precise:

\begin{definition}
    Let $D \in \obj \flag_{\bbz / d \bbz} (R)$. Then there is a canonical filtration of the homology $H(D)$ defined via
    $$H^i_j (D) := \im \left( H_j(D^i) \to H_j(D) \right) = \frac{Z(D^i)}{B(D) \cap Z (D^i)}.$$
    This filtration is exhaustive; that is:
    $$H_j (D) = \bigcup_{i \geq 0} H^i_j (D).$$
    We define the associated graded pieces of this filtration via
    $$H_{i,j} (D) := \frac{H^i_{i+j} (D)}{H^{i-1}_{i+j-1} (D)}, \quad i \in \bbz, \ j \in \bbz / d \bbz.$$
\end{definition}

\begin{remark}
    Given a short exact sequence of objects in $\flag_{\bbz / d \bbz} (R)$, there is an induced long exact sequence in homology that more closely mimics the induced long exact sequence in the category of complexes:
\[\begin{tikzcd}
	\vdots & \vdots & \vdots \\
	{H^i_{j+i} (D)} & {H^i_{j+i} (E)} & {H^i_{j+i} (F)} \\
	\vdots & \vdots & \vdots \\
	{H^1_{j+1} (D)} & {H^1_{j+1} (E)} & {H^1_{j+1} (F)} \\
	{H^0_j (D)} & {H^0_j (E)} & {H^0_j (F)} & 0
	\arrow[from=5-1, to=5-2]
	\arrow[from=5-2, to=5-3]
	\arrow[from=5-3, to=5-4]
	\arrow[from=4-3, to=5-1]
	\arrow[from=4-1, to=4-2]
	\arrow[from=4-2, to=4-3]
	\arrow[from=2-1, to=2-2]
	\arrow[from=2-2, to=2-3]
	\arrow[from=2-3, to=3-1]
	\arrow[from=3-3, to=4-1]
	\arrow[from=1-3, to=2-1]
\end{tikzcd}\]
\end{remark}

The following spectral sequence will be a key feature for making precise the assertion that complexes and differential modules are essentially indistinguishable for many K-theoretic purposes.

\begin{lemma}\label{lem:ZZdSS}
    Let $D \in \flag_{\bbz / d \bbz} (R)$. There is a $\bbz \times \bbz / d \bbz$-graded spectral sequence with
    $$E^2_{r,s} = H_r (D_{\bullet, s}) \implies H_{r+s} (D),$$   
\end{lemma}

\begin{proof}
    This is an immediate consequence of the spectral sequence associated to a filtration and can be deduced from the classic work of Cartan--Eilenberg (\cite{cartan2016homological}, see also \cite[Subsection 2.6]{avramov2007class}); the addition of the $\bbz / d \bbz$-grading is well-defined since the differential has degree $-1$. 
\end{proof}

\begin{example}
    Suppose that $D \in \flag_{\bbz / 2 \bbz} (R)$. Then the page maps for the spectral sequence of Lemma \ref{lem:ZZdSS} behave in an identical way to a $\bbz / d \bbz$-graded spectral sequence, but with the caveat that the vertical direction is taken ``modulo $d$". For instance the page maps in the $d=2$ case behave as follows: 
\[\begin{tikzcd}
	{E^1:} & {E^1_{0,0}} & {E^1_{1,0}} & {E^1_{2,0}} & {E^1_{3,1}} & \cdots \\
	& {E^1_{0,1}} & {E^1_{1,1}} & {E^1_{2,1}} & {E^1_{3,1}} & \cdots \\
	{E^2:} & {E^2_{0,0}} & {E^2_{1,0}} & {E^2_{2,0}} & {E^2_{3,1}} & \cdots \\
	& {E^2_{0,1}} & {E^2_{1,1}} & {E^2_{2,1}} & {E^2_{3,1}} & \cdots \\
	{E^3:} & {E^3_{0,0}} & {E^3_{1,0}} & {E^3_{2,0}} & {E^3_{3,1}} & \cdots \\
	& {E^3_{0,1}} & {E^3_{1,1}} & {E^3_{2,1}} & {E^3_{3,1}} & \cdots \\
	\vdots
	\arrow[from=1-3, to=1-2]
	\arrow[from=1-4, to=1-3]
	\arrow[from=1-5, to=1-4]
	\arrow[from=2-4, to=2-3]
	\arrow[from=2-5, to=2-4]
	\arrow[from=2-3, to=2-2]
	\arrow[from=1-6, to=1-5]
	\arrow[from=2-6, to=2-5]
	\arrow[from=3-4, to=4-2]
	\arrow[from=4-4, to=3-2]
	\arrow[from=3-5, to=4-3]
	\arrow[from=4-5, to=3-3]
	\arrow[curve={height=12pt}, from=5-5, to=5-2]
	\arrow[curve={height=12pt}, from=5-6, to=5-3]
	\arrow[curve={height=-12pt}, from=6-5, to=6-2]
	\arrow[curve={height=-12pt}, from=6-6, to=6-3]
\end{tikzcd}\]
\end{example}

One fact that may not seem obvious at first glance is that if the anchor of a projective flag $D$ is a resolution, then in fact the homology of $D$ is precisely the $0$th homology of the anchor. A similar result in the absence of a $\bbz / d \bbz$-grading is proved in \cite{banks2023differential}. 

\begin{prop}
    Suppose that $D \in \flag_{\bbz / d \bbz} (R\proj)$ is a projective flag anchored on a resolution. Then
    $$H_j (D) = H_0 (D_{\bullet , j}) \quad \text{for each} \ j=0, \dots , d-1.$$
\end{prop}

\begin{proof}
    The spectral sequence of Lemma \ref{lem:ZZdSS} reads
    $$E^2_{r,s} = \underbrace{H_r (D_{ \bullet , s})}_{=0 \ \text{for} \ r > 0} \implies H_{r+s} (D).$$
    Thus the sequence degenerates and there is an equality $H_j (D) = H_0 (D_{\bullet, j})$ for each $j = 0 ,\dots , d-1$. 
\end{proof}

\subsection{Twisting the Categories of Free Flags and an Equivalence with $A_\infty$-modules}\label{subsec:flagTaus}

In this subsection, we introduce a ``twisted" version of the categories $\flag_{\bbz / d \bbz} (R)$ and $\pflag_{\bbz/ d \bbz} (R)$. As previously mentioned, the main difference in these two categories comes from a difference in sign convention for the notions of morphisms and homotopies. While this sign convention may seem inconsequential, it turns out that this change in perspective allows us to identify the category of projective flags with a category of $A_\infty$-modules over a particularly simple algebra (see Lemma \ref{lem:flagTauEquivAndAinfty}). We first introduce the relevant notation for sign conventions that we will use:
\begin{definition}
    Given a homogeneous element $z \in M$ of a graded $R$-module, define
    $$\overline{z} := (-1)^{|z|+1} z, \quad \text{where} \ |z| := \deg z.$$
    For a nonhomogeneous element, the notation $\overline{z}$ simply extends by linearity. 

    Every graded $R$-module comes equipped with an automorphism $\tau$ which is defined by $$\tau (z) := - \overline{z}.$$
\end{definition}

\begin{remark}
    One thing to be careful with is that we will often work with objects that have a $\bbn \times \bbz / d \bbz$-grading, but the notation $\overline{\cdot}$ is defined \emph{only} to take account of the $\bbn$-grading. 
\end{remark}

The following proposition is a straightforward verification:
\begin{prop}\label{prop:tauProperties}
    Let $\phi$ and $\psi$ be graded endomorphisms of some $\bbz$-graded module $M$. Define $\tau_* := \hom_R (\tau , -)$ and $\tau^* := \hom_R ( - , \tau)$. Then there are equalities
    $$\tau^* (\phi \circ \psi) = \phi \circ \tau^*(\psi) = - \tau^*(\phi) \circ \overline{\psi},$$
    $$\tau^*( \phi) \circ \tau^* (\psi) = - \phi \circ \overline{\psi}, \quad \text{and}$$ 
    $$\tau_* \phi = - \tau^* (\overline{\phi}).$$
\end{prop}

Before defining the categories $\flag^\tau_{\bbz / d \bbz}$ and $\pflag^\tau_{\bbz / d \bbz} (R)$, it will be useful to define the following variant of the Hom complex (to take into account auxiliary gradings).

\begin{definition}
    Let $C = C_{\bullet, \bullet}$ denote a $\bbz \times \bbz / d \bbz$-graded complex, so that for each $j \in \bbz / d \bbz$ the object
    $$C_{\bullet , j}$$
    is a $\bbz$-graded complex. Given two $\bbz \times \bbz / d \bbz$-graded complexes $C$ and $D$, define the \defi{bigraded Hom complex} $\hom^{\bullet,\bullet} (C , D)$ via
    $$\hom_R^{i,j} (C_{\bullet, \bullet} , D_{\bullet, \bullet}) := \prod_{a \in \bbz , \ b \in \bbz / d \bbz} \hom_R (C_{a,b} , D_{a-i,b-i+j} ),$$
    where the second indices are taken modulo $d$, and the differential is the standard differential of the Hom complex (after ignoring the auxiliary $\bbz / d \bbz$-grading):
    $$d^{\hom_R (C,D)} (\phi) := d^D \circ \phi - (-1)^{i} \phi \circ d^C, \quad \text{where} \ \phi \in \hom^{i,j}_R (C,D).$$
    The \defi{endomorphism complex} $\End_R (C)$ is simply defined via $\End_R (C) := \hom_R (C,C).$ 
\end{definition}

\begin{remark}
    Notice that the additional bigrading coming from the auxiliary $\bbz / d \bbz$-grading is not what one may initially expect it to be. We have chosen this ``totalized" convention for convenience of defining objects in $\flag^\tau_{\bbz / d \bbz} (R)$ and easily seeing the compatibility with anchors of projective flags.
\end{remark}

We can now define the relevant categories of this section:

\begin{definition}[The Categories $\flag^\tau_{\bbz / d \bbz} (R)$ and $\pflag^\tau_{\bbz / d \bbz} (R)$]\label{def:flagTauDef}
    An object  of $\flag^\tau_{\bbz / d \bbz} (R)$ is the data of a $\bbz_{\geq 0} \times \bbz / d \bbz$-graded chain complex $C$ equipped with a sequence of maps $\delta_1^C , \dots , \delta_\ell^C , \dots$ of $\bbz / d \bbz$-degree $-1$ satisfying
    \begin{equation}\label{eqn:pertEqn}
        \delta_i^C \in \End^{i+1,1} (C), \quad \text{and} \quad d^C \delta_i^C -  \overline{\delta_i^C} d^C = -\sum_{j=1}^{i-1} \overline{\delta_j^C} \delta_{i-j}^C \quad \text{for all} \ i \geq 1. 
    \end{equation}
    A morphism $\phi : (C , \delta^C) \to (D , \delta^D)$ is a sequence of elements $\phi_0 , \dots , \phi_\ell, \dots $ that satisfy
    \begin{equation}\label{eqn:morphEqn}
        \phi_i \in \hom^{i,0} (C,D), \quad d^D \phi_i + \overline{\phi_i} d^C = -\sum_{j=1}^{i-1} \left( \overline{\phi_j} \delta_{i-j} + \delta_j \phi_{i-j} \right) \quad \text{for all} \ i \geq 0.
    \end{equation}
    The composition of morphisms $\phi : (C , \delta^C) \to (D , \delta^D)$, $\psi: (D , \delta^D ) \to (E , \delta^E)$ is defined via
    \begin{equation}\label{eqn:compEqn}
        (\psi \circ \phi)_i := \sum_{j=0}^i \psi_j \phi_{j-i} \quad \text{for all} \ i \geq 0.
    \end{equation}
    Two morphisms $\phi , \psi : (C , \delta^C) \to (D , \delta^D)$ are homotopic if there is a sequence of maps $\sigma_0, \dots , \sigma_\ell , \dots$ with
    \begin{equation}\label{eqn:htpyEqn}
        \sigma_i \in \hom^{i-1,-1} (C , D), \quad \phi_i - \psi_i = -\sum_{j=0}^{i} \left( \overline{\delta_{j}^D} \sigma_{i-j} + \overline{\sigma_j} \delta_{i-j}^C \right) \quad \text{for all} \ i \geq 0,
    \end{equation}
    where the above summation uses the convention that $\delta_0^C = d^C$ and $\delta_0^D = d^D$. A nullhomotopy is any morphism homotopic to the $0$ map.
\end{definition}

\begin{remark}
    An alternative perspective on the category $\flag^\tau (R)$ can be given by the theory of Koszul operads \cite{loday2012algebraic}. More precisely, let $D := R[t]/(t^2)$ where $t$ has degree $2$. As described in \cite[Chapter 10.3.7]{loday2012algebraic}, an object of $\flag^\tau (R)$ is equivalently a $D_\infty$-module (that is, a dg module over the cobar complex applied to the quadratic dual coalgebra).
\end{remark}

\begin{remark}\label{rk:reintEqualities}
    One can reinterpret the equalities of Definition \ref{def:flagTauDef} as follows: the equation (\ref{eqn:pertEqn}) can equivalently be rewritten as
    $$d^{\End(C)} (\delta_i) = -\sum_{j=1}^{i-1} \overline{\delta_j^C} \delta_{i-j}^C \quad \text{for all} \ i \geq 1.$$
    Equation (\ref{eqn:morphEqn}) can equivalently be rewritten
    $$d^{\hom(C,D)} (\phi_i) =  -\sum_{j=1}^{i-1} \left( \overline{\phi_j} \delta_{i-j} + \delta_j \phi_{i-j} \right) \quad \text{for all} \ i \geq 0,$$
    and equation (\ref{eqn:compEqn}) may be rewritten as
$$\phi_i - \psi_i = -\sum_{j=1}^{i} \left( \overline{\delta_{j}^D} \sigma_{i-j} + \overline{\sigma_{i-j}} \delta_{j}^C \right) - d^{\hom(C,D)} (\sigma_i) \quad \text{for all} \ i \geq 0.$$
    These reformulations will be particularly useful, since they allow us to rephrase the existence of certain types of morphisms and homotopies as showing that certain elements are boundaries in the Hom complex.
\end{remark}

\begin{lemma}\label{lem:flagTauUnfolding}
    The data of an object $D \in \flag^\tau_{\bbz / d \bbz} (R\proj)$ induces the data of a $\bbz / 2 \bbz \times \bbz / d \bbz$-graded differential module equipped with a projective flag
\[\begin{tikzcd}
	D && D
	\arrow["{\delta^D}"', curve={height=-18pt}, from=1-1, to=1-3]
	\arrow["{-\overline{\delta^D}}", curve={height=-18pt}, from=1-3, to=1-1]
\end{tikzcd}\]
    where the differentials have bidegree $(-1,-1)$. This correspondence is functorial and converts homotopies into $\bbz / 2 \bbz \times \bbz / d \bbz$-graded homotopies.
\end{lemma}

\begin{remark}
    Unpacking the data of a $\bbz / 2 \bbz \times \bbz / d \bbz$-graded projective flag as above, this means that $D$ decomposes as a direct sum
    $$D = \bigoplus_{\substack{i \geq 0, \\ j \in \bbz / d \bbz}} D_{i,j},$$
    where the differential $\delta^D$ restricts to maps
    $$\delta^D|_{D_{i,j}} : D_{i,j} \to D_{i-1,j-1}$$
    for all choices of $i \geq 0$, $j \in \bbz / d \bbz$.
\end{remark}

\begin{proof}
    Given an object $D \in \flag^\tau (R)$, it is evident that one may define the differential $\delta^D := d^D + \delta_1^D + \cdots$; the equation (\ref{eqn:pertEqn}) implies that $- \delta^D \overline{\delta^D} = 0$, in which case the complex 
\[\begin{tikzcd}
	D && D
	\arrow["{\delta^D}"', curve={height=-18pt}, from=1-1, to=1-3]
	\arrow["{-\overline{\delta^D}}", curve={height=-18pt}, from=1-3, to=1-1]
\end{tikzcd}\]
    is well-defined and comes equipped with the claimed flag structure. The functoriality and preservation of homotopy equivalence is ensured by construction.  
\end{proof}

As it turns out, when restricted to the category of projective $R$-modules, the categories $\flag_{\bbz /d \bbz} (R\proj)$ and $\flag^\tau_{\bbz / d \bbz} (R\proj)$ are equivalent (and likewise for $\pflag$). Note that we cannot replace $R\proj$ with $R$ in general, since the filtration on an arbitrary flag may not split.

\begin{lemma}\label{lem:flagAndFlagtauEquiv}
    There are homotopy-preserving isomorphisms of categories $$\Phi : \flag_{\bbz / d \bbz} (R\proj) \to \flag^\tau_{\bbz / d \bbz} (R\proj), \quad \text{and}$$
    $$\Phi : \pflag_{\bbz / d \bbz} (R\proj) \to \pflag^\tau_{\bbz / d \bbz} (R\proj).$$
\end{lemma}

\begin{proof}
    We will construct the functor $\Phi$ on the category $\flag_{\bbz / d \bbz} (R\proj)$ and prove the relevant properties. The functor $\Phi$ will be defined identically on $\pflag_{\bbz / d \bbz} (R\proj)$ with the only difference being that morphisms are not necessarily flag preserving. Throughout this proof, we will tacitly be using the equalities of Proposition \ref{prop:tauProperties}.

    Consider the unfolding $\unfold^{\bbz / 2 \bbz} (D)$ and notice that there is an isomorphism of $\bbz / 2 \bbz$-graded differential modules that preserves the auxiliary $\bbz / d \bbz$-gradings:
\[\begin{tikzcd}
	D && D \\
	\\
	D && D
	\arrow["\tau", from=1-1, to=3-1]
	\arrow["1", from=1-3, to=3-3]
	\arrow["{d^D}", curve={height=-18pt}, from=1-1, to=1-3]
	\arrow["{d^D}"', curve={height=-18pt}, from=1-3, to=1-1]
	\arrow["{\tau^* d^D}"', curve={height=-18pt}, from=3-1, to=3-3]
	\arrow["{-\tau^* \overline{d^D}}", curve={height=-18pt}, from=3-3, to=3-1]
\end{tikzcd}\]
    By Lemma \ref{lem:flagTauUnfolding}, the target of this isomorphism is induced from the data of an object $D^\tau \in \flag^\tau_{\bbz / d \bbz} (R\proj)$, and we define $\Phi (D) := D^\tau$. Concretely, the object $D^\tau$ has underlying complex $(D , \tau^* \delta_0^D)$ and $\delta_i^{D^\tau} := \tau^* \delta_i^D$ for all $i \geq 1$. 

    On morphisms $\phi : D \to D'$, define 
    $$\Phi (\phi) := - \overline{\phi}.$$
    The fact that this is well-defined is a consequence of the following computation: using that $\phi$ is a morphism in $\flag_{\bbz / d \bbz} (R\proj)$ implies that there is an equality
    \begingroup\allowdisplaybreaks
    \begin{align*}
        &\sum_{j=0}^i \phi_j \delta_{j-i}^D = \sum_{j=0}^i \delta_{j-i}^{D'} \phi_j \\
        \text{Apply} \ \tau^* \implies&  \sum_{j=0}^i \phi_j (\tau^* \delta^D_{j-i}) = - \sum_{j=0}^i (\tau^* \delta^{D'}_{j} ) \overline{\phi_{i-j}}\\ 
        \text{Rearrange} \ \implies& (\tau^* d^{D'}) (- \overline{\phi_i}) + \phi_j (\tau^* d^D) = - \sum_{j=1}^{i-1} \left( \phi_j (\tau^* \delta^D_{i-j}) + (\tau^* \delta^{D'}_j) (- \overline{ \phi_{i-j}} ) \right). 
    \end{align*}
    \endgroup
    By definition, the map $\Phi(\phi) := - \overline{\phi}$ is a morphism $D^\tau \to {D'}^\tau$. The fact that $\Phi(\psi \circ \phi) = \Phi(\psi) \circ \Phi(\phi)$ is by construction. 

    Assume now that $\phi : D \to D'$ is a nullhomotopy by some sequence of maps $\sigma_0 , \dots , \sigma_\ell , \dots$. By definition we have the following:
    \begingroup\allowdisplaybreaks
    \begin{align*}
        &\phi_i = \sum_{j=0}^i (\delta_j^{D'} \sigma_{i-j} + \sigma_{i-j} \delta_j^D) \\
        \text{Apply} \ \Phi \implies & - \overline{\phi_i} = \sum_{j=0}^i (\overline{\delta_j^{D'}} \overline{\sigma_{i-j}} + \overline{\sigma_{i-j}} \overline{\delta_j^D} ) \\
        \implies & \Phi (\phi_i) = - \sum_{j=0}^i \left( (\tau^* \overline{\delta^{D'}_j} ) (\tau^* \sigma_{i-j}) + (\tau^* \overline{\sigma_{i-j}}) (\tau^* \delta_j^C) \right).
    \end{align*}
    \endgroup
    By definition, the map $\Phi (\phi)$ is a nullhomotopy in the category $\flag^\tau_{\bbz / d \bbz} (R\proj)$.  
    
    The inverse functor $\Phi^{-1}$ is described as follows: let $(C , \delta^C) \in \flag^\tau_{\bbz / d \bbz}$ and consider the isomorphism of $\bbz / d \bbz$-graded objects
\[\begin{tikzcd}
	C && C \\
	\\
	C && C
	\arrow["{\delta^C}"', curve={height=-18pt}, from=3-1, to=3-3]
	\arrow["{-\overline{\delta^C}}", curve={height=-18pt}, from=3-3, to=3-1]
	\arrow["\tau", from=3-1, to=1-1]
	\arrow["1"', from=3-3, to=1-3]
	\arrow["{\tau^* \delta_C}", curve={height=-18pt}, from=1-1, to=1-3]
	\arrow["{\tau^* \delta_C}"', curve={height=-18pt}, from=1-3, to=1-1]
\end{tikzcd}\]
    Thus the inverse $\Phi^{-1}$ sends an object $(C , \delta^C)$ to the object whose underlying $R$-module is $C$, equipped with the square-zero endomorphism $\tau^* \delta^C$. On morphisms, the inverse $\Phi^{-1}$ is defined identically: $\Phi^{-1} (\phi) := - \overline{\phi}$. 
\end{proof}

\begin{remark}
    Notice that the functor $\Phi^{-1}$ is always well-defined regardless of whether the ambient $R$-module is projective, but the functor $\Phi$ need not be well-defined since the filtration induced by the flag structure is not necessarily split. 
\end{remark}

We recall some definitions replaced to $A_\infty$-algebras and $A_\infty$-modules. For in-depth details on the theory of $A_\infty$-algebras in general, see, for instance, \cite{lefevre2003infini,keller2001introduction,sagave2010dg,cirici2018derived,letz2023transfer}. For our purposes, we will be looking at $A_\infty$-modules over associative algebras, which still have a surprisingly rich structure relative to just modules. 

\begin{definition}[$A_\infty$-Algebras and Modules]
    A (homologically graded, strictly unital, connected) $A_\infty$ algebra $A$ is a nonnegatively-graded $R$-module
    $$A = \bigoplus_{i \geq 0} A_i$$
    equipped with maps $m_i : A^{\otimes i} \to A$ of degree $i-2$ satisfying the \defi{Stasheff identities}:
    $$\sum_{\substack{i = r+s +t, \\ r,t \geq 0, \ s \geq 1}} (-1)^{r + st} m_{r+t+1} (1^{\otimes r} \otimes m_s \otimes 1^{\otimes t}) = 0, \quad \text{for all} \ i \geq 1.$$
    A left $A_\infty$-module $M$ over an $A_\infty$-algebra $A$ is a nonnegatively-graded $R$-module
    $$M = \bigoplus_{i \geq0} M_i$$
    equipped with maps $m_i^M : A^{\otimes i-1} \otimes_R M \to M$ of degree $i-2$ satisfying a similar set of Stasheff identities:
    \begin{equation}\label{eqn:stasheffIds}
        \sum_{\substack{i = r+s +t, \\ r \geq 0, \ s,t \geq 1}} (-1)^{r + st} m_{r+t+1}^M (1^{\otimes r} \otimes m_s \otimes 1^{\otimes t}) + \sum_{\substack{i=r+s \\ r \geq 0, \ s \geq 1}} (-1)^r m_{r+1}^M (1^{\otimes r} \otimes m_s^M ) = 0, \quad \text{for all} \ i \geq 1.
    \end{equation}
    A \defi{morphism of $A_\infty$-modules} $\phi : M \to N$ is a family of maps
    $$\phi_i : A^{\otimes i-1} \otimes_R M \to N$$
    of degree $i-1$ such that for each $i \geq 1$ there is an equality
    \begin{equation}\label{eqn:AinftyMorphism}
        \sum_{\substack{i = r+s+t \\ r \geq 0, \ t,s \geq 1}} (-1)^{r+st} \phi_{r+t+1} (1^{\otimes r} \otimes m_s \otimes 1^{\otimes r}) + \sum_{\substack{i=r+s \\ r \geq 0, \ s \geq 1}} (-1)^r \phi_{r+1} (1^{\otimes r} \otimes m_s^M) = \sum_{\substack{i = u+v, \\ u \geq 0, \ v \geq 1}} m_{1+u}^N (1^{\otimes u} \otimes \phi_v) \quad \text{for all} \ i \geq 1.
    \end{equation}
    Finally, a morphism of $A_\infty$-modules $\phi : M \to N$ is \defi{nullhomotopic} if there is a sequence of maps
    $$\sigma_i : A^{\otimes i-1} \otimes M \to N$$
    of degree $i$ such that for all $i \geq 1$ there is an equality:
    \begin{equation}\label{eqn:AinftyHtpy}
        \phi_i = \sum_{s=1}^i (-1)^{i-s} m_{i-s+1}^N (1^{\otimes i-s} \otimes \sigma_s) + \sum_{\substack{i = r+s+t \\ r \geq 0, \ s,t \geq 1}} (-1)^{r+st} \sigma_{1+r+t} (1^{\otimes r} \otimes m_s \otimes 1^{\otimes t}) + \sum_{\substack{i = r+s, \\ r \geq 0, \ s \geq 1}} (-1)^r \sigma_{r+1} (1^{\otimes r} \otimes m_s^M).
    \end{equation}
    The category of (left) $A_\infty$-modules over $A$ with the above definition of morphisms is denoted $\Mod_{A_\infty} (A)$. If $A$ and $M$ have a $\bbz \times \bbz / d \bbz$-grading, then all of the above data is assumed to have the same degree with respect to the $\bbz / d \bbz$-grading.
\end{definition}

\begin{remark}
    $A_\infty$-modules are sometimes referred to as ``polydules", which is terminology arising from the thesis of Lefevre-Hasegawa \cite{lefevre2003infini}; see also \cite{jensen2009degeneration}. 
\end{remark}

\begin{remark}
    We assume that all our $A_\infty$-algebras, modules, and morphisms are \defi{unital}, implying that the following identities hold:
    $$m_2^M (1 \otimes m) = m,  \quad \text{and for} \ i>2, \  m_i^M (a_1 \otimes \cdots \otimes a_{i-1} \otimes m) = 0 \quad \text{if} \ a_j = 1 \ \text{for some} \ 1 \leq j <i.$$
    A similar set of identities holds for morphisms.
\end{remark}

As previously noted, there is an equivalence of categories $\Mod_{A_\infty} \left( \frac{R[t]}{(t^2)} \right) \cong \flag^\tau_{\bbz / d \bbz} (R)$. This result also helps explain the origin of the seemingly arbitrary change in the sign convention for defining $\flag^\tau_{\bbz / d \bbz} (R)$. 

\begin{lemma}\label{lem:flagTauEquivAndAinfty}
    View $R[t]/(t^2)$ as a $\bbz \times \bbz / d \bbz$-graded $R$-algebra, where $t$ is an indeterminate of bidegree $(2,1)$. Then there is a homotopy-preserving equivalence of categories
    $$\Mod_{A_\infty} \left( \frac{R[t]}{(t^2)} \right) \cong \flag^\tau_{\bbz / d \bbz} (R).$$
\end{lemma}

\begin{proof}
    Notice that the algebra $A := R[t]/(t^2)$ being viewed as an $A_\infty$-algebra with $m_i = 0$ for $i=1$ or $i >2$, and $m_2$ is the standard multiplicative structure on this algebra. 

    Let $M$ be a graded left $A_\infty$-module over $A$ and define the maps
    $$\delta_i^M := m_{i+1}^M (t^{\otimes i} \otimes -) : M[2i] \to M.$$
    This map must have degree $i-1$ and thus by definition induces a degree $0$ map
    $$M[2i] \to M[i-1] \implies \delta_i^D \in \End^{i+1} (M).$$
    Moreover, since $m_i (t^{\otimes i}) = 0$ for all $i \geq 1$ on the algebra $R[t]/(t^2)$, the Stasheff identities (\ref{eqn:stasheffIds}) imply that there is an equality
    $$\sum_{\substack{i=r+s \\ r \geq 0, \ s \geq 1}} (-1)^r m_{r+1}^M (t^{\otimes r} \otimes m_s^M (t^{\otimes s-1} \otimes -) ) = 0 \implies \sum_{\substack{i=r+s \\ r \geq 0, \ s \geq 1}} (-1)^r \delta_r^M \circ \delta_{s-1}^M = 0.$$
    Isolating the terms where $r= 0$ and $s=1$ and reindexing the above sum yields the equality of Equation (\ref{eqn:pertEqn}).

    Let $\phi : M \to N$ be a morphism of $A_\infty$-modules over $A$. Define the maps
    $$\psi_i := \phi_{i+1} (t^{\otimes i} \otimes -) : M[2i] \to N.$$
    These maps have degree $i$ and thus induce degree $0$ maps
    $$M[2i] \to N[i] \implies \psi_i \in \hom^i (M,N).$$
    The identity (\ref{eqn:AinftyMorphism}) implies that there is an equality
    $$\sum_{\substack{i=r+s \\ r \geq 0, \ s \geq 1}} (-1)^r \phi_{r+1} (t^{\otimes r} \otimes m_s^M (t^{\otimes s-1} \otimes -)) = \sum_{\substack{i = u+v, \\ u \geq 0, \ v \geq 1}} m_{1+u}^N (t^{\otimes u} \otimes \phi_v (t^{\otimes v-1} \otimes -))$$
    $$\implies \sum_{\substack{i=r+s \\ r \geq 0, \ s \geq 1}} (-1)^r \psi_r \circ \delta_{s-1}^M = \sum_{\substack{i = u+v, \\ u \geq 0, \ v \geq 1}} \delta_u^N \circ \psi_{v-1}.$$
    Again, isolating terms and reindexing the summations yields the equality (\ref{eqn:morphEqn}).

    Finally, assume that the morphism of $A_\infty$-modules $\phi : M \to N$ is a nullhomotopy via some sequence of degree $i$ maps $\sigma_i : A^{\otimes i-1} \otimes M \to N$. Define $$\psi_i := \phi_{i+1} (t^{\otimes i} \otimes -) : M [2i] \to N, \quad h_i := \sigma_{i+1} (t^{\otimes i} \otimes -) : M [2i] \to N.$$
    The map $h_i$ induces a degree $0$ map
    $$M[2i] \to N[i+1] \implies h_i \in \hom^{i-1} (M , N).$$
    Again, a nearly identical argument shows that the identity (\ref{eqn:AinftyHtpy}) reduces to equation (\ref{eqn:htpyEqn}). 
\end{proof}

\subsection{The Category of Homotopy-Commutative Diagrams}\label{subsec:HarrowCategory}

In this section, we define the category of homotopy-commutative diagrams of differential modules. Again, this category serves as a formal construction that is quite helpful for understanding the compatibility between homotopy equivalent representatives of objects under many of the standard operations on the category of differential modules, such as mapping cones. Most of the results here are straightforward computations and extensions of known facts over the category of complexes, but we provide details for sake of completeness/convenience of reference.

\begin{definition}[Homotopy-Commutative Diagrams]\label{def:HarrowCategory}
    Let $R$ be any commutative ring. Define $\har (\ch (R))$ to be the category with objects
    $$\obj \har(\dm_{\bbz / d \bbz} (R)) := \{ \text{morphisms of $\bbz / d \bbz$-graded differential modules} \ \phi : C \to D \}.$$
    Morphisms of $\har (\dm_{\bbz / d \bbz}(R))$ are defined to be triples $(\psi , \nu , h) : (\phi : C \to D) \to (\phi' : C' \to D')$ where 
    \begin{itemize}
        \item $\psi : C \to C'$ and $\nu : D \to D'$ are morphisms of differential modules, and
        \item $h : C \to D'[1]$ is a homotopy with $\phi' \psi - \nu \phi = d^{D'} h + h d^{C}$.  
    \end{itemize}
    In other words, morphisms are represented by diagrams that commute up to some fixed homotopy $h$:
\[\begin{tikzcd}
	C && {C'} \\
	\\
	D && {D'}
	\arrow["{\phi'}", from=1-3, to=3-3]
	\arrow["\phi"', from=1-1, to=3-1]
	\arrow["\psi", from=1-1, to=1-3]
	\arrow["\nu"', from=3-1, to=3-3]
\end{tikzcd}\]
    Given morphisms $(\psi , \nu , h) : (\phi : C \to D) \to (\phi' : C' \to D')$ and $(\psi', \nu' , h') : (\phi' : C' \to D') \to (\phi'' : C'' \to D'')$ the composition is defined as
    $$(\psi' , \nu' , h') \circ (\psi , \nu , h) := (\psi' \circ \psi , \nu' \circ \nu , \nu' \circ h + h' \circ \psi),$$
    and the identity map $\id_{\phi : C \to D}$ is the morphism
    $$\id_{\phi : C \to D} := (\id_C , \id_D , 0).$$
\end{definition}

The following is a straightforward and satisfying verification:

\begin{obs}
    Given any commutative ring $R$, the category $\har (\dm_{\bbz / d \bbz} (R))$ is well-defined. 
\end{obs}

\begin{remark}
    Notice that $\har (\dm_{\bbz / d \bbz} (R))$ is not equivalently the category of arrows in the homotopy category of $\dm_{\bbz / d \bbz} (R)$, since we keep track of the homotopy that makes the corresponding diagram commute. Indeed, in view of the following lemma, we know that this cannot be true since there is no functorial cone construction on the homotopy category.
\end{remark}

An advantage of using the category $\har (\dm_{\bbz / d \bbz} (R))$ over the standard homotopy category $K ( \dm_{\bbz / d \bbz} (R))$ is that there is a legitimately functorial construction of the mapping cone in $\har (\dm_{\bbz / d \bbz} (R))$:

\begin{lemma}\label{lem:coneOnHarrow}
    There is a well-defined functor
    $$\cone : \har (\dm_{\bbz / d \bbz} (R)) \to \dm_{\bbz / d \bbz} (R)$$
    defined on objects via $\cone (\phi : C \to D) = \cone (\phi)$, the standard mapping cone in the category of differential modules. On morphisms $(\psi , \nu , h) : (\phi : C \to D) \to (\phi' : C' \to D')$ it is defined to be the morphism of differential modules:
    $$\cone (\psi , \nu , h) := \cone (\phi) \xra{\begin{pmatrix}
        \psi & -h \\ 0 & \nu 
    \end{pmatrix}} \cone (\phi').$$
\end{lemma}\

\begin{proof}
    The only nontrivial check to perform is that $\cone$ is compatible with the morphisms in $\har (\dm_{\bbz / d \bbz} (R))$. Let us first check that $\cone (\psi , \nu , h)$ is a well-defined morphism of differential modules; this comes down the the following computation:
    \begingroup\allowdisplaybreaks
    \begin{align*}
        \begin{pmatrix}
            d^{D'} & - \phi' \\ 0 & -d^{C'} 
        \end{pmatrix} \begin{pmatrix}
            \psi & -h \\ 0 & \nu 
        \end{pmatrix} &= \begin{pmatrix}
            d^{D'} \psi & - d^{D'} h - \phi' \nu \\ 0 & -d^{C'} \nu 
        \end{pmatrix}\\
        &= \begin{pmatrix}
            \psi d^{D} & h d^C - \psi \phi \\
            0 & - \nu d^C 
        \end{pmatrix} \\
        &= \begin{pmatrix}
            \psi & -h \\
            0 & \nu 
        \end{pmatrix} \begin{pmatrix}
            d^D & - \phi \\ 0 & -d^C
        \end{pmatrix}.
    \end{align*}
    \endgroup
    Next, we show that $\cone$ commutes with compositions; let $(\psi , \nu , h) : (\phi : C \to D) \to (\phi' : C' \to D')$ and $(\psi', \nu' , h') : (\phi' : C' \to D') \to (\phi'' : C'' \to D'')$. Then:
    \begingroup\allowdisplaybreaks
    \begin{align*}
        \cone (\psi' , \nu' , h') \circ \cone (\psi , \nu , h) &= \begin{pmatrix}
            \psi' & -h' \\
            0 & \nu'
        \end{pmatrix} \begin{pmatrix}
            \psi & -h \\
            0 & \nu 
        \end{pmatrix} \\
        &= \begin{pmatrix}
            \psi' \psi & - \psi' h - h' \nu \\
            0 & \nu' \nu 
        \end{pmatrix} \\
        &= \cone ( (\psi' , \nu' , h') \circ (\psi , \nu , h)).
    \end{align*}
    \endgroup
\end{proof}

The following notation will be employed ubiquitously throughout the paper:

\begin{notation}[Homotopy Equivalences]\label{not:HtpyEquivalence}
    We will use the notation
\[\begin{tikzcd}
	{h, \quad C} && {D, \quad h'}
	\arrow["p", curve={height=-18pt}, from=1-1, to=1-3]
	\arrow["\iota", curve={height=-18pt}, from=1-3, to=1-1]
\end{tikzcd}\]
    to denote a homotopy equivalence between differential modules $C$ and $D$ via the homotopies $h$ and $h'$. More precisely, the maps $p : C \to D$ and $i : D \to C$ satisfy
    $$\iota \circ p = d^C h + h d^C, \quad p \circ \iota = d^D h' + h' d^D.$$
\end{notation}

\begin{lemma}\label{lem:HtpyEquivAndCommutingDiagrams}
    Assume that there are homotopy equivalences
\[\begin{tikzcd}
	h,\quad C && D ,\quad h' & {s,\quad C'} && {D',\quad s'}
	\arrow["p", curve={height=-18pt}, from=1-1, to=1-3]
	\arrow["\iota", curve={height=-18pt}, from=1-3, to=1-1]
	\arrow["{p'}", curve={height=-18pt}, from=1-4, to=1-6]
	\arrow["{\iota'}", curve={height=-18pt}, from=1-6, to=1-4]
\end{tikzcd}\]
and that the diagram
\[\begin{tikzcd}
	C && {C'} \\
	\\
	E && {E'}
	\arrow["\nu"', from=1-1, to=3-1]
	\arrow["{\nu'}", from=1-3, to=3-3]
	\arrow["\phi", from=1-1, to=1-3]
	\arrow["\psi"', from=3-1, to=3-3]
\end{tikzcd}\]
commutes up to some homotopy $\sigma$ (that is, $(\phi , \psi , \sigma) : (\nu : C \to E) \to (\nu' : C' \to E')$). Then the diagram
\[\begin{tikzcd}
	D && {D'} \\
	\\
	E && {E'}
	\arrow["{\nu \iota}"', from=1-1, to=3-1]
	\arrow["{\nu' \iota'}", from=1-3, to=3-3]
	\arrow["{p' \phi \iota}", from=1-1, to=1-3]
	\arrow["\psi"', from=3-1, to=3-3]
\end{tikzcd}\]
commutes up to the homotopy $(\sigma + \nu' h \phi) \iota$ (that is, there is a well-defined morphism $(p' \phi \iota , \psi , (\sigma + \nu' s' \phi ) \iota) : (\nu \iota : D \to E ) \to (\nu' \iota' : D' \to E')$). 
\end{lemma}

\begin{proof}
    This is a straightforward computation:
    \begingroup\allowdisplaybreaks
    \begin{align*}
        \nu' \circ \underbrace{\iota' p'}_{= \id_{C'} + d^{C'} s' + s' d^{C'}} \circ \phi \iota - \psi \nu \iota =& \underbrace{\nu' \phi}_{= \psi \nu + d^{E'}\sigma + \sigma d^{C}}  \circ \iota + \nu' d^{C'} s' \phi \iota + \nu' s' d^{C'} \phi \iota - \psi \nu \iota \\
        =& d^{E'} \sigma \iota + \sigma d^{C} \iota + d^{E'} \nu' s' \phi \iota + \nu' s' \phi \iota d^D \\
        =& d^{E'} (\sigma + \nu' s' \phi) \iota + (\sigma + \nu' s' \phi) \iota d^D.
    \end{align*}
    \endgroup
\end{proof}

\bibliographystyle{amsalpha}
\bibliography{biblio}

\newcommand{\etalchar}[1]{$^{#1}$}
\providecommand{\bysame}{\leavevmode\hbox to3em{\hrulefill}\thinspace}
\providecommand{\MR}{\relax\ifhmode\unskip\space\fi MR }
\providecommand{\MRhref}[2]{%
  \href{http://www.ams.org/mathscinet-getitem?mr=#1}{#2}
}
\providecommand{\href}[2]{#2}
\begin{thebibliography}{BMTW17b}

\bibitem[ABI07]{avramov2007class}
Luchezar~L. Avramov, Ragnar-Olaf Buchweitz, and Srikanth Iyengar, \emph{Class and rank of differential modules}, Invent. Math. \textbf{169} (2007), no.~1, 1--35. \MR{2308849}

\bibitem[Avr98]{avramov1998infinite}
Luchezar~L Avramov, \emph{Infinite free resolutions}, Six lectures on commutative algebra, Springer, 1998, pp.~1--118.

\bibitem[BBG{\etalchar{+}}24]{banks2024multigraded}
Maya Banks, Michael~K Brown, Tara Gomes, Prashanth Sridhar, Eduardo~Torres Davila, and Alexandre Zotine, \emph{The multigraded bgg correspondence in macaulay2}, arXiv preprint arXiv:2402.12293 (2024).

\bibitem[BCLP23]{briggs2023koszul}
Benjamin Briggs, James~C Cameron, Janina~C Letz, and Josh Pollitz, \emph{Koszul homomorphisms and universal resolutions in local algebra}, arXiv preprint arXiv:2310.08400 (2023).

\bibitem[BD10]{boocher2010rank}
Adam Boocher and Justin~W DeVries, \emph{On the rank of multigraded differential modules}, arXiv preprint arXiv:1011.2167 (2010).

\bibitem[BE21]{brown2021tate}
Michael~K. Brown and Daniel Erman, \emph{Tate resolutions on toric varieties}, To Appear, Journal of the European Mathematical Society. arXiv preprint arXiv:2108.03345 (2021).

\bibitem[BE22]{brown2021minimal}
Michael Brown and Daniel Erman, \emph{Minimal free resolutions of differential modules}, Transactions of the American Mathematical Society \textbf{375} (2022), no.~10, 7509--7528.

\bibitem[BGS87]{buchweitz1987cohen}
Ragnar-Olaf Buchweitz, Gert-Martin Greuel, and F-O Schreyer, \emph{Cohen-macaulay modules on hypersurface singularities ii}.

\bibitem[BL91]{barnes1991fixed}
Donald~W Barnes and Larry~A Lambe, \emph{A fixed point approach to homological perturbation theory}, Proceedings of the American Mathematical Society \textbf{112} (1991), no.~3, 881--892.

\bibitem[BMTW17a]{brown2017adams}
Michael Brown, Claudia Miller, Peder Thompson, and Mark Walker, \emph{Adams operations on matrix factorizations}, Algebra \& Number Theory \textbf{11} (2017), no.~9, 2165--2192.

\bibitem[BMTW17b]{brown2017cyclic}
Michael~K Brown, Claudia Miller, Peder Thompson, and Mark~E Walker, \emph{Cyclic adams operations}, Journal of Pure and Applied Algebra \textbf{221} (2017), no.~7, 1589--1613.

\bibitem[Bro67]{brown1967twisted}
Ronnie Brown, \emph{The twisted eilenberg-zilber theorem}, Celebrazioni Arch. Secolo XX, Simp. Top (1967), 34--37.

\bibitem[Bur15]{burke2015higher}
Jesse Burke, \emph{Higher homotopies and golod rings}, arXiv preprint arXiv:1508.03782 (2015).

\bibitem[BV22]{banks2022differential}
Maya Banks and Keller VandeBogert, \emph{Differential modules with complete intersection homology}, To Appear, Journal of Commutative Algebra. arXiv preprint arXiv:2203.15017 (2022).

\bibitem[BV23]{banks2023differential}
\bysame, \emph{Differential modules and deformations of free complexes}, arXiv preprint arXiv:2308.01973 (2023).

\bibitem[Car86]{carlsson1986free}
Gunnar Carlsson, \emph{Free {$(\bbz/2)^k$}-actions and a problem in commutative algebra}, Transformation Groups Pozna{\'n} 1985, Springer, 1986, pp.~79--83.

\bibitem[CE16]{cartan2016homological}
Henry Cartan and Samuel Eilenberg, \emph{Homological algebra (pms-19), volume 19}, Princeton university press, 2016.

\bibitem[Chr98]{christensen1998ideals}
J~Daniel Christensen, \emph{Ideals in triangulated categories: phantoms, ghosts and skeleta}, Advances in Mathematics \textbf{136} (1998), no.~2, 284--339.

\bibitem[Cra04]{crainic2004perturbation}
Marius Crainic, \emph{On the perturbation lemma, and deformations}, arXiv preprint math/0403266 (2004).

\bibitem[CSLW18]{cirici2018derived}
Joana Cirici, Daniela~Egas Santander, Muriel Livernet, and Sarah Whitehouse, \emph{Derived a-infinity algebras and their homotopies}, Topology and its applications \textbf{235} (2018), 214--268.

\bibitem[Dol58]{dold1958homology}
Albrecht Dold, \emph{Homology of symmetric products and other functors of complexes}, Annals of Mathematics \textbf{68} (1958), no.~1, 54--80.

\bibitem[Dut83]{dutta1983frobenius}
Sankar~P Dutta, \emph{Frobenius and multiplicities}, Journal of Algebra \textbf{85} (1983), no.~2, 424--448.

\bibitem[Eis80]{eisenbud1980homological}
David Eisenbud, \emph{Homological algebra on a complete intersection, with an application to group representations}, Transactions of the American Mathematical Society \textbf{260} (1980), no.~1, 35--64.

\bibitem[Erm21]{erman2021matrix}
Daniel Erman, \emph{Matrix factorizations of generic polynomials}, arXiv preprint arXiv:2112.08864 (2021).

\bibitem[FHT12]{felix2012rational}
Yves F{\'e}lix, Stephen Halperin, and J-C Thomas, \emph{Rational homotopy theory}, vol. 205, Springer Science \& Business Media, 2012.

\bibitem[FPTV24]{favero2023factorizations}
David Favero, Sasha Pevzner, Tim Tribone, and Keller VandeBogert, \emph{The {ZZ}d{F}actorizations package in {M}acaulay2}, In Preparation, Github repository may be found \href{https://github.com/Macaulay2/Workshop-2023-Minneapolis/tree/matrix-factorizations/Matrix-Factorizations}{here}. (2024).

\bibitem[GS87]{gillet1987intersection}
Henri Gillet and Christophe Soul{\'e}, \emph{Intersection theory using adams operations}, Inventiones mathematicae \textbf{90} (1987), no.~2, 243--277.

\bibitem[GS97]{grayson1997macaulay}
Daniel Grayson and Michael Stillman, \emph{Macaulay 2--a system for computation in algebraic geometry and commutative algebra}, 1997.

\bibitem[Hal85]{halperin1985rational}
Stephen Halperin, \emph{Rational homotopy and torus actions}, London Math. Soc. Lecture Note Series \textbf{93} (1985), 293--306.

\bibitem[HHW12]{hawwa2012koszul}
F.~T. Hawwa, J.~William Hoffman, and Haohao Wang, \emph{Koszul duality for multigraded algebras}, Eur. J. Pure Appl. Math. \textbf{5} (2012), no.~4, 511--539. \MR{2996984}

\bibitem[Hit22]{hitchcock2022perturbation}
Rohan Hitchcock, \emph{The perturbation lemma for linear factorisations}.

\bibitem[HK91]{huebschmann1991small}
Johannes Huebschmann and Tornike Kadeishvili, \emph{Small models for chain algebras}, Mathematische Zeitschrift \textbf{207} (1991), no.~1, 245--280.

\bibitem[IKM17]{iyama2017derived}
Osamu Iyama, Kiriko Kato, and Jun-ichi Miyachi, \emph{Derived categories of n-complexes}, Journal of the London Mathematical Society \textbf{96} (2017), no.~3, 687--716.

\bibitem[IW18]{iyengar2018examples}
Srikanth~B Iyengar and Mark~E Walker, \emph{Examples of finite free complexes of small rank and small homology}, Acta Mathematica \textbf{221} (2018), no.~1, 143--158.

\bibitem[JMS09]{jensen2009degeneration}
Bernt Jensen, Dag Madsen, and Xiuping Su, \emph{Degeneration of a-infinity modules}, Transactions of the American Mathematical Society \textbf{361} (2009), no.~8, 4125--4142.

\bibitem[Kan55]{kan1955abstract}
Daniel~M Kan, \emph{Abstract homotopy}, Proceedings of the National Academy of Sciences \textbf{41} (1955), no.~12, 1092--1096.

\bibitem[Kel94]{keller1994deriving}
Bernhard Keller, \emph{Deriving dg categories}, Annales scientifiques de l'Ecole normale sup{\'e}rieure, vol.~27, 1994, pp.~63--102.

\bibitem[Kel01]{keller2001introduction}
\bysame, \emph{Introduction to a-infinity algebras and modules}, Homology, Homotopy and Applications \textbf{3} (2001), no.~1, 1--35.

\bibitem[KR00]{KRAdamsOps}
Kazuhiko Kurano and Paul~C. Roberts, \emph{Adams operations, localized {C}hern characters, and the positivity of {D}utta multiplicity in characteristic {$0$}}, Trans. Amer. Math. Soc. \textbf{352} (2000), no.~7, 3103--3116. \MR{1707198}

\bibitem[Let23]{letz2023transfer}
Janina~C Letz, \emph{Transfer of a-infinity structures to projective resolutions}, arXiv preprint arXiv:2312.13696 (2023).

\bibitem[LH03]{lefevre2003infini}
Kenji Lefevre-Hasegawa, \emph{Sur les a-infini cat$\backslash$'egories}, arXiv preprint math/0310337 (2003).

\bibitem[Lur16]{lurie2016higher}
Jacob Lurie, \emph{Higher algebra (2017)}, Available at http://www. math. harvard. edu/\~{} lurie (2016).

\bibitem[LV12]{loday2012algebraic}
Jean-Louis Loday and Bruno Vallette, \emph{Algebraic operads}, vol. 346, Springer Science \& Business Media, 2012.

\bibitem[MS00]{miller2000intersection}
Claudia~M Miller and Anurag~K Singh, \emph{Intersection multiplicities over gorenstein rings}, Mathematische Annalen \textbf{317} (2000), no.~1, 155--171.

\bibitem[OR20]{oblomkov2020soergel}
Alexei Oblomkov and Lev Rozansky, \emph{Soergel bimodules and matrix factorizations}, arXiv preprint arXiv:2010.14546 (2020).

\bibitem[Rob98]{RobertsBook}
Paul~C. Roberts, \emph{Multiplicities and chern classes in local algebra}, no. 133, Cambridge University Press, 1998.

\bibitem[Sag10]{sagave2010dg}
Steffen Sagave, \emph{Dg-algebras and derived a-infinity algebras.}, Journal f{\"u}r die Reine und Angewandte Mathematik \textbf{2010} (2010), no.~639.

\bibitem[Sha69]{shamash1969poincareseries}
Jack Shamash, \emph{The poincare series of a local ring}, Journal of Algebra \textbf{12} (1969), no.~4, 453--470.

\bibitem[Sta10]{Stasheff+2010+203+215}
Jim Stasheff, \emph{A twisted tale of cochains and connections}, Georgian Mathematical Journal \textbf{17} (2010), no.~1, 203--215.

\bibitem[Sta18]{stai2018triangulated}
Torkil Stai, \emph{The triangulated hull of periodic complexes}, Mathematical Research Letters \textbf{25} (2018), no.~1, 199--236.

\bibitem[VW23]{vandebogert2023total}
Keller VandeBogert and Mark~E Walker, \emph{The total rank conjecture in characteristic two}, To Appear, Duke Mathematical Journal. arXiv preprint arXiv:2305.09771 (2023).

\bibitem[Wal17]{walker2017total}
Mark~E Walker, \emph{Total betti numbers of modules of finite projective dimension}, Annals of Mathematics \textbf{186} (2017), no.~2, 641--646.

\end{thebibliography}

\end{document}